\documentclass[11pt]{article}

\usepackage[margin=1in]{geometry}  
\usepackage{graphicx}              
\usepackage{amsmath}               
\usepackage{amsfonts}              
\usepackage{amsthm}                
\usepackage[all,cmtip]{xy}
\usepackage{yhmath}
\usepackage{stmaryrd}
\usepackage[linktocpage,colorlinks=true,raiselinks]{hyperref}
\hypersetup{urlcolor=blue, citecolor=red, linkcolor=blue}

\usepackage[utf8]{inputenc}
\usepackage[english,vietnamese]{babel}

\usepackage{amsmath,amsxtra,amssymb,latexsym, amscd,amsthm,fancyhdr,}
\usepackage{mathrsfs}
\usepackage[all]{xy}
\usepackage{tikz}
\usepackage{tikz-cd}

\usepackage{pdfpages}
\usepackage{fancyhdr}
\newtheorem{theorem}{Theorem}[section]
\newtheorem{lemma}[theorem]{Lemma}
\newtheorem{definition}[theorem]{Definition}
\newtheorem{ex}[theorem]{Example}
\newtheorem{coro}[theorem]{Corollary}
\newtheorem{pro}[theorem]{Proposition}
\newtheorem{re}[theorem]{Remark}

\newcommand{\Lrightarrow}{\hbox to1cm{\rightarrowfill}}
\newcommand{\Ldownarrow}{\bigg\downarrow}

\DeclareMathOperator{\Hom}{Hom}

\DeclareMathOperator{\bX}{\mathbf{X}}
\DeclareMathOperator{\bU}{\mathbf{U}}
\DeclareMathOperator{\bY}{\mathbf{Y}}
\DeclareMathOperator{\mA}{\mathcal{A}}
\DeclareMathOperator{\mC}{\mathcal{C}}
\DeclareMathOperator{\mO}{\mathcal{O}}
\DeclareMathOperator{\mD}{\mathcal{D}}
\DeclareMathOperator{\mI}{\mathcal{I}}
\DeclareMathOperator{\mL}{\mathcal{L}}

\DeclareMathOperator{\mN}{\mathcal{N}}
\DeclareMathOperator{\mR}{\mathcal{R}}
\DeclareMathOperator{\mX}{\mathcal{X}}
\DeclareMathOperator{\mY}{\mathcal{Y}}

\DeclareMathOperator{\bV}{\mathbf{V}}
\DeclareMathOperator{\bW}{\mathbf{W}}
\DeclareMathOperator{\bZ}{\mathbf{Z}}

\DeclareMathOperator{\Aut}{\mathrm{Aut}}

\begin{document}

\selectlanguage{english}

\def \bangle{ \atopwithdelims \langle \rangle}
\title{\textsc{ Weakly holonomic equivariant $\mathcal{D}$-modules on rigid analytic spaces}}
\author{Tobias Schmidt and Thi Minh Phuong Vu}
\date{}
\maketitle
\begin{abstract}
Let $G$ be a $p$-adic Lie group. We develop a dimension theory for coadmissible $G$-equivariant $\mathcal{D}$-modules on smooth rigid analytic spaces. We introduce the category of weakly holonomic $G$-equivariant $\mathcal{D}$-modules, study its duality and its preservation under various operations.
\end{abstract}
\textbf{Key words}: Rigid analytic varieties, $p$-adic Lie groups, equivariant $\mathcal{D}$-modules, weak holonomicity, duality.
\tableofcontents
\makeatletter
\@starttoc{toc}
\makeatother

\section{Introduction}

The category of coherent $\mathcal{D}_{X}$-modules (differential equations) on a smooth complex analytic variety $X$ is omnipresent in many areas of mathematics. Among its many applications to representation theory, we mention the Beilinson-Bernstein theorem, which relates the representations of a given semi-simple complex Lie-algebra to $\mathcal{D}$-modules on its flag variety \cite{BB}. Many interesting representations correspond thereby to so-called holonomic modules and satisfy many finiteness properties. A non-zero coherent $\mathcal{D}_{X}$-module is called holonomic if the dimension of its associated characteristic variety is as small as possible, i.e. equal to $\dim X$. An equivalent definition makes use of the \textit{duality functor} 
$$\mathbb{D}: \,\, D^{-}(\mathcal{D}_X) \longrightarrow D^{+}(\mathcal{D}_X)^{op},\;\; M^.\mapsto R\mathcal{H}om_{\mathcal{D}_X}(M^., \mathcal{D}_X) \otimes_{\mathcal{O}_X}\Omega_X^{\otimes-1}[\dim X]$$ on the derived category $D^{-}(\mathcal{D}_X)$. A coherent $\mathcal{D}_X$-module $M$ is then holonomic if and only if $H^i(\mathbb{D} M)=0$, for all $i \neq 0$. 

\vskip5pt

In the setting of rigid analytic spaces, let $K$ be a discretely valued complete non-archimedean field of mixed characteristic $(0, p)$ with valuation ring $\mathcal{R}$ and uniformiser $\pi$. Let $\bX$ be a smooth rigid analytic variety over $K$. 
In \cite{AWI} Ardakov-Wadsley introduced a certain sheaf of infinite order differential operators $\wideparen{\mathcal{D}}_{\bX}$ on $\bX$ and used it to define the abelian category $\mathcal{C}_{\bX}$ of coadmissible $\wideparen{\mathcal{D}}_{\bX}$-modules. Coadmissibility is a certain finiteness condition replacing coherence in the complex analytic setting. The sheaf $\wideparen{\mathcal{D}}_{\bX}$ is in fact a certain Fr\'echet completion of the sheaf of usual finite order algebraic differential operators $\mathcal{D}_{\bX}$. In the case of rigid analytic flag varieties the sheaf corresponding to $\wideparen{\mathcal{D}}$ on the Zariski-Riemann space  was independently introduced and studied by Huyghe-Patel-Schmidt-Strauch in \cite{CT}, where it is called $\mathscr{D}_{\infty}$.

\vskip5pt 

In the context of $\wideparen{\mathcal{D}}$-modules on smooth rigid analytic varieties, the notion of characteristic variety 
is much more complicated and not yet developed. In order to define a notion of weak holonomicity for $\wideparen{\mathcal{D}}$-modules, the authors  in \cite{AWIII} introduced a dimension theory for coadmissible $\wideparen{\mathcal{D}}$-modules by using the homological grade of a module as its codimension. This is based on the key fact that whenever $\bX$ is affinoid 
with free tangent module $\mathcal{T}(\bX)$, then $\wideparen{\mathcal{D}}(\bX)$ is almost Auslander-Gorenstein (it is a well-behaved inverse limit of Auslander-Gorenstein $K$-algebras). Weak holonomicity is then defined as being of small dimension. The weakly holonomic modules form a full abelian subcategory $\mathcal{C}^{wh}_{\bX}\subset \mathcal{C}_{\bX}$ closed under extensions.  It satisfies Bernstein's inequality and admits an auto-duality. However, as explained in \cite{AWIII}, the category 
$\mathcal{C}^{wh}_{\bX}$ does not yet satisfy all the finiteness and stability properties one would expect from holonomic modules and serves thus only as a first approximation (hence the adjective 'weak'). 

\vskip5pt 
Now let $G$ be a $p$-adic Lie group acting on the smooth rigid space $\bX$. Recently, K. Ardakov introduced in \cite{AW} the category of coadmissible $G$-equivariant $\mathcal{D}_{\bX}$-modules. 
Coadmissible equivariant modules form an abelian category  $\mathcal{C}_{\bX/G}$. In the case  $G=1$, one recovers the former category $\mathcal{C}_{\bX}$.
Since classical equivariant $\mathcal{D}$-modules (e.g. Harish-Chandra sheaves) admit many applications to representation theory, it is expected that the category 
$\mathcal{C}_{\bX/G}$ (for suitable $\bX$ and $G$) will have important applications to 
the representation theory of $G$. A first manifestation of this principle is the equivariant version of the 
rigid analytic Beilinson-Bernstein localization theorem relating $\mathcal{C}_{\bX/G}$ to locally analytic $G$-representations \cite{AW,CT}. Here, $G$ is a reductive group and $\bX$ is its rigid analytic flag variety.

\vskip5pt 
It is therefore a natural question whether the notion of (weakly) holonomicity and its duality can be generalized to the equivariant setting. This is the aim of the present article. 
\vskip5pt

 After recalling some elements from the theory of equivariant $\mathcal{D}$-modules in section \ref{section_two}, we start in section \ref{section_three} the development of Ext functors and equivariant dimension in a local situation. So let $G$ be compact and stabilize a "small" affinoid $\bX$ of dimension $d$. In this case, $\mathcal{C}_{\bX/G}$ is equivalent to the category of coadmissible modules (in the original algebraic sense introduced by Schneider-Teitelbaum \cite{ST2003}) over a certain Fr\'{e}chet-Stein algebra $\wideparen{\mathcal{D}}(\bX,G)$. The latter is a suitable Fréchet completion of the skew-group ring $\mathcal{D}(\bX)\rtimes G$. It is not difficult to see that the Fr\'echet-Stein structure of $\wideparen{\mathcal{D}}(\bX,G)$ has noetherian Banach algebras, which are  Auslander-Gorenstein rings of injective dimension at most $2d$. This allows us to follow \cite{ST2003} and define Ext functors and a homological dimension for coadmissible $\wideparen{\mathcal{D}}(\bX,G)$-modules. 
 
 \vskip5pt 
 In the next section \ref{section_four} we globalize these constructions to general smooth rigid analytic varieties $\bX$ endowed with a continuous $G$-action, using 
 admissible open coverings. We point out that, in contrast to the nonequivariant situation, there is no global sheaf $\wideparen{\mathcal{D}}(-,G)$ playing the role of the coherent sheaf 
$\wideparen{\mathcal{D}}$ in the equivariant setting. The globalization is therefore more subtle than in the case $G=1$ and the technical heart of the paper, see \ref{heart}. To sum up, we construct, for all non negative integers $i\in \mathbb{N}$, Ext-functors $E^i: \mathcal{C}_{\bX/G}\rightarrow \mathcal{C}_{\bX/G}^r,$ from left to right modules, where the sheaf $E^i(\mathcal{M})$ is defined locally on "small" affinoids $\bU$, as 
\[E^i(\mathcal{M})(\bU):=\lim_H Ext^{i}_{\wideparen{\mathcal{D}}(\bU,H)}(\mathcal{M}(\bU), \wideparen{\mathcal{D}}(\bU,H)),\]
where $H$ runs over a suitable set of compact open subgroups of $G$ stabilizing $\bU$.
Similarly, the notion of homological dimension globalizes to a well-defined dimension on coadmissible $G$-equivariant (left or right) $\mathcal{D}_{\bX}$-modules. The usual side-changing functors 
 $\Omega_{\bX}\otimes_{\mathcal{O}_{\bX}}$ and 
 $\mathcal{H}om_{\mathcal{O}_{\bX}}(\Omega_{\bX},-)$ between $\mathcal{C}_{\bX/G}$ and $\mathcal{C}^r_{\bX/G}$
preserve the dimension.
 
\vskip5pt

After having introduced the dimension theory on the category $\mathcal{C}_{\bX/G}$, 
 we define weak holonomicity in section \ref{section_five}: a module $\mathcal{M}\in \mathcal{C}_{\bX/G}$ is weakly holonomic, if $\dim \mathcal{M}\leq \dim \bX$. The weakly holonomic modules form a full Serre subcategory $\mathcal{C}^{wh}_{\bX/G}$ of $\mathcal{C}_{\bX/G}$. As in the classical setting, one needs Bernstein's inequality to go further, cf. \cite[2.3.2]{Hotta}. The following theorem is our first main result, cf. \ref{bernsteinlem_global}.

\vskip5pt 

\textbf{Theorem 1.} {\it Let $\bX$ have good reduction, i.e. a formal model which is smooth. Then Bernstein's inequality holds in $\mathcal{C}_{\bX/G}$: any non-zero $\mathcal{M}\in \mathcal{C}_{\bX/G}$ satisfies ${\rm dim}\; \mathcal{M} \geq {\rm dim} \bX.$ }

\vskip5pt 
Unfortunately, we do not know at present, if Bernstein's inequality holds in $\mathcal{C}_{\bX/G}$ if $\bX$ is a general rigid analytic space with $G$-action. The main problem is that we do now know whether a given smooth affinoid with $G$-action can be viewed as a closed stable subspace of a polydisc with $G$-action, a fact which is obvious when $G=1$. 
\footnote{Note that the corresponding algebraic problem is also obvious: if $G$ is an algebraic group, then any affine $G$-variety admits a closed equivariant immersion into a finite-dimensional $G$-module, e.g. \cite[Prop. 2.2.5]{Brion}.} In any case, the methods developed in this paper allow us to go beyond the case of good reduction and treat, for example, $G$-stable smooth closed subspaces of $G$-spaces with good reduction.  
\vskip5pt

As a next step, we establish the auto-duality on $\mathcal{C}^{wh}_{\bX/G}$, whenever Bernstein's inequality holds. We have Ext functors $\mathcal{E}^i$ on 
$\mathcal{C}_{\bX/G}$ by composing $E^i$ with the side-changing functor $\mathcal{H}om_{\mathcal{O}_{\bX}}(\Omega_{\bX},-)$. Our second main result is the following, cf. \ref{main_2}.
 
 \vskip5pt
\textbf{Theorem 2.} {\it Assume that Bernstein's inequality holds in  $\mathcal{C}_{\bX/G}$. Then  $$\mathbb{D}:= \mathcal{E}^{\dim\bX}\vert_{\mathcal{C}_{\bX/G}^{wh}}$$ is an auto-duality on  $\mathcal{C}_{\bX/G}^{wh}$, i.e. satisfies $\mathbb{D}^2={\rm id}$.}

\vskip5pt

The functor $\mathbb{D}$ is the correct equivariant generalization of the duality functor from \cite{AWIII} and we call it the {\it duality functor} on $\mathcal{C}_{\bX/G}^{wh}$.

\vskip5pt 

In the last section \ref{section_five} we discuss the preservation of weak holonomicity under various operations. This allows us to exhibit large classes of weakly holonomic modules.
We first generalize the extension functor from \cite{AWIII} to the equivariant setting and obtain a functor $E_{\bX/G}$ from $G$-equivariant coherent $\mathcal{D}_{\bX}$-modules to the category $\mathcal{C}_{\bX/G}$. As expected, it takes modules of minimal dimension (in the sense of \cite{MN}) into $\mathcal{C}_{\bX/G}^{wh}$. In the algebraic case, when $\bX=\mathbb{X}^{an}$ for some smooth algebraic $K$-variety $\mathbb{X}$, any holonomic algebraic $\mathcal{D}_{\mathbb{X}}$-module gives rise to a coherent $\mathcal{D}_{\bX}$-module of minimal dimension. 

Back in the case of a general smooth rigid space $\bX$ with $G$-action, we show that all (strongly) $G$-equivariant integrable connections give rise to objects of $\mathcal{C}_{\bX/G}^{wh}$. Of course, the structure sheaf $\mathcal{O}_{\bX}$ is weakly holonomic. We go on and prove a dimension formula for the equivariant pushforward functor, generalizing 
\cite[Thm. 6.1]{AWIII} to the equivariant setting. Making use of the equivariant Kashiwara equivalence for coadmissible modules \cite{AWG}, we arrive at our third main result, cf.
\ref{Kash}.
\vskip5pt
\textbf{Theorem 3.} {\it Let $i: \,\, \bY \rightarrow \bX$ be a smooth Zariski closed subspace which is $G$-stable. Then Kashiwara's equivalence restricts to an equivalence between $\mathcal{C}^{wh}_{\bY/G}$ and the category of weakly holonomic equivariant $\wideparen{\mathcal{D}}_{\bX}$-modules supported on $\bY$.}

\vskip5pt

As a corollary, the module $i_+\mathcal{O}_{\bY}$ is a weakly holonomic $G$-equivariant $\wideparen{\mathcal{D}}_{\bX}$-module for any smooth Zariski closed subspace $i: \,\, \bY \rightarrow \bX$ which is $G$-stable. Here $i_+$ denotes equivariant push-forward along $i$. This yields a large class of examples in $\mathcal{C}_{\bX/G}^{wh}$. We finally show that weak holonomicity is preserved under the geometric induction 
 functor from \cite{AWG}, cf. \ref{indpro}.
 \vskip5pt
 \textbf{Theorem 4.} {\it Let $P$ be a closed co-compact subgroup of $G$. 
 Geometric induction $\text{\rm ind}_P^G: \mathcal{C}_{\bX/P}\rightarrow \mathcal{C}_{\bX/G}$ preserves weak holonomicity.}
 \vskip5pt
 In particular, $ \mathrm{ind}_{G_{\bY}}^G i_+\mathcal{O}_{\bY}$ is a weakly holonomic $G$-equivariant $\wideparen{\mathcal{D}}_{\bX}$-module for any smooth Zariski closed subspace $i: \,\, \bY \rightarrow \bX$ that has a co-compact stabilizer $G_{\bY}$.
 
 \vskip10pt

\textbf{Acknowledgements}. Much of this work is contained in the second author's PhD thesis at Rennes university under the supervision of Tobias Schmidt. She likes to thank him for encouragement and support during the prepartion of this thesis. She also likes to thank the Henri Lebesgue Center Bretagne and Region Bretagne for financial support.

\vskip7pt 

\textbf{Notation}: Throughout this paper, $K$ denotes a complete discrete valuation field of mixed characteristic $(0,p)$ with valuation ring $\mathcal{R}$, a uniformiser $\pi \in \mathcal{R}$ and residue field $k$. The algebraic closure of $K$ is denoted by $\overline{K}$. If $M$ is an $\mathcal{R}$-module, its $\pi$-adic completion is denoted by $\widehat{M}$. All rings, appearing in the article, except for Lie algebras, are supposed to be associative and unital. All modules are left modules, if not further specified. 

\section{Reminder on coadmissible equivariant $\mD$-modules}\label{section_two}
We review some elements of the theory of coadmissible equivariant $\mD$-modules on rigid analytic spaces, as developed in \cite{AW}. This also allows us to set up notation for the sequel.
\subsection{Some algebraic background}
\subsubsection{Lie-Rinehart algebras}
Let $R$ be a commutative ring and $A$ be a commutative $R$-algebra. 
Let $Der_R(A)$ be the space of $R$-linear derivations on $A$. A $R$-Lie algebra $L$  is called \textit{Lie-Rinehart algebra} or a \textit{$(R,A)$-Lie algebra} if it is also an $A$-module equipped with an $A$-linear Lie algebra homomorphism $\rho: L \rightarrow Der_R(A)$ such that $$[x, ay ]=a[x, y] + \rho(x)(a)y $$
for all $x, y \in L$ and $a \in A$, cf. \cite{RG}. Of course, $Der_R(A)$ itself is an \textit{$(R,A)$-Lie algebra} (with $\rho=id$).

Let $(L, \rho)$ be an $(R,A)$-Lie algebra. The \textit{enveloping algebra} of $L$ is the universal associative $R$-algebra $U(L)$ equipped with homomorphisms
\begin{center}
$i_A: A \longrightarrow U(L)$\, and \, $ i_L: L \longrightarrow U(L)$
\end{center}
satisfying the property $i_L(ax)=i_A(a)i_L(x)$ and $[i_L(x), i_A(a)]=i_A(\rho(x)(a))$ for any $ a\in A, x \in L$.  If $A$ is a Noetherian ring and $L$ is a finitely generated $A$-module, then $U(L)$ is a (left and right) Noetherian ring. The Lie algebra $L$ is called \textit{smooth} if it is a  finitely generated projective $A$-module. In this case, the morphisms $i_A$ and $i_L$ are injective and we can identify $A$ and $L$ with its images in $U(L)$.\\

Let $\varphi: A \rightarrow B$ be a morphism of $R$-algebras. We say that \textit{the action of $L$ on ${A}$ lifts to ${B}$} if there exists an ${A}$-linear Lie algebra homomorphism $\sigma: L\longrightarrow Der_R({B})$ such that for every $x \in{L}$, the diagram
\[\begin{tikzcd}
{A} \arrow{r}{\rho(x)} \arrow[swap]{d}{\varphi} & {A} \arrow{d}{\varphi} \\%
{B} \arrow{r}{\sigma(x)}& {B}
\end{tikzcd}\]
is commutative. In this case, the base change $(B\otimes_A L, 1\otimes \sigma)$ is an $(R,B)$-Lie algebra \cite[Lemma $2.2$]{AWI}.\\
\subsubsection{Crossed products}
We assemble some facts on crossed product rings. Our main reference for such rings is \cite{Passman}. For any ring $R$, we let $R^\times$ denote the multiplicative group of invertible elements (units) in $R$.\\
\begin{definition}\label{crossed_prod}
Let $R$ be a ring and $G$ be a group. Then a crossed product $R \ast G$ of $R$ and $G$ is a ring containing $R$ and a set of units $\bar{G}= \lbrace \bar{g}, \, g \in G\rbrace \subset (R\ast G)^{\times}$ which is in bijection with $G$ such that:
\begin{itemize}
\item[(i)] $R\ast G$ is free as a right $R$-module with basis $\bar{G}$  
\item[(ii)] $\bar{g_1}R=R\bar{g_1}$ and $\bar{g_1}\bar{g_2}R=\overline{g_1g_2}R$ for all $g_1, g_2 \in G$.
\end{itemize}
\end{definition}

It follows from (ii) that $R\ast G$ is also freely generated on $\bar{G}$ as a left $R$-module. 
The following flatness lemma for crossed products will be useful for the next sections.

\begin{lemma} Let $\varphi: R \rightarrow A$ be a morphism of rings such that $\varphi$ is left (resp. right) flat
and that it factors through
\[R \longrightarrow R\ast G \rightarrow A.\]
Then the morphism $R\ast G \rightarrow A$ is left (resp. right) flat.
\end{lemma}
\begin{proof} This is \cite[Lemma $2.2$]{Tobias}.\end{proof}

An important example of a crossed product for us is the skew group ring. Let $G$ act on $R$ via a group homomorphism $\sigma: G \rightarrow Aut(R)$. The skew product ring $R\rtimes G$ is, by definition, the free right $R$-module with basis $ G$:
\[R\rtimes G= \lbrace  {g_0}r_0+...+{g_n}r_n, \, r_i \in R, g_i \in G, n \in \mathbb{N}\rbrace.\]
The multiplication on $R \rtimes G$ is defined by:
$$ (g_1r_1)(g_2r_2)=(g_1g_2)((g_2^{-1}.r_1)r_2)$$
for any $ r_1, r_2 \in R$ and $g_1, g_2 \in G$. Here, we let $g.r$ (resp. $r.g$) denote the image of $r$ under $\sigma (g) $ (resp. $\sigma (g^{-1}))$. The ring $R \rtimes G$ naturally contains $R$ as a subring and contains $G$ as a subgroup of $(R\rtimes G)^\times$ and one has the relation 
$grg^{-1}= g.r, \, \,$ for any $g \in G, r \in R.$ 

 \begin{re} We will often use the presentation of $R \rtimes G$ as a free left $R$-module with basis $G$: each element in $R \rtimes G$ has a unique representation  $\sum_{g \in G}r_{g}g,$ where $r_g \in R$ are zero for all but finitely many $g \in G$. 
Under this representation, one can rewrite the multiplication as follows:
\begin{align} \label{skew}
(rg)(r'g')=(r(g.r'))(gg').
\end{align}
 \end{re}
Let a skew group ring $R \rtimes G$ be given. Following \cite[Definition $2.2.1$]{AW} a \textit{trivialisation} of $R \rtimes G$ is a group homomorphism $\beta : G \rightarrow R^\times$ such that
\begin{center}
 $\beta(g)r\beta(g)^{-1}=g.r$ for all $g \in G$ and $r \in R$. 
 \end{center}
 A trivialisation $\beta: G \rightarrow R^\times$ induces a ring isomorphism  
 \begin{align*}
 \tilde{\beta}: R[G] &\stackrel{\simeq}{\longrightarrow} R \rtimes G\\
 r &\longmapsto r\\
 g &\longmapsto \beta(g)^{-1}g
 \end{align*}
with the ordinary group ring $R[G]$, cf. \cite[Lemma $2.2.2$]{AW}. Given a 
a normal subgroup of $G$ and $\beta: N \rightarrow R^\times$ a trivialisation of $R \rtimes N$, one defines the left $R\rtimes G$-module
 $$R \rtimes_N G = R\rtimes_{N}^\beta G:= \frac{R\rtimes G}{(R \rtimes G)(\tilde{\beta}(N)-1)}.$$
When $\beta$ is $G$-equivariant, which means $\beta(gng^{-1})=g.\beta(n)$ for every $n \in N$ and $g \in G$, then $R \rtimes_NG$ is an associative ring containing $R$ as a subring and there is a  group homomorphism $G \rightarrow (R \rtimes_N G)^\times$, cf. \cite[Lemma $2.2.4$]{AW}.

\subsection{Coadmissible equivariant $\mD$-modules on rigid analytic spaces}
Let $\bX$ be a rigid analytic space. If $\bX$ is quasi-compact and quasi-separated (qcqs), then a choice of formal model $\mathcal{X}$ for $\bX$ leads to a certain Hausdorff topology on the automorphism group 
$\Aut(\bX,\mO_{\bX})$ of the $G$-ringed topological space 
$(\bX,\mO_{\bX})$ which is compatible with the group structure and which is independent of the choice of $\mathcal{X}$. A filter base is given by the congruence subgroups in $\Aut(\mX,\mO_{\mX})$, viewed in $\Aut(\bX,\mO_{\bX})$ via the generic fibre functor \cite[3.1.5]{AW}.

\vskip5pt Now let $\bX$ be a general rigid space and $G$ a $p$-adic Lie group together with a group homomorphism $\rho: G\rightarrow \Aut(\bX,\mO_{\bX})$. The action of $G$ on $\bX$ is {\it continuous} in the sense of \cite[Definition 3.1.8]{AW} if the stabilizer $G_{\bU}$ of any qcqs admissible open $\bU$ is open in $G$ and the induced action map 
$G_{\bU}\rightarrow \Aut(\bU,\mO_{\bU})$ is continuous. \\

Now let $\bX$ be a smooth affinoid $K$-variety and $G$ be a compact $p$-adic Lie group which acts continuously on $\bX$. Then any affine formal model $\mA$ of $A:= \mathcal{O}(\bX)$ is contained in a $G$-stable affine formal model \cite[3.2.4]{AW}. Let us fix a $G$-stable affine formal model $\mathcal{A}$ in $A$ in the following. \\
\\
 Let $L:= Der_K(A)$ denote the $(K,A)$-Lie algebra of $K$-derivations endowed with the natural action of $G$.  An $\mathcal{A}$-submodule $\mathcal{L}$ of $L$ is called  \textit{$G$-stable  $\mathcal{A}$-Lie lattice} in L if it is a finitely presented $\mathcal{A}$-module which spans $L$ as a $K$-vector space and is stable under the $G$-action and the Lie bracket on $L$. For such a $G$-stable $\mathcal{A}$-Lie lattice $\mathcal{L}$, we denote by $\widehat{U(\mathcal{L})}$ the $\pi$-adic completion of the envelopping algebra $U(\mathcal{L})$ and write $\widehat{U(\mathcal{L})_K}:= \widehat{U(\mathcal{L})} \otimes_{\mathcal{R}}K$. By functoriality, these rings have natural $G$-actions.\\
\\
Now fix a Lie lattice $\mathcal{L}$ which is \textit{smooth} (i.e. finitely generated projective) as an $\mathcal{A}$-module.Then the unit ball of the $K$-Banach algebra $\widehat{U(\mathcal{L})_K}$ is isomorphic to $\widehat{U(\mathcal{L})}$. Consider the skew product $\widehat{U(\mathcal{L})_K} \rtimes G$. Since $\mathcal{A}$ is $G$-stable, the morphism $\rho: G \rightarrow \Aut(A)$ takes values in the subgroup $\Aut(\mathcal{A}) \subset \Aut(A)$. Write
\begin{align}
G_\mathcal{L}:= \rho^{-1}(\exp(p^\epsilon \mathcal{L})) \subset G. 
\end{align}
Here $\epsilon=1$ if $p=1$; $\epsilon =2$ if $p >2$ and $\rho: G \rightarrow \Aut(\mathcal{A})$. By \cite[Theorem $3.2.12$]{AW}, there is a $G$-equivariant trivialisation 
\[\beta_\mathcal{L}: G_\mathcal{L} \longrightarrow \widehat{U(\mathcal{L})_K}^{\times}\]
of the  $G_\mathcal{L}$-action on $\widehat{U(\mathcal{L})_K}$. Hence the $K$-algebra $\widehat{U(\mathcal{L})_K} \rtimes_H G$ is defined for any open normal subgroup $H$ of $G$ contained in $G_\mathcal{L}$. \\
\\
Recall \cite[Definition 3.2.13]{AW} that a pair $(\mathcal{L},J)$ is called an $\mathcal{A}$-trivialising pair if  $\mathcal{L}$ is a $G$-stable $\mathcal{A}$-Lie lattice in $L$ and $J$ is an open normal subgroup of $G$ contained in the subgroup $G_\mathcal{L}$ of $G$ (which generally depends on $\mathcal{L}$). The set $\mathcal{I}(\mathcal{A},\rho, G)$ of all $\mathcal{A}$-trivialising pairs is a partially ordered set via 
$(\mathcal{L}_1, N_1) \leq (\mathcal{L}_2, N_2)$ \,\,iff\,\, $\mathcal{L}_2 \subset \mathcal{L}_1$ and $N_2 \subset N_1$. In this situation, we can form the \textit{completed skew-group algebra}
\[\wideparen{\mathcal{D}}(\bX,G)= \varprojlim_{(\mathcal{L}, J)}\widehat{U(\mathcal{L})_K}\rtimes_J G, \]
where   $(\mathcal{L},J)$ runs over the set  $\mathcal{I}(\mathcal{A},\rho, G)$ of $\mathcal{A}$-trivialising pairs.\\
\\
In our situation, the pair $(\bX,G)$ is called {\it small} \cite[3.4.4]{AW}, if $L$ has a $G$-stable free $A$-Lie lattice $\mL$ for some $G$-stable affine formal model $\mA$ of $A$. In this case, there is 
a {\it good chain} $(J_\bullet)$ for $\mathcal{L}$, i.e. an open normal subgroup $J_n$ of $G_{\pi^n\mathcal{L}}$ for each non negative integer $n \geq 0$ such that $\bigcap_n J_n=\lbrace 1\rbrace$. Finally, there is a canonical isomorphism of $K$-algebras

\[\wideparen{\mathcal{D}}(\bX,G) \simeq \varprojlim_n \widehat{U(\pi^n\mathcal{L})_K} \rtimes_{J_n}G.\]
For each $n\geq 0$, $  \widehat{U(\pi^n\mathcal{L})_K} \rtimes_{J_n}G $ is a noetherian Banach $K$-algebra and the transition maps in the projective limit are flat ring homomorphisms \cite[3.4.8]{AW}. This shows that $\wideparen{\mathcal{D}}(\bX,G)$ is a two-sided Fr\'{e}chet-Stein algebra in the sense of \cite[6.4]{ST2003} for small $(\bX,G)$.
\\
\begin{re} \label{gamma} Let $\mathcal{D}(\bX)= U(\mathcal{O}(\bX))= {U(\mathcal{L})}\otimes_\mathcal{R}K$ be the ring of global differential operators of finite order on $\bX$. There is a canonical group homomorphism 
$\gamma:  G \rightarrow (\wideparen{\mathcal{D}}(\bX,G))^\times $
and a canonical $K$-algebra homomorphism
$\iota: \mathcal{D}(\bX) \rightarrow \wideparen{\mathcal{D}}(\bX,G).$
These are defined as the inverse limit of the natural maps
$\gamma_n: G \rightarrow \widehat{U(\pi^n\mathcal{L})}_K\rtimes_{J_n} G$ 
and
$ \iota_n: \mathcal{D}(\bX)\cong {U(\pi^n\mathcal{L})}\otimes_\mathcal{R}K \rightarrow \widehat{U(\pi^n\mathcal{L})}_K\rtimes_{J_n} G$
respectively. The maps $\gamma$ and $\iota$ are compatible and define a natural morphism 
\[\iota \rtimes \gamma: \mathcal{D}(\bX)\rtimes G \longrightarrow \wideparen{\mathcal{D}}(\bX,G).\]
\end{re}

For any smooth rigid-analytic space $\bX$ we will write $\mathcal{T}_{\bX}$ or simply $\mathcal{T}$ for the tangent sheaf $Der_K(\mathcal{O}_{\bX})$. We denote $\bX_{w}(\mathcal{T})$ the set of all affinoid subdomains $\bU$ of $\bX$ such that $\mathcal{T}(\bU)$ admits a free $\mathcal{A}$-Lie lattice for some affine formal model $\mathcal{A}$ in $\mathcal{O}(\bU)$. 
The set $\bX_w(\mathcal{T})$ is a basis for the strong Grothendieck topology $\bX_{\rm rig}$ on $\bX$ \cite[Lemma 9.3]{AWI}.\\
\\
Suppose that $(\bX,G)$ is small. Since  $\wideparen{\mathcal{D}}(\bX,G)$ is a Fr\'{e}chet-Stein $K$-algebra, there is the abelian category $\mathcal{C}_{\wideparen{\mathcal{D}}(\bX,G)}$ (resp. $\mathcal{C}_{\wideparen{\mathcal{D}}(\bX,G)}^r$) of coadmissible left (resp. right) $\wideparen{\mathcal{D}}(\bX,G)$-modules from \cite{ST2003}.
It is possible to view coadmissible $\wideparen{\mathcal{D}}(\bX,G)$-modules as $G$-equivariant sheaves on $\bX$, as we now briefly recall \cite[$3.5$]{AW}. Let $M \in \mathcal{C}_{\wideparen{\mathcal{D}}(\bX,G)}$ be a coadmissible left $\wideparen{\mathcal{D}}(\bX,G)$-module, we first define a presheaf on the set $\bX_w(\mathcal{T})$. For each $\bU \in \bX_w(\mathcal{T})$, we set
\begin{center}
$M(\bU,H) := \wideparen{\mathcal{D}}(\bU,H) \wideparen{\otimes}_{\wideparen{\mathcal{D}}(\bX,H)} M$. 
\end{center}
Here $\wideparen{\otimes}$ denotes the completed tensor product of coadmissible modules, which is defined in \cite[Section 7.3]{AWI}. In particular, $M(\bU,H)$ is a coadmissible (left) $\wideparen{\mathcal{D}}(\bU,H)$-module. If $N \leq H$ is another open subgroup of $G$ such that $(\bU,N)$ is small, then there is an isomorphism of $\wideparen{\mD}(\bU,N)$-modules $M(\bU,N)\tilde{\longrightarrow} M(\bU,H)$. Thus we may form the limit when $H$ runs over the set of open subgroups of $G$ such that $(\bU,H)$ is small
\[\mathcal{P}_X(M)(\bU):=\varprojlim_H M(\bU,H).\]
 
Note that the correspondence $\mathcal{P}_{\bX}(M) : \bU\in \bX_w(\mathcal{T}) \longmapsto \mathcal{P}_{\bX}(M)(\bU)$ defines a presheaf on $\bX_w(\mathcal{T})$. The $G$-action on  $\mathcal{P}_X(M)$ is determined as follows. Let $g \in G$ , then there is a continuous isomorphism of $K$- Fr\'{e}chet algebras
$$\wideparen{g}_{\bU,H}:  \wideparen{\mathcal{D}}(\bU,H) \longrightarrow \wideparen{\mathcal{D}}(g\bU, gHg^{-1}).$$
This isomorphism, together with the group homomorphism $\gamma$ in Remark \ref{gamma}, determines the following isomorphism:
\begin{align*}
g^M_{\bU,H}: M(\bU,H) &\longrightarrow M(g\bU, gHg^{-1})\\
a \wideparen{\otimes} m &\longmapsto \wideparen{g}_{\bU,H}(a) \wideparen{\otimes} \gamma(g) m
\end{align*}
which is linear relative to $\wideparen{g}_{\bU,H}$. We then see that there is a $G$-equivariant structure on  $\mathcal{P}_{\bX}(M)$ which is locally determined by the inverse limit of the maps $g^{M}_{\bU,H}$ when $H$ runs over all the open subgroups $H$ of $G$ such that $(\bU,H)$ is small. According to \cite[3.5.8/11]{AW}), the presheaf $\mathcal{P}_{\bX}(M)$ is a $G$-equivariant sheaf of $\mathcal{D}_{\bX}$-modules on $X_w(\mathcal{T})$. It can therefore be extended uniquely to a sheaf on $\bX$, which is denoted by $Loc_{\bX}^{\wideparen{\mD}(\bX,G)}(M)$ or $Loc_{\bX}(M)$ for simplicity (whenever there is no ambiguity).\\
\\
Now we drop the assumption that $(\bX,G)$ is small and let $\bX$ be a smooth rigid analytic variety and $G$ be a $p$-adic Lie group acting continuously on $\bX$. There is the well-known notion of a $G$-equivariant (left or right) $\mathcal{D}_{\bX}$-module, which we do not recall here, e.g. \cite[Definition $2.3.4$]{AW}. Instead, we recall that a $G$-equivariant left $\mathcal{D}_{\bX}$-module $\mathcal{M}$ on $\bX$ is called \textit{locally Fr\'{e}chet} if for each $\bU \in \bX_w(\mathcal{T})$, $\mathcal{M}(\bU)$ is equipped with a $K$-Fr\'{e}chet topology and the maps $g^\mathcal{M}(\bU): \mathcal{M}(\bU) \longrightarrow \mathcal{M}(g\bU)$ are continuous for any $g \in G$, cf. \cite[Definition $3.6.7$]{AW}. There is the obvious notion of a morphism of $G$-equivariant locally Fr\'{e}chet $\mathcal{D}$-modules. The category of $G$-equivariant locally Fr\'{e}chet left $\mathcal{D}_{\bX}$-modules is denoted by $\mathrm{Frech}(G-\mathcal{D})$. 

A $G$-equivariant locally Fr\'{e}chet $\mathcal{D}_{\bX}$-module $\mathcal{M}$ is called \textit{coadmissible} if there exists a $\bX_w(\mathcal{T})$-covering $\mathcal{U}$ of $\bX$ satisfying that for every $\bU \in \mathcal{U}$, there is an open compact subgroup $H$ of $G$ stabilising $\bU$ and a coadmissible $\wideparen{\mathcal{D}}(\bU,H)$-module $M$  such that one has an isomorphism 
\[Loc_{\bU}(M) \simeq \mathcal{M}\mid_{\bU}\]
of $H$-equivariant locally Fr\'{e}chet $\mathcal{D}_{\bU}$-modules. The category of coadmissible $G$-equivariant $\mathcal{D}_{\bX}$-modules is denoted by $\mathcal{C}_{\bX/G}$. This is a full subcategory of ${\rm Frech}(G-\mathcal{D})$. In the case where $(\bX,G)$ is small, the functor 
\[Loc_{\bX} : \mathcal{C}_{\wideparen{\mathcal{D}}(\bX,G)} \longrightarrow \mathcal{C}_{\bX/G}\]
is an equivalence of categories \cite[3.6.11]{AW}.

\vskip5pt
Note that the category $\mathcal{C}_{\bX/G}^r$ of coadmissible $G$-equivariant right $\mathcal{D}$-modules can be defined similarly and, in the case of small $(\bX,G)$, the above equivalence of categories still holds for the category $\mathcal{C}_{\wideparen{\mathcal{D}}(\bX,G)}^r$ via a localization functor ${}^rLoc_{\bX}(-)$ on $\mathcal{C}_{\wideparen{\mathcal{D}}(\bX,G)}^r$, which is defined in complete analogy.  For future reference, let us note that the group $G$ acts (locally) on ${}^rLoc_{\bX}(M)$ as follows: if $g \in G$ and $(\bU,H)$ is small, then $g$ produces an isomorphism of $K$-modules 
\begin{align*}
g^M_{\bU,H}: \,\, M \wideparen{\otimes}_{\wideparen{\mathcal{D}}(\bX,H)} \wideparen{\mathcal{D}}(\bU,H) &\tilde{\longrightarrow} M \wideparen{\otimes}_{\wideparen{\mathcal{D}}(\bX,gHg^{-1})} \wideparen{\mathcal{D}}(g\bU,gHg^{-1})\\
m \wideparen{\otimes}a & \longmapsto  m\gamma(g^{-1} )\wideparen{\otimes} \wideparen{g}(a).
\end{align*}

Next, we recall from \cite{AWI} some important classes of affinoid subdomains of $\bX$. Let $\bU$ be an affinoid subdomain of $\bX$ and let $r^{\bX}_{\bU}: \mathcal{O}(\bX) \rightarrow \mathcal{O}(\bU)$ be the restriction morphism. Fix an affine formal model $\mathcal{A}$ of $\mathcal{O}(\bX)$ and an $\mathcal{A}$-Lie lattice $\mathcal{L}$ in $\mathcal{T}(\bX)$. 

\begin{definition}
\begin{itemize}
\item[(i)] An affine formal model $\mathcal{B}$ in $\mathcal{O}(\bU)$ is called $\mathcal{L}$-stable if $r^{\bX}_{\bU}(\mathcal{A}) \subset \mathcal{B}$ and the action of $\mathcal{L}$ on $\mathcal{A}$ lifts to $\mathcal{B}$. If $\bU$ admits a $\mathcal{L}$-stable affine formal model, then $\bU$ is said to be $\mathcal{L}$-admissible.
\item[(ii)] Suppose that $\bU$ is rational. Then $\bU$ is $\mathcal{L}$-accessible in $n$-steps if $\bU=\bX$ for $n=0$ and for $n >0$,  there is a chain $\bU \subset \bZ \subset \bX$ such that
\begin{itemize}
\item[$\centerdot$] $\bZ \subset \bX$ is $\mathcal{L}$-accessible in $(n-1)$-steps,
\item[$\centerdot$]$\bU=\bZ(f)$ or $\bZ(1/f)$ for some non-zero $f \in \mathcal{O}(\bZ)$,
\item[$\centerdot$] there is a $\mathcal{L}$-stable affine formal model $\mathcal{C} \subset \mathcal{O}(\bZ)$ such that $\mathcal{L}.f \subset \pi \mathcal{C}$.
\end{itemize}
\item[(iii)] An affinoid subdomain (not necessary rational) $\bU$ of $\bX$ is called $\mathcal{L}$-accessible if it is $\mathcal{L}$-admissible and there is a finite covering $\bU=\cup_{i=1}^r \bU_i$, where each $\bU_i$ is a $\mathcal{L}$-accessible rational subdomain of $\bX$.
\end{itemize}
\end{definition}
We also recall that an arbitrary affinoid subdomain $\bU\subset\bX$ becomes, after "rescaling the lattice", $\pi^n\mathcal{L}$-accessible for any sufficiently large $n$, cf. \cite[Prop. 7.6]{AWI}.

Following \cite{AW} we denote by $\bX_{w}(\mathcal{L}, G)$ and $\bX_{ac}(\mathcal{L}, G)$ the sets of $G$-stable affinoid subdomains of $\bX$ which are also $\mathcal{L}$-admissible  and $\mathcal{L}$-accessible  respectively (note that $\bX_{ac}(\mathcal{L}, G) \subset \bX_{w}(\mathcal{L}, G)$). These sets form Grothendieck topologies  on $\bX$ (with respect to inclusion). 

If $N$ is a subgroup of $G$ such that $(\mathcal{L},N)$ is an $\mathcal{A}$-trivialising pair, then following \cite[Section 4]{AW}, we may construct the presheaf $\widehat{\mathcal{U}(\mathcal{L})}_K \rtimes_{N}G$ on $\bX_{w}(\mathcal{L},G)$ as follows. 

\begin{definition}\label{defq}
Let $\bU \in \bX_{w}(\mathcal{L},G)$. Then for any choice of a $G$-stable $\mathcal{L}$-stable affine formal model $\mathcal{B}$ of $\mathcal{O}(\bU)$, we set:
\begin{center}
$(\widehat{\mathcal{U}(\mathcal{L})}_K \rtimes_{N}G)(\bU):= \widehat{U(\mathcal{B}\otimes_{\mathcal{A}}\mathcal{L})}_K \rtimes_{N}G$. 
\end{center}
\end{definition}

By \cite[Cor. $4.3.12$]{AW} this gives rise to a well-defined sheaf on $\bX_{w}(\mathcal{L},G)$.

\begin{pro} \label{flat}
If $\mathcal{L}$ is smooth as an $\mathcal{A}$-module and $\bU \in \bX_{ac}(\mathcal{L},G)$ is $\mathcal{L}$-accessible, then the noetherian ring $(\widehat{\mathcal{U}(\mathcal{L})}_K\rtimes_{N}G)(\bU)$ is flat as a (left and right) $\widehat{U(\mathcal{L})}_K \rtimes_{N}G$-module.
\end{pro}
\begin{proof}\cite[ Theorem 4.3.14]{AW}.\end{proof}
  \vskip5pt
Finally, we briefly recall the side-changing functors between $\mathcal{C}_{\bX/G}$ and $\mathcal{C}_{\bX/G}^r$, cf. \cite{AWG}. Let
\begin{center}
	$\Omega_{\bX}:= \mathcal{H}om_{\mathcal{O}_{\bX}}(\bigwedge^{\dim \bX}_{\mathcal{O}_{\bX}}\mathcal{T}, \mathcal{O}_{\bX})$
\end{center}
be the canonical sheaf on $\bX$. This is an invertible sheaf of $\mathcal{O}_{\bX}$-modules.

\begin{theorem}\label{chanfunc} 
	\begin{itemize}
		\item[$(i)$] The functors $\Omega_{\bX} \otimes_{\mathcal{O}_{\bX}}-$ and $\mathcal{H}om_{\mathcal{O}_{\bX}}(\Omega_{\bX}, -)$ are mutually quasi-inverse equivalences of categories between $\mathcal{C}_{\bX/G}$ and $\mathcal{C}^r_{\bX/G}$.
		\item[$(ii)$] Let $(\bX,G)$ be small. Then $\Omega(\bX) \otimes_{\mathcal{O}(\bX)}-$ and $Hom_{\mathcal{O}(\bX)}(\Omega(\bX),-)$ are quasi-inverse equivalences between the category of coadmissible left resp. right $\wideparen{\mathcal{D}}(\bX,G)$-modules, interchanging the two localization functors $Loc_{\bX}$ and ${}^rLoc_{\bX}$.	
	\end{itemize}
\end{theorem}
\begin{proof} This is \cite[Theorem $4.1.14$, $4.1.15$]{AWG}. \end{proof}
\section{Ext groups and dimension for coadmissible $\wideparen{\mathcal{D}}(\bX,G)$-modules}\label{section_three}
In this section, we develop the local theory of Ext functors and equivariant dimension theory, generalizing the non-equivariant approach from \cite{AWIII}. As in loc.cit., all is based on the notion of an {Auslander-Gorenstein ring}, which we recall briefly. Let $A$ be a ring. The \textit{grade} of an $A$-module $M$ is defined to be 
$$j_A(M):= \min \lbrace i : Ext^i_A(M,A) \neq 0\rbrace$$
and $\infty$ if no such $i$ exists.  The ring $A$ is called \textit{Gorenstein} if it is 
two-sided noetherian and has finite left and right injective dimension $injdim(A).$
A Gorenstein ring $A$ is an \textit{Auslander-Gorenstein ring (or an AG ring)} if it is Gorenstein and satisfies the \textit{Auslander condition}: \begin{itemize}
\item For any finitely generated $A$-module $M$ and any $i \geq 0$, one has $j_A(N) \geq i$ whenever $N$ is a (right) submodule of $Ext^i_A(M,A)$.
\end{itemize}

 In the next section, we want to develop a dimension theory for coadmissible $\wideparen{\mathcal{D}}(\bX,G)$-modules, generalizing the case $G=1$ from 
 \cite[Section 5]{AWIII}. We shall need the following lemma, which is a mild generalization of \cite[Lemma 8.8]{ST2003} to the non-noetherian case. 
\begin{lemma} \label{keylemma1} Let $R_0 \rightarrow R_1$ be a unital homomorphism of unital (possibly non-noetherian) rings. Suppose that there are units $b_0=1,b_1,...,b_m \in (R_1)^\times$ which form a basis of $R_1$ as a left  $R_0$-module and which satisfy:
\begin{itemize}
\item[(i)]$b_iR_0=R_0b_i$ for any $1 \leq i \leq m$.
\item[(ii)] For any $0 \leq i, j \leq m$, there is a natural integer $k$ with $0 \leq k \leq m$ such that $b_ib_j \in b_kR_0$.
\item[(iii)] For any $0 \leq i\leq m$, there is a a natural integer $l$ with $0\leq l\leq m$ such that $b_i^{-1}\in b_lR_0$.

\end{itemize}
Then for any (left or right) $R_1$-module $M$ and (left or right) $R_0$-module $N$, we have an isomorphism of $R_0$-modules
\begin{align*}
Hom_{R_1}(M, R_1 \otimes_{R_0} N) &\tilde{\longrightarrow} Hom_{R_0}(M, N)\\
f &\longmapsto p \circ f.
\end{align*}
Here, $p:  R_1 \rightarrow R_0$ is the projection map  onto the first summand in the decomposition 
$R_1= \bigoplus_{i=0}^{m}b_iR_0=\bigoplus_{i=0}^{m}R_0b_i$
and is $R_0$-linear on both sides. In particular, for any integer $i \geq 0$, this induces an isomorphism of (right or left) $R_0$-modules.
\[Ext_{R_1}^i(M, R_1\otimes_{R_0}N)\simeq Ext^i_{R_0}(M,N).\]

\end{lemma}

\begin{proof} The proof is partly similar to \cite[Lemma $8.8$]{ST2003}.
 Note that $p$ is $R_0$-linear on both sides. Indeed, if $a \in R_0$ and $\sum_{i=0}^m a_ib_i \in R_1$, one has
\begin{itemize}
\item[.]$p(a.\sum_{i=0}^{m}a_ib_i)=p(\sum_{i=0}^{m}aa_ib_i)=aa_0=a.p(\sum_{i=0}^{m}a_ib_i)$
\item[.]$p((\sum_{i=0}^{m}a_ib_i).a)=p(\sum_{i=0}^m a_ib_ia)=p(\sum_{i=0}^ma_ia_i^{'}b_i)= a_0a_0^{'}=a_0a=p(\sum_{i=0}^ma_ib_i).a$,

\end{itemize}
here $a'_i \in R_0$ such that $a=a'_0 $ and $b_ia=a_i^{'}b_i\,\, \forall i \geq 1$ , since $b_iR_0=R_0b_i$ from $(i)$. Thus the morphism:
\begin{align*}
\tilde{p}: R_1 \otimes_{R_0}N & \longrightarrow   R_0\otimes_{R_0}N \tilde{\longmapsto} N\\
b\otimes n &\longmapsto  p(b)\otimes n \longmapsto p(b)n
\end{align*}
is $R_0$-linear. Now by using a free resolution $P^.$ of the $R_1$-module $M$, which is also a free resolution of $M$ as a $R_0$-module, we see that the map $\tilde{p}$ induces a map
\[Ext_{R_1}^i(M, R_1 \otimes_{R_0}N)=h^i(Hom_{R_1}(P^., R_1\otimes_{R_0}N))\longrightarrow h^i(Hom_{R_0}(P^., N))=Ext_{R_0}^i(M, N).\] 
Therefore, it suffices to show that for any $N \in Mod(R_0)$ and $M \in Mod(R_1)$, we have an isomorphism
\[Hom_{R_1}(M, R_1 \otimes_{ R_0 }N ) \tilde{\longrightarrow} Hom_{R_0}(M, N).\]

Take a presentation of $M$ by free $R_1$-modules:
\[R_1^{I}\longrightarrow R_1^{J}\longrightarrow M \longrightarrow 0.\]
Since $Hom_{R_1}(-,N)$ is left exact, we obtain the following commutative diagram:
\[
  \setlength{\arraycolsep}{2pt}
  \begin{array}{*{9}c}
    0 &\Lrightarrow & Hom_{R_1}(M, R_1\otimes_{R_0}N) & \Lrightarrow & Hom_{R_1}(R_1^{J},R_1\otimes_{R_0}N) & \Lrightarrow & Hom_{R_1}(R_1^{I},R_1\otimes_{R_0}N) \\
    & & \Ldownarrow & & \Ldownarrow & & \Ldownarrow & & \\
     0 &\Lrightarrow & Hom_{R_0}(M, N) & \Lrightarrow & Hom_{R_0}(R_1^{J},N) & \Lrightarrow & Hom_{R_0}(R_1^{I},N)
  \end{array}
\]
Hence it is enough to consider the case $M=R_1$ and to prove that
\begin{align*}
\Phi: Hom_{R_1}(R_1, R_1 \otimes_{ R_0} N ) &\tilde{\longrightarrow} Hom_{R_0}(R_1, N)\\
\psi & \longmapsto \tilde{p} \circ \psi .
\end{align*}
This is well-defined since $\tilde{p}$ is $R_0$-linear.

\begin{itemize}
\item[$(1)$] $\Phi$ is surjective. Indeed, if $\phi : R_1 \longrightarrow N$ be an $R_0$-linear map, one defines:
\begin{align*}
\psi: R_1 & \longrightarrow R_1 \otimes_{R_0}N\\
b &\longmapsto \sum_{i=0}^{m}b_i \otimes\phi(b_i^{-1}b).
\end{align*}
Then
\begin{itemize}
\item[$\centerdot$] $\tilde{p} \circ \psi (b)=\tilde{p}(\sum_{i=0}^{m}b_i \otimes \phi(b_i^{-1}b))= \sum_{i=0}^{m}p(b_i)\phi(b_i^{-1}b)= \phi(b)$, \,\, since $p(b_i)=0$ for $i\neq 0$.
\item[$\centerdot$] $\psi$ is $R_1$-linear. Indeed, 
 if $b= \sum_{i=0}^m a_ib_i$ with $a_i \in R_0$ and $b' \in R_1$, one can compute:
\begin{align*}
\psi(bb')& =\sum_ib_i\otimes \phi(b_i^{-1}bb')=\sum_j\sum_ib_i \otimes \phi(b_i^{-1}a_jb_jb')\\
&= \sum_j\sum_i b_i \otimes \phi(a_j^{'}b_i^{-1}b_jb')=\sum_j\sum_i b_ia_j^{'}\otimes \phi(b_i^{-1}b_jb')\\
&=\sum_j\sum_ia_jb_i\otimes \phi(b_i^{-1}b_jb')= \sum_j\left(\sum_i a_jb_jb_j^{-1}b_i \otimes \phi(b_i^{-1}b_jb')\right)\\
&= \sum_j a_jb_j\psi(b')= b \psi(b').
\end{align*}
Here, thanks to $(ii)$ and $(iii)$, we have $\psi(b')= \sum_ib_i\otimes \phi(b_i^{-1}b')=\sum_i b_jb_i\otimes \phi(b_i^{-1}b_jb')$. Therefore $\psi$ is $R_1$-linear. This implies  $\psi \in Hom_{R_1}(R_1, R_1 \otimes_{R_0}N)$ and $\Phi$ is surjective.
\item[$(2)$] $\Phi$ is injective. Indeed, let us first prove that if $\psi: R_1 \longrightarrow R_1 \otimes_{R_0}N$ is an $R_1$-linear map, then 
\[\psi(b)=\sum_{i=0}^m b_i \otimes (\tilde{p}\circ \psi)(b_i^{-1}b).\]
Indeed, suppose that $\psi(b)= \sum_{i}b_i\otimes n_i$, with $n_i \in N$ for all $i$. ( recall that $R_1\otimes_{R_0}N \simeq \bigoplus_{i}b_iR_0\otimes_{R_0}N \simeq \bigoplus_i b_i\otimes N)$, then
\[\psi(b_i^{-1}b)=b_i^{-1}\psi(b)=\sum_j b_i^{-1}b_j \otimes n_j.\]
Thus, 
\[\sum_{i=0}^m b_i \otimes \tilde{p}\circ \psi(b_i^{-1}b)= \sum_{i=0}^m b_i \otimes \sum_{j=0}^mp(b_i^{-1}b_j)n_j= \sum_{i=0}^mb_i \otimes n_i= \psi(b).\]
Consequently, if $\Phi(\psi)= 0 \Longleftrightarrow \tilde{p}\circ \psi =0 \rightarrow \psi(b)=0$ for all $b$. this implies $\Phi$ is injective.
\end{itemize}
\end{itemize}
\end{proof}

\begin{pro} \label{keypro} Let $R_0, R_1$ be two rings which satisfy the assumptions in the above lemma. If a (left or right) $R_1$-module $N$ is injective, then $N$ is also injective as $R_0$-module. Moreover
\begin{itemize}
\item[(i)] $injdim(R_0)=injdim(R_1)$,
\item[(ii)] $Ext^i_{R_1}(N,R_1) \simeq Ext^i_{R_0}(N,R_0)$ and $j_{R_1}(N)=j_{R_0}(N),$
\item[(iii)]If $R_0, R_1$ are noetherian and if $R_0$ is Auslander-Gorenstein, then so is $R_1$.
\end{itemize}
\end{pro}
\begin{proof}
Suppose that $N$ is an injective $R_1$-module. By assumption, $R_1$ is free over $R_0$ on both sides, so it is  flat as a left and right $R_0$-module. Moreover,
\[Hom_{R_0}(M,N)\simeq Hom_{R_1}(R_1\otimes_{R_0}M, N)\]
for any $M \in Mod(R_0)$. By consequence, $N$ is also injective as an $R_0$-module. \\
Now  (ii) is a direct consequence of Lemma \ref{keylemma1} while (iii) can be proved by using $(i)$ and $(ii)$, it remains to prove $(i)$.\\
 If $0 \rightarrow R_0 \rightarrow I^{\cdot} $ is an injective resolution of $R_0$ , then it follows from Lemma \ref{keylemma1} that if $M$ is an $R_1$-module, then $Hom_{R_1}(M, R_1 \otimes_{R_0}I^k) \simeq Hom_{R_0}(M, I^k)$ for any component $I^k$ of the complex $I^.$. Thus $R_1 \otimes_{R_0}I^k$ is an injective $R_1$-module for every $k$. This proves that $0 \rightarrow R_1 \rightarrow R_1 \otimes_{R_0}I^{\cdot} $ is an injective resolution of $R_1$ by $R_1$-modules. Therefore 
\[injdim(R_1) \leq injdim(R_0).\]
It remains to prove that $injdim(R_0) \leq injdim(R_1)$. Suppose that $injdim(R_1)=n < \infty$, so we need to prove that $injdim(R_0) \leq n$. This is equivalent to \[Ext^{n+1}_{R_0}(N,R_0)=0 \,\,\, \text{for any} \,\,N  \in Mod(R_0).\]
Notice that $$Ext_{R_0}^{n+1}(N,R_0)\otimes_{R_0}R_1 \simeq Ext^{n+1}_{R_1}(R_1 \otimes_{R_0}N , R_1).$$ Since $n = injdim(R_1)$, one has $Ext^{n+1}_{R_1}(R_1 \otimes_{R_0}N , R_1)=0$ implying that $Ext_{R_0}^{n+1}(N,R_0)\otimes_{R_0}R_1 =0$. On the other hand, $R_1$ is a free $R_0$-module on both sides, thus $R_1$ is faithfully flat over $R_0$ on both sides. As a result, $Ext^{n+1}_{R_0}(N, R_0)=0$ which proves that $injdim(R_0) \leq n= injdim(R_1)$.

\end{proof}
 Now, as an application, let us consider the following example which will be important for the next section. Suppose that $\bX=Sp(A)$ is a smooth affinoid $K$-variety  for a $K$-affinoid algebra $A$ and $G$ is a compact p-adic Lie group which acts continuously on $\bX$ such that $(\bX,G)$ is small. We assume the following extra conditions:
 \begin{itemize}
 \item[$\ast$]$H$ is an open \textit{normal} subgroup of $G$,
 \item[$\ast$]$\mathcal{A}$ is a $G$-stable affine formal model in $A$,
 \item[$\ast$]$(\mathcal{L},J)$ is an $\mathcal{A}$-trivialising pair such that $J \leq H$.
 \end{itemize}

 Then Lemma \ref{keylemma1} and Proposition \ref{keypro} can partially be  applied to the case where $R_1= \widehat{U(\mathcal{L})}_K \rtimes_JG$ and $R_0=\widehat{U(\mathcal{L})}_K\rtimes_J H$ as follows: 

\begin{lemma} \label{keylemma2} The natural morphism of rings 
$\widehat{U(\mathcal{L})}_K\rtimes_JH \longrightarrow \widehat{U(\mathcal{L})}_K \rtimes_JG$ satisfies the hypothesis of Lemma \ref{keylemma1}. In particular, there is a two-sided $\widehat{U(\mathcal{L})}_K \rtimes_J H$ linear projection map

\begin{equation} \label{projection1}
p^{\bX}_{G,H,J}: \widehat{U(\mathcal{L})}_K\rtimes_JG \longrightarrow \widehat{U(\mathcal{L})}_K\rtimes_JH.
\end{equation}
\end{lemma}
\begin{proof}
Following \cite[Lemma $2.2.6$]{AW}, the ring $R_1:=\widehat{U(\mathcal{L})}_K\rtimes_{J}G$ is isomorphic to  $(\widehat{U(\mathcal{L})}_K \rtimes_JH) \rtimes_HG$ and the latter is isomorphic to the crossed product $(\widehat{U(\mathcal{L})}_K \rtimes_J H)\ast G/H$ (\cite[Lemma $2.2.4$]{AW}). If we denote by $S= \lbrace 1=g_1,g_2,...,g_m \rbrace$ the representatives of the right cosets of $H$ in $G$, then $R_1$ is freely generated over the subring $R_0:=\widehat{U( \mathcal{L})}_K \rtimes_JH$ by the units $\bar{S}=\lbrace \bar{g_1},..., \bar{g}_m\rbrace$, equal to the image of $S$ in $\widehat{U(\mathcal{L})}_K \rtimes_J G$  \cite[Lemma $5.9(i)$ ]{MR}. Now the properties (i) and (ii) from \ref{keylemma1} follow directly from the properties of a crossed product, as recalled in Def. \ref{crossed_prod}. For property (iii), given $i$ there is $l$ such that $g_i^{-1}\in g_lH$, i.e. there is $h\in H$ such that $1\in g_i g_lh$. This means $1\in \bar{g}_i  \bar{g}_lR_0$, whence $\bar{g}_i^{-1} \in \bar{g_l}R_0$, as claimed.
Finally, there is the projection map as claimed. \end{proof}
\begin{re} \label{keyre} The injection $\widehat{U(\mathcal{L})}_K \rightarrow \widehat{U(\mathcal{L})}_K \rtimes_JG$ also satisfies the hypothesis of Lemma \ref{keylemma1}. This can be shown along the lines of the preceding proof, using $\widehat{U(\mathcal{L})}_K \rtimes_JG \cong \widehat{U(\mathcal{L})}_K \ast G/J.$
\end{re}

\begin{pro} \label{keypro1} Suppose that $(\bX,G)$ is small and $H$ is an open normal subgroup of $G$. The ring $\wideparen{\mathcal{D}}(\bX,G)$ is freely generated over $\wideparen{\mathcal{D}}(\bX,H)$ with a basis satisfying the hypothesis of Lemma \ref{keylemma1}. 
\end{pro}
\begin{proof}
By taking the inverse limit of the morphisms $p_{G,H,J}^{\bX}$ in Lemma \ref{keylemma2} when $(\mathcal{L}, J)$ runs over the set of all $\mathcal{A}$-trivialising pairs , we see that $\wideparen{\mathcal{D}}(\bX,G)$ is freely generated as a $\wideparen{\mathcal{D}}(\bX,H)$-module by the image $\tilde{S}= \lbrace \tilde{g_1},\tilde{g_2},..,\tilde{g_m}\rbrace $ of $S$ in $\wideparen{\mathcal{D}}(\bX,G)$ which defines a two-sided $\wideparen{\mathcal{D}}(\bX,H)$-linear map 
\begin{align}\label{projection2}
 p_{G,H}^{\bX}: \wideparen{\mathcal{D}}(\bX,G) &\longrightarrow \wideparen{\mathcal{D}}(\bX,H)\\
 \sum_{i=1}^m a_i\tilde{g}_i &\longmapsto a_0.\nonumber
\end{align}
\end{proof}
\begin{coro} \label{coro3}
\begin{itemize}
 \item[(i)] The maps $p^{\bX}_{G,H}$ and $p^{\bX}_{G,H,J}$ fit into a commutative diagram
\[\begin{tikzcd}
\wideparen{\mathcal{D}}(\bX,G) \arrow{r}{p_{G,H}^{\bX}} \arrow[swap]{d}{q_{G,J}} & \wideparen{\mathcal{D}}(\bX,H)\arrow{d}{q_{H,J}} \\
\widehat{U(\mathcal{L})}_K\rtimes_JG \arrow{r}{p_{G,H,J}^{\bX}} & \widehat{U(\mathcal{L})}_K\rtimes_JH,\\
\end{tikzcd}\]
where $q_{G,J}: \wideparen{\mathcal{D}}(\bX,G)\longrightarrow \widehat{U(\mathcal{L})}_K\rtimes_JG$ and $q_{H,J}: \wideparen{\mathcal{D}}(\bX,H) \longrightarrow \widehat{U(\mathcal{L})}_K\rtimes_JH$ denote the canonical maps induced from the definition of $\wideparen{\mathcal{D}}(\bX,G)$ and $\wideparen{\mathcal{D}}(\bX,H)$ respectively.
\item[(ii)] If $\bU \in \bX_{w}(\mathcal{T})$ is such that $(\bU,G)$ is small, then the diagram
\[\begin{tikzcd}
\wideparen{\mathcal{D}}(\bX,G) \arrow{r}{p_{G,H}^{\bX}} \arrow[swap]{d}{r^{\bU}_{G}} & \wideparen{\mathcal{D}}(\bX,H)\arrow{d}{r^{\bU}_{H}} \\
\wideparen{\mathcal{D}}(\bU,G) \arrow{r}{p_{G,H}^{\bU}} & \wideparen{\mathcal{D}}(\bU,H) \\
\end{tikzcd}\]
is commutative.

\end{itemize}
\end{coro}

\begin{proof}
The statement $(i)$ is evident from definition. To show $(ii)$, let us fix a $G$-stable free $\mathcal{A}$-Lie lattice $\mathcal{L}$ in $\mathcal{T}(\bX)$ for some $G$-stable affine formal model $\mathcal{A}$ of $A$. By rescaling $\mathcal{L}$ if necessary, we may assume that $\bU$ is $\mathcal{L}$-admissible \cite[Lemma $7.6$]{AWI}. Under this assumption, \cite[Proposition $4.3.6$]{AW} showed that $\mathcal{L}':= \mathcal{B}\otimes_{\mathcal{A}}\mathcal{L}$ is a $G$-stable $\mathcal{B}$-Lie lattice in $\mathcal{T}(\bU)$ for any choice of a $G$ stable $\mathcal{L}$-stable affine formal model $\mathcal{B}$ in $\mathcal{O}(\bU)$. This is even free as a $\mathcal{B}$-module. Let $J \leq G_\mathcal{L}$ be an open normal subgroup of $G$ such that $(\mathcal{L},J)$ and ($\mathcal{L}',J)$ are trivialising pairs (this is thanks to \cite[Proposition $4.3.6$]{AW}). By definition, it is enough to show that the diagram
\[\begin{tikzcd}
\widehat{U(\mathcal{L})}_K \rtimes_{J}G \arrow{r}{p_{G,H,J}^{\bX}} \arrow[swap]{d}{r^{\bU}_{G,J}} & \widehat{U(\mathcal{L})}_K \rtimes_{J}H\arrow{d}{r^{\bU}_{H,J}} \\
\widehat{U(\mathcal{L'})}_K \rtimes_{J}G \arrow{r}{p_{G,H,J}^{\bU}} & \widehat{U(\mathcal{L'})}_K \rtimes_{J}H \\
\end{tikzcd}\]
is commutative.\\
Note that $J$ is of finite index in $G$ and in $H$, so that we can choose a set of representatives $ 1=g_1,g_2,...,g_m,..., g_n$ ($m \leq n$) of $G$ modulo $J$ such that $G/J= \lbrace \bar{g_1},...,\bar{g_n}\rbrace$ and  $H/J= \lbrace \bar{g_1},\bar{g_2},...\bar{g_m} \rbrace $. Therefore
\begin{center}
$\widehat{U( \mathcal{L})}_K \rtimes_{J} G \simeq \widehat{U(\mathcal{L})}_K \ast G/J= \lbrace \sum_{i=1}^n a_i \bar{g_i} : \, \, a_i \in \widehat{U(\mathcal{L})}_K \rbrace$
\end{center}
and
\begin{center}
$\widehat{U(\mathcal{L})}_K \rtimes_{J} H \simeq \widehat{U(\mathcal{L})}_K \ast H/J= \lbrace \sum_{i=1}^m a_i \bar{g_i} : \, \, a_i \in \widehat{U(\mathcal{L})}_K \rbrace.$
\end{center}
Notice that here we identified each $\bar{g_i} \in G/J$ with its image in $\widehat{U(\mathcal{L})}_K \ast G/J$. Furthermore, these formulas still hold when we replace $\mathcal{L}$ by $\mathcal{L'}$.
Thus \[r^{\bU}_{H,J} \circ p_{G,H,J}^{\bX}(\sum_{i=1}^n a_i \bar{g_i})= r^{\bU}_{H,J}(\sum_{i=1}^m a_i\bar{g_i} )= \sum_{i=1}^m \tilde{a_i} \bar{g}_i\]
and \[p_{G,H,J}^{\bU} \circ r^{\bU}_{G,J}( \sum_{i=1}^n a_i \bar{g_i})= p^{\bU}_{G,H,J}(\sum_{i=1}^n \tilde{a_i}\bar{g_i} )= \sum_{i=1}^m \tilde{a_i}\bar{g_i}.\]
Here for each $i$, $\tilde{a_i}$ denotes the image of $a_i$ in $ \widehat{U(\mathcal{L'})}_K $ via the canonical morphism $  \widehat{U(\mathcal{L})}_K \longrightarrow  \widehat{U(\mathcal{L'})}_K  $. This proves the commutativity of the diagram.
\end{proof}

\begin{coro} \label{keycoro} Suppose that $(\bX,G)$ is small with $\dim \bX =d$ and that the $\mathcal{A}$-Lie lattice $\mathcal{L}$ is smooth as an $\mathcal{A}$-module. Then there exist $m\geq 0$ such that the ring $\widehat{U(\pi^n\mathcal{L})}_K \rtimes_{J_n} G$ is an Auslander-Gorenstein ring of (injective) dimension at most $2d$ for any $n \geq m$ and  for any open normal subgroup $J_n$ of $G$ which is contained in $G_{\pi^n\mathcal{L}}$.
\end{coro}
\begin{proof}
Following \cite[Theorem $4.3$]{AWIII}, there exists $m \geq 0 $ such that the ring $\widehat{U(\pi^n \mathcal{L}})_K $ is Auslander-Gorenstein of dimension at most $2d$ for all $n \geq m$.
Thanks to Proposition \ref{keypro} and Remark \ref{keyre} it follows that $\widehat{U(\pi^n\mathcal{L})}_K \rtimes_J G$ is  Auslander-Gorenstein of dimension at most $2d$.
\end{proof}

Recall from \cite[Section $5.1$]{AWIII} that a two-sided Fr\'{e}chet-Stein algebra $A \simeq \varprojlim_n A_n$ is called \textbf{c-Auslander-Gorenstein} of dimension at most $d$ if each $A_n$ is an Auslander-Gorenstein ring with injective dimension at most $d$ for a non negative integer $d.$  If A is c-Auslander-Gorenstein of dimension at most $d$, then every (non-zero) coadmissible $A$-module has grade at most $d$. Moreover, every coadmissible $A$-module $M$ satisfies the \textbf{c-Auslander condition}, i.e. for any $i \geq 0$, one has $j_A(N) \geq i$ whenever $N$ is a coadmissible (right) submodule of $Ext^i_A(M,A)$.

\begin{theorem} \label{keytheorem}
Let $\bX=Sp(A)$ be a smooth affinoid variety of dimension $d$ and $G$ be a compact $p$-adic Lie group acting continuously on $\bX$ such that $(\bX,G)$ is small. Then the Fr\'{e}chet-Stein $K$-algebra $\wideparen{\mathcal{D}}(\bX,G)$ is coadmissibly Auslander-Gorenstein of dimension at most $2d$.
\end{theorem}
\begin{proof}
We may choose a $G$-stable affine formal model $\mathcal{A}$ in $A$  and a $G$-stable \textit{free} $\mathcal{A}$-Lie lattice $\mathcal{L}$ in $L=Der_K(A)$ and a good chain $(J_n)$ for $\mathcal{L}$ such that
\[\wideparen{\mathcal{D}}(\bX,G)\simeq \varprojlim_n \widehat{U(\pi^n \mathcal{L}})_K \rtimes_{J_n}G.\]
By Corollary \ref{keycoro}, there exists $m \geq 0 $ such that the ring $\widehat{U(\pi^n \mathcal{L}})_K \rtimes_{J_n}G$ is Auslander-Gorenstein of dimension at most $2d$ for each $n \geq m$, so the theorem follows.
\end{proof}
\begin{definition}\label{defdim} Let $M$ be a non-zero (left or right) coadmissible $\wideparen{\mathcal{D}}(\bX,G)$-module. The dimension of $M$ is defined by:
\[d_G(M):=2d-j_{\wideparen{\mathcal{D}}(\bX,G)}(M).\]
We set $d_G(M)=0$ if $M=0$.
\end{definition}
\begin{re}\label{redim}
\begin{itemize}
\item[(i)] By the above discussion, one has $0\leq d_G(M) \leq 2d$ for any coadmissible $M$.
\item[(ii)] If $H$ is an open subgroup of $G$, then there exists an open normal subgroup $N$ of $G$ which is contained in $H$ (\cite{AW}, Lemma $3.2.1$) and 
\begin{center}
 $\wideparen{\mathcal{D}}(\bX,G) \simeq \wideparen{\mathcal{D}}(\bX,N)\rtimes_N G \simeq \wideparen{\mathcal{D}}(\bX,N)\ast G/N.$
 \end{center}
 Then the $\wideparen{\mathcal{D}}(\bX,G)$-module $M$ is also coadmissible as a $\wideparen{\mathcal{D}}(\bX,N)$-module. Therefore $d_G(M)=d_N(M)$ by Proposition \ref{keypro}(ii) and we obtain $d_G(M)=d_N(M)=d_H(M).$
 For this reason, we will write $d(M)$ instead of $d_G(M)$ for simplicity.
\end{itemize}
\end{re}

\begin{pro} \label{exact} Let \[0 \longrightarrow M_1 \longrightarrow M_2 \longrightarrow M_3 \longrightarrow 0\]
be an exact sequence of coadmissible $\wideparen{\mathcal{D}}(\bX,G)$-modules. Then 
\[d(M_2)=\max\lbrace d(M_1), d(M_3)\rbrace.\]

\end{pro}
\begin{proof} Suppose that 
$$\wideparen{\mathcal{D}}(\bX,G) \cong \varprojlim_n \widehat{U(\pi^n \mathcal{L}})_K \rtimes_{J_n}G$$
for a $G$-stable free Lie lattice $\mathcal{L}$ of $Der_K(\mathcal{O}(\bX))$ and a good chain $(J_n)$ for $\mathcal{L}$. Write $\wideparen{{D}}:=\wideparen{\mathcal{D}}(\bX,G)$ and $D_n:= \widehat{U(\pi^n \mathcal{L}})_K \rtimes_{J_n}G$.
Note that there exists an integer $m$ such that for every $i$ and $n \geq m$, one has that (Remark \ref{redim}(i)):
 $$j_{\wideparen{{D}}}(M_i)= j_{D_n}(D_n \otimes_{\wideparen{{D}}}M_i).$$
 Since $\wideparen{{D}}\longrightarrow D_n$ is a flat morphism (\cite[Remark $3.2$]{ST2003}), it follows that the sequence
 \[0 \longrightarrow D_n\otimes_{\wideparen{{D}}}M_1 \longrightarrow D_n\otimes_{\wideparen{{D}}}M_2 \longrightarrow D_n\otimes_{\wideparen{{D}}}M_3 \longrightarrow 0 \]
 is exact.
 Now applying \cite[Proposition $4.5(ii)$]{LT} gives the result.
\end{proof}
\begin{ex}
The  $\wideparen{\mathcal{D}}(\bX,G)$-module $\wideparen{\mathcal{D}}(\bX,G)$  is of dimension $2d$. Indeed
\[Hom_{\wideparen{\mathcal{D}}(\bX,G)}(\wideparen{\mathcal{D}}(\bX,G), \wideparen{\mathcal{D}}(\bX,G)) \cong \wideparen{\mathcal{D}}(\bX,G).\]
Hence $j(\wideparen{\mathcal{D}}(\bX,G))=0$, so that $d(\wideparen{\mathcal{D}}(\bX,G))= 2d.$  Similarly, the free $\wideparen{\mD}(\bX,G)$-module $\wideparen{\mD}(\bX,G)^n $ of rank $n \geq 1$ is of dimension $2d$. 
\end{ex}
A less trivial example is given by the following proposition:
\begin{pro}\label{example}
Let $\bX$ be a smooth affinoid variety of dimension d and $P\in \mathcal{D}(\bX)$ be a regular differential operator (i.e $P$ is not a zero divisor of $\mathcal{D}(\bX)$). Then the coadmissible left $\wideparen{\mathcal{D}}(\bX,G)$-module $$M= \wideparen{\mathcal{D}}(\bX,G)/\wideparen{\mathcal{D}}(\bX,G)P$$ is of dimension $d(M) \leq 2d-1$. 
\end{pro}
\begin{proof}
Write $D:= \mathcal{D}(\bX)$ and $\wideparen{D}:= \wideparen{\mathcal{D}}(\bX,G)$. 
Choose a $G$-stable free $\mathcal{A}$-Lie lattice $\mathcal{L}$ of $Der_K(\mathcal{O}(\bX))$ for some $G$-stable affine formal model $\mathcal{A}$ in $\mathcal{O}(\bX)$. Then 
\[\wideparen{D} \cong \varprojlim_n \widehat{U(\pi^n\mathcal{L})}_K \rtimes_{J_n}G\]
  is a Fr\'{e}chet-Stein structure on $\wideparen{D}$. Write $D_n:= \widehat{U(\pi^n\mathcal{L})}_K \rtimes_{J_n}G$, then  
 \[M \cong \varprojlim_n D_n/D_nP.\]
 Thus there is a $n\geq 0$ such that $d(M)= d(D_n/D_nP)$. Furthermore, one has that 
 $$D_n/D_nP \cong D_n \otimes_{D}D/DP.$$
 The ring $D_n$ is flat as a right $D$-module. It follows 
 \[Ext^i_{D}(D/DP, D) \otimes_D D_n \cong Ext^i_{D_n}(D_n\otimes_DD/DP, D_n).\]
 As a consequence, we obtain the inequality $d_{D_n}(D_n/D_nP) \leq d_D(D/DP).$ But 
 $d_D(D/DP)<2d$, since otherwise one would have $j_D(D/DP)=0$, whence $Hom_D(D/DP,D)= \lbrace Q \in D: QP=0 \rbrace \neq 0$, in contradiction to the regularity of $P$. \end{proof}
\medspace

Let $\bX$ be an affinoid variety and $G$ a $p$-adic Lie group acting continuously on $\bX$ such that $(\bX,G)$ is small. The following proposition relates the dimension to the side-changing functors \ref{chanfunc}.
\begin{pro} \label{dimcom} Let $M$ be a coadmissible left $\wideparen{\mathcal{D}}(\bX,G)$-module. Then there is an isomorphism of left $\wideparen{\mathcal{D}}(\bX,G)$-modules
\[Ext^i_{\wideparen{\mathcal{D}}(\bX,G)}(\Omega(\bX) \otimes_{\mathcal{O}(\bX)} M, \wideparen{\mathcal{D}}(\bX,G))\simeq Hom_{\mathcal{O}(\bX)}(\Omega(\bX), Ext^i_{\wideparen{\mathcal{D}}(\bX,G)}(M, \wideparen{\mathcal{D}}(\bX,G))).\]
In particular, $d(M)=d(\Omega(\bX) \otimes_{\mathcal{O}(\bX)} M)$.
\end{pro}
\begin{proof}
The proof uses the same arguments as in \cite[Lemma $5.2$]{AWIII}. Write $A=\mathcal{O}(\bX)$, $\Omega:= \Omega(\bX)$, $\wideparen{{D}}:= \wideparen{\mathcal{D}}(\bX,G)$. Then the left hand side is exactly the $i$-th cohomology of the complex $RHom_{\wideparen{{D}}}(\Omega \otimes_{A}^\mathbb{L}  M, \wideparen{{D}})$, as $\Omega$ is a projective $A$-module. Now, the right hand side is the $i$-th cohomology of $RHom_A(\Omega, RHom_{\wideparen{{D}}}(M,\wideparen{{D}}))$. Thus, using the derived tensor-hom adjunction gives the first part of the proposition. For the second part, note that since $\Omega$ is a finitely generated projective $A$-module, one has
\[Hom_{A}(\Omega,Ext^i_{\wideparen{{D}}}(M,\wideparen{{D}})) \cong \Omega^*\otimes_A Ext^i_{\wideparen{{D}}}(M,\wideparen{{D}})),\]
where $\Omega^*=Hom_A(\Omega,A)$ is its dual. Thus, if  $Hom_{A}(\Omega,Ext^i_{\wideparen{{D}}}(M,\wideparen{{D}}))=0 $, then 
\[Ext^i_{\wideparen{{D}}}(M,\wideparen{{D}}) \cong (\Omega\otimes_A\Omega^*) \otimes_A Ext^i_{\wideparen{{D}}}(M,\wideparen{{D}})) \cong \Omega \otimes_A Hom_{A}(\Omega,Ext^i_{\wideparen{{D}}}(M,\wideparen{{D}})) =0. \]
Here, $\Omega\otimes_A\Omega^* \cong A$, as $\Omega$ is an invertible $A$-module. By consequence, $Ext^i_{\wideparen{{D}}}(M,\wideparen{{D}})=0$  if and only if $Hom_{A}(\Omega,Ext^i_{\wideparen{{D}}}(M,\wideparen{{D}}))=0 $ and hence $d(M)=d(\Omega(\bX) \otimes M)$.
\end{proof}
\section{Ext functors for coadmissible equivariant $\mathcal{D}$-modules}\label{section_four}

In this section, we develop the global theory of equivariant Ext functors. This is considerably more complicated than in the non-equivariant setting \cite{AWIII}, since
there is no global sheaf $\wideparen{\mathcal{D}}(-,G)$ playing the role of the coherent sheaf 
$\wideparen{\mathcal{D}}$ in the equivariant setting.

\subsection{Modules over the sheaf of rings $\mathcal{Q}$}\label{section_Q}
In this subsection we prepare on the Banach level the globalization of the Ext functors, by showing several compatibilities of the local Ext groups.\\

Let $\bX$ be a smooth affinoid variety of dimension $d$ and $G$ be a compact $p$-adic Lie group acting continuously on $\bX$. Fix a $G$-stable affine formal model $\mathcal{A}$ in $A=\mathcal{O}(\bX)$, a $G$-stable $\mathcal{A}$-Lie lattice $\mathcal{L}$ of $\mathcal{T}(\bX)=Der_K(A)$ and an open normal subgroup $J$ of $G$ which is contained in $G_\mathcal{L}$ (which means that $(\mathcal{L}, J)$ is a $\mathcal{A}$-trivialising pair).\\ 
\\
\textbf{Notation}: Throughout this section, we will be working under the following notations and assumptions:
\begin{itemize}
\item[$\ast$]$\mathcal{L}$ is a smooth $\mathcal{A}$-module, which means that $\mathcal{L}$ is projective and finitely generated over $\mathcal{A}$.
\item[$\ast$] When $H$ is an open subgroup of $G$, $\bX_{w}(\mathcal{T})/H$ denotes the set of all open affinoid subsets $\bU \in \bX_{w}(\mathcal{T})$ such that $(\bU,H)$ is small. If $\bU \in \bX_{w}(\mathcal{T})/H$, then $H$ is called an \textbf{$\bU$-small} subgroup of $G$.
\end{itemize}

Before Def. \ref{defq}, we have recalled the Grothendieck topologies $\bX_{ac}(\mathcal{L},G) \subset \bX_{w}(\mathcal{L},G)$ of $G$-stable $\mathcal{L}$-accessible resp. $\mathcal{L}$-admissible affinoid subdomains of $\bX$. Recall from \ref{defq} the sheaf of rings on $\bX_{w}(\mathcal{L},G)$

\[\mathcal{Q}(-,G):=\widehat{\mathcal{U}(\mathcal{L})}_K \rtimes_JG.\]

If $H \leq G$ is an open compact subgroup of $G$ and $J$ is contained in $G_{\mathcal{L}}\cap H$, then $\mathcal{Q}(-,H)$ is also a sheaf on the canonical (strong) Grothendieck topology $\bX_{w}(\mathcal{L},H)$ containing all the $H$-stable $\mathcal{L}$-admissible affinoid subdomains of $\bX$.  In the sequel, if there is no ambiguity, we denote $\mathcal{Q}(-,G)$ simply by $\mathcal{Q}$ whenever the groups $G$ and $J$ are given. 
\begin{definition} (\cite[Definition $4.3.17$]{AW}
Let $M$ be a finitely generated $\mathcal{Q}(\bX)$-module. Then there is a presheaf $Loc_\mathcal{Q}(M)$ on $\bX_{ac}(\mathcal{L},G)$  associated to $M$ which is defined as follows: 
\[Loc_\mathcal{Q}(M)(\bY):= \mathcal{Q}(\bY)\otimes_{\mathcal{Q}(\bX)}M\]
for all $\bY \in \bX_{ac}(\mathcal{L},G)$.
\end{definition}
Following \cite[Corollary $4.3.19$]{AW}, under the extra assumption on $\mathcal{L}$ that $[\mathcal{L}, \mathcal{L}]\subset \pi\mathcal{L}$ and $\mathcal{L}.A \subset \pi \mathcal{A}$, then $Loc_\mathcal{Q}(M)$ is a sheaf of $\mathcal{Q}$-modules on $\bX_{ac}(\mathcal{L},G)$ for every finitely generated $\mathcal{Q}(\bX)$-module $M$.

\begin{definition} Let $\mathcal{U}$ be a $\mathbf{X}_{ac}(\mathcal{L},G)$-covering of $\bX$. Then a $\mathcal{Q}$-module $\mathcal{M}$ on $\bX_{ac}(\mathcal{L},G)$ is said to be $\mathcal{U}$-coherent if for any $\bY \in \mathcal{U}$, there exists a finitely generated $\mathcal{Q}(\bY)$-module $M$ such that
\[Loc_{\mathcal{Q}\vert_\mathcal{Y}}(M) \cong \mathcal{M}\vert_\mathcal{Y},\]
where $\mathcal{Y}:= \bX_{ac}(\mathcal{L},G)\cap \bY_w$.
\end{definition}
It is proved in \cite[Theorem $4.3.21$]{AW} that if $[\mathcal{L},\mathcal{L}] \subset \pi\mathcal{L}$, $\mathcal{L}. \mathcal{A} \subset \pi \mathcal{A}$, then for any $\mathcal{U}$-coherent sheaf of $\mathcal{Q}$-modules $\mathcal{M}$, $\mathcal{M}(\bX)$ is a finitely generated $\mathcal{Q}(\bX)$-module and we have an isomorphism of $\mathcal{Q}$-modules
\[Loc_\mathcal{Q}(\mathcal{M}(\bX)) \tilde{\longrightarrow} \mathcal{M}.\]
In the following, we fix a sheaf $\mathcal{M}$ of $\mathcal{Q}$-modules on $\bX_{ac}(\mathcal{L},G)$.

\begin{pro}\label{isoQ} Let $H$ be an open normal subgroup of $G$. There is an isomorphism of right $\mathcal{Q}(\bX,H)$-modules
\[{p}^i_{G,H}(\bX):\,\,\, Ext^i_{\mathcal{Q}(\bX,G)}(\mathcal{M}(\bX),\mathcal{Q}(\bX,G))\tilde{\longrightarrow} Ext^i_{\mathcal{Q}(\bX,H)}(\mathcal{M}(\bX),\mathcal{Q}(\bX,H)).\]
Furthermore, if $H' \leq H$ is another open normal subgroup of $G$, then one has $${p}^i_{H,H'}(\bX) \circ {p}^i_{G,H}(\bX)= {p}^i_{G,H'}(\bX).$$
\end{pro}
\begin{proof}
Write $M:=\mathcal{M}(\bX)$. The first part of the proposition is in fact a consequence of Lemma \ref{keylemma1} and Lemma \ref{keylemma2}. When $i=0$,   then  
$$p_{G,H}(\bX)(f):={p}^0_{G,H}(\bX)(f)= p_{G,H}^{\bX} \circ f $$ 
for $f \in Hom_{\mathcal{Q}(\bX,G)}(M, \mathcal{Q}(\bX,G))$, where $p_{G,H}^{\bX}$ is the projection map $ \mathcal{Q}(\bX,G) \longrightarrow \mathcal{Q}(\bX,H)$ which is defined in Lemma \ref{keylemma1}.
For the second part, if $H' \leq H$ are open normal subgroups of $G$, then both $H$ and $H'$ are of finite index in $G$ and $H'$ is of finite index in $H$ (since $G$ is compact). Hence we can choose a $\mathcal{Q}(\bX,H')$-basis $\lbrace 1=g_1,g_2,...,g_m,...,g_n \rbrace $ of $\mathcal{Q}(\bX,G)$ such that $\lbrace g_1,...,g_m \rbrace$  is a basis of $\mathcal{Q}(\bX,H)$ as a $\mathcal{Q}(\bX,H')$-module. Then by definition
\begin{center}
$p_{G,H'}^{\bX}(a_1g_1+a_2g_2+...+a_mg_m +...+a_ng_n)=a_1$
\end{center}
and
\begin{center}
$p_{H,H'}^{\bX} \circ p_{G,H}^{\bX}(a_1g_1+a_2g_2+...+a_mg_m +...+a_ng_n)=p_{H,H'}^{\bX}(a_1g_1+a_2g_2+...+a_mg_m)=a_1$
\end{center}
This implies  $p_{G,H'}^{\bX}=p_{H,H'}^{\bX} \circ p_{G,H}^{\bX}$. Therefore ${p}_{H,H'}(\bX) \circ {p}_{G,H}(\bX)= {p}_{G,H'}(\bX)$, which means that the assertion is true for $i=0$. For  $i > 0$, by taking a resolution of $M$ by free $\mathcal{Q}(\bX,G)$-modules of finite rank, the case $i>0$ reduces to the case $i=0$.
\end{proof}
\begin{lemma}\label{lemma4} Let $\varphi: A \rightarrow B$ be a flat morphism of rings and $M$ be a finitely presented $A$-module. There is an isomorphism of right $B$-modules
\[Ext^i_{A}(M,A) {\otimes}_A B \longrightarrow Ext^i_B( B {\otimes}_A M ,B).\]
\end{lemma}
\begin{proof}
This is a consequence of the flatness of $\varphi$ and the Five lemma, using the fact that $M$ is finitely presented as an $A$-module.
\end{proof}

\begin{pro}Let $\bU \in \bX_{ac}(\mathcal{L},G)$. There is a morphism of right $\mathcal{Q}(\bX,G)$-modules
\[\tau_{\bX,\bU,G}^{i}: \,\,Ext^i_{\mathcal{Q}(\bX,G)}(\mathcal{M}(\bX), \mathcal{Q}(\bX,G))\rightarrow Ext^i_{\mathcal{Q}(\bU,G)}(\mathcal{M}(\bU),\mathcal{Q}(\bU,G)).\]
\end{pro}
\begin{proof}
Denote $M:= \mathcal{M}(\bX)$. Then
\[\mathcal{M}(\bU) \cong \mathcal{Q}(\bU,G)\otimes_{\mathcal{Q}(\bX,G)}M.\]
Since $U$ is $\mathcal{L}$-accessible, the morphism 
\[\mathcal{Q}(\bX,G)\longrightarrow \mathcal{Q}(\bU,G)\]
is flat (Proposition \ref{flat}). Now applying Lemma \ref{lemma4} gives 
\[Ext^i_{\mathcal{Q}(\bU,G)}(\mathcal{M}(\bU),\mathcal{Q}(\bU,G))\cong Ext^i_{\mathcal{Q}(\bX,G)}(M, \mathcal{Q}(\bX,G))\otimes_{\mathcal{Q}(\bX,G)}\mathcal{Q}(\bU, G).\]
By consequence, we obtain the natural morphism of right $\mathcal{Q}(\bX,G)$-modules:
\[\tau_{\bX,\bU,G}^i: \,\,Ext^i_{\mathcal{Q}(\bX,G)}(\mathcal{M}(\bX), \mathcal{Q}(\bX,G))\rightarrow Ext^i_{\mathcal{Q}(\bU,G)}(\mathcal{M}(\bU),\mathcal{Q}(\bU,G)).\]
\end{proof}

\begin{pro} \label{commutativeQ} Let $H$ be a normal open subgroup of $G$ and $\bU \in \bX_{ac}(\mathcal{L},G)$. Then the following diagram is commutative:
\begin{equation}
 \begin{tikzcd}
Ext^i_{\mathcal{Q}(\bX,G)}(\mathcal{M}(\bX), \mathcal{Q}(\bX,G)) \arrow{r}{{p}_{G,H}^i(\bX)} \arrow[swap]{d}{\tau_{\bX,\mathbf{U},G}^i} & Ext^i_{\mathcal{Q}(\bX,H)}(\mathcal{M}(\bX), \mathcal{Q}(\bX,H)) \arrow{d}{\tau_{\bX,\bU,H}^i} \\
Ext^i_{\mathcal{Q}(\bU,G)}(\mathcal{M}(\bU), \mathcal{Q}(\bU,G)) \arrow{r}{{p}_{G,H}^i(\bU)}& Ext^i_{\mathcal{Q}(\bU,H)}(\mathcal{M}(\bU), \mathcal{Q}(\bU,H))
\end{tikzcd}
\end{equation}
\end{pro}
\begin{proof}
Write $M:= \mathcal{M}(\bX)$. Then $Loc_\mathcal{Q}(M)\cong \mathcal{M}$. Hence
 $$\mathcal{M}(\bU)\cong \mathcal{Q}(\bU,G)\otimes_{\mathcal{Q}(\bX,G)}M\cong \mathcal{Q}(\bU,H)\otimes_{\mathcal{Q}(\bX,H)}M.$$
Now take a resolution $P^{.}$ of $M$ by free $\mathcal{Q}(\bX,G)$-modules of finite rank. Since $\bU \in \bX_{ac}(\mathcal{L},G)$ is supposed to be $\mathcal{L}$-accessible, the ring $\mathcal{Q}(\bU,G)$ is flat over $\mathcal{Q}(\bX,G)$ (Proposition \ref{flat}). This implies  $\mathcal{Q}(\bU,G)\otimes_{\mathcal{Q}(\bX,G)}P^{.}$ is also  a free resolution of $\mathcal{Q}(\bU,G)\otimes_{\mathcal{Q}(\bX,G)}M \cong \mathcal{M}(\bU)$. Hence it reduces to prove that for any $\mathcal{Q}(\bX,G)$-module $P$, the following diagram is commutative:
\[\begin{tikzcd}
Hom_{\mathcal{Q}(\bX,G)}(P,\mathcal{Q}(\bX,G)) \arrow{r}{p_{G,H}(\bX)} \arrow[swap]{d}{} & Hom_{\mathcal{Q}(\bX,H)}(P,\mathcal{Q}(\bX,H))\arrow{d}{} \\
Hom_{\mathcal{Q}(\bU,G)}(\mathcal{Q}(\bU,G)\otimes_{\mathcal{Q}(\bX,G)}P,\mathcal{Q}(\bU,G))\arrow{r}{p_{G,H}(\bU)} & Hom_{\mathcal{Q}(\bU,H)}(\mathcal{Q}(\bU,H)\otimes_{\mathcal{Q}(\bX,H)}P,\mathcal{Q}(\bU,H)).\\
\end{tikzcd}
\]
This means the diagram
\[\begin{tikzcd}
\mathcal{Q}(\bX,G) \arrow{r}{p^{\bX}_{G,H}} \arrow[swap]{d}{} & \mathcal{Q}(\bX,H)\arrow{d}{} \\
\mathcal{Q}(\bU,G)\arrow{r}{p^{\bU}_{G,H}} & \mathcal{Q}(\bU,H)\\
\end{tikzcd}
\]
is commutative, which is already contained in (the proof of) Corollary \ref{coro3}(ii).\\
\end{proof}

\subsection{Globalization of Ext functors}\label{heart}

In this subsection we prepare on the Fr\'echet level the globalization of the Ext functors, by showing several compatibilities of the local Ext groups.\\

 Let $\bX$ be a smooth rigid analytic space and $G$ be a $p$-adic Lie group acting continuously on $\bX$. For each non negative integer $i \in \mathbb{N}$, we will construct a global Ext functor $E^i$ from coadmissible $G$-equivariant left $\mathcal{D}_{\bX}$-modules to coadmissible $G$-equivariant right $\mathcal{D}_{\bX}$-modules. Let $\mathcal{M} \in \mathcal{C}_{\bX/G}$ be a coadmissible $G$-equivariant $\mathcal{D}_{\bX}$-module. Then locally we want $E^i(\mathcal{M})(\bU)$ to be isomorphic to $Ext^i_{\wideparen{\mathcal{D}}(\bU,H)}(\mathcal{M}(\bU), \wideparen{\mathcal{D}}(\bU,H))$ for every open affinoid subset $\bU \in \bX_{w}(\mathcal{T})$ and open subgroup $H \leq G$ such that $(\bU,H)$ is small. We need to show that such a local definition of $E^i(\mathcal{M})(\bU)$ is well-defined, i.e. independent of the choice of the subgroup $H$.

\begin{pro} \label{isoD} Suppose that $\bX$ is a smooth affinoid variety and $G$ is such that $(\bX,G)$ is small and $H$ is an open normal subgroup of $G$. Then for any left $\wideparen{\mathcal{D}}(\bX,G)$-module $M$, there is an isomorphism of right $\wideparen{\mathcal{D}}(\bX,H)$-modules:
\[\wideparen{p}^{i}_{G,H}(\bX):\,\,\, Ext^i_{\wideparen{\mathcal{D}}(\bX,G)}(M,\wideparen{\mathcal{D}}(\bX,G))\tilde{\longrightarrow} Ext^i_{\wideparen{\mathcal{D}}(\bX,H)}(M,\wideparen{\mathcal{D}}(\bX,H)).\]
Furthermore, if $H' \leq H$ is another open normal subgroup of $G$, then one has $$\wideparen{p}^i_{H,H'}(\bX) \circ \wideparen{p}^i_{G,H}(\bX)= \wideparen{p}^i_{G,H'}(\bX).$$
\end{pro}
\begin{proof} Since Lemma \ref{keylemma1} holds for the morphism of rings $ \wideparen{\mathcal{D}}(\bX,H) \longrightarrow \wideparen{\mathcal{D}}(\bX,G)$ (Proposition \ref{keypro1}), the proof of this proposition uses exactly the same arguments as in the proof of Proposition \ref{isoQ}. We just write down here the definition of $\wideparen{p}^{i}_{G,H}(\bX)$. Let $P^.$ be a resolution of $M$ by free $\wideparen{\mathcal{D}}(\bX,G)$-modules. Then $\wideparen{p}_{G,H}^i(\bX)$ is determined by taking the $i-th$-cohomology of the following isomorphism of complexes:
\begin{align*}
Hom_{\wideparen{\mathcal{D}}(\bX,G)}(P^., \wideparen{\mathcal{D}}(\bX,G)) &\longrightarrow Hom_{\wideparen{\mathcal{D}}(\bX,H)}(P^.,\wideparen{\mathcal{D}}(\bX,H))\\
f^. &\longmapsto  p_{G,H}^{\bX} \circ f^.
\end{align*}

In particular, when $i=0$ then for every $f \in Hom_{\wideparen{\mathcal{D}}(\bX,G)}(M, \wideparen{\mathcal{D}}(\bX,G))$, one has 
$$\wideparen{p}_{G,H}(\bX)(f):=\wideparen{p}_{G,H}^0(\bX)(f):= p_{G,H}^{\bX} \circ f.$$
Here we recall that $p_{G,H}^{\bX}$ is the projection map 
\begin{align*}
p_{G,H}^{\bX}: \,\,\wideparen{\mathcal{D}}(\bX,G) &\longrightarrow \wideparen{\mathcal{D}}(\bX,H)\\
\sum_{i=0}^m a_i\bar{g}_i &\longmapsto a_0,
\end{align*}
where $\bar{g_0},...,\bar{g}_m$ denote the images of the set of cosets $G/H$ (which is finite) in $\wideparen{\mathcal{D}}(\bX,G)$.
\end{proof}



 Let $(\bX,G)$ be small as above and $M$ be a coadmissible (left) $\wideparen{\mathcal{D}}(\bX,G)$-module. Suppose that  $H \leq G$ is an open normal subgroup of $G$. Let us choose a $G$-stable free $\mathcal{A}$-Lie lattice $\mathcal{L}$ for some $G$-stable affine formal model $\mathcal{A}$ in $\mathcal{O}(\bX)$ and a good chain ($J_n)$ for this Lie lattice such that $J_n \leq H$ for any $n$. Then we may form the sheaves of rings
\begin{equation}\label{Q1}
\mathcal{Q}_n(-,G)=\widehat{\mathcal{U}(\pi^n\mathcal{L})}_K\rtimes_{J_n}G,\,\, \text{and}\,\,\mathcal{Q}_n(-,H)=\widehat{\mathcal{U}(\pi^n\mathcal{L})}_K\rtimes_{J_n}H
\end{equation}
 on $\bX_{ac}(\mathcal{L},G)$ and $\bX_{ac}(\mathcal{L},H)$ respectively. Hence 
\begin{center}
 $\wideparen{\mathcal{D}}(\bX,G)\simeq \varprojlim_n \mathcal{Q}_n(\bX,G)$ and $\wideparen{\mathcal{D}}(\bX,H)\simeq \varprojlim_n \mathcal{Q}_n(\bX,H)$. 
\end{center}
 Thus the projection map (\ref{projection2}) : $p_{G,H}^{\bX}: \wideparen{\mathcal{D}}(\bX,G) \longrightarrow \wideparen{\mathcal{D}}(\bX,H) $ is defined as the inverse limit of the maps (\ref{projection1}):  $p_{G,H,n}^{\bX}: \mathcal{Q}_n(\bX,G) \longrightarrow \mathcal{Q}_n(\bX,H)$. Suppose that $M\cong \varprojlim_n M_n$ with $M_n= \mathcal{Q}_n(\bX,G)\otimes_{\wideparen{\mathcal{D}}(\bX,G)}M$, which is finitely generated over $\mathcal{Q}_n(\bX,G)$. Then following Proposition \ref{isoQ} for every $n$, there is also an isomorphism of $D_n(\bX,H)$-modules
\[{p}^i_{G,H,n}(\bX):\,\,\, Ext^i_{\mathcal{Q}_n(\bX,G)}(M_n,\mathcal{Q}_n(\bX,G))\tilde{\longrightarrow} Ext^i_{\mathcal{Q}_n(\bX,H)}(M_n,\mathcal{Q}_n(\bX,H)).\]

\begin{lemma} \label{plim} There is a commutative diagram
\[\begin{tikzcd}
Ext^i_{\wideparen{\mathcal{D}}(\bX,G)}(M,\wideparen{\mathcal{D}}(\bX,G)) \arrow{r}{\wideparen{p}_{G,H}^i(\bX)} \arrow[swap]{d}{} & Ext^i_{\wideparen{\mathcal{D}}(\bX,H)}(M, \wideparen{\mathcal{D}}(\bX,H))\arrow{d}{} \\
Ext^i_{\mathcal{Q}_n(\bX,G)}(M_n,\mathcal{Q}_n(\bX,G)) \arrow{r}{{p}^i_{G,H,n}(\bX)} & Ext^i_{\mathcal{Q}_n(\bX,H)}(M_n,\mathcal{Q}_n(\bX,H)).
\end{tikzcd}
\]
In particular, $\wideparen{p}_{G,H}^i(\bX)$ equals the inverse limit of the maps ${p}_{G,H,n}^i(\bX)$.\\

\end{lemma}

\begin{proof}

Note that $\wideparen{\mathcal{D}}(\bX,G)$ (which is finitely freely generated as a $\wideparen{\mathcal{D}}(\bX,H)$-module ) is a coadmissible $\wideparen{\mathcal{D}}(\bX,H)$-module. It follows that $M$ is coadmissbile as a $\wideparen{\mathcal{D}}(\bX,H)$-module, so 
\[Ext^i_{\wideparen{\mathcal{D}}(\bX,G)}(M,\wideparen{\mathcal{D}}(\bX,G)) \cong \varprojlim_n Ext^i_{\mathcal{Q}_n(\bX,G)}(M_n,\mathcal{Q}_n(\bX,G))\]
and
\[Ext^i_{\wideparen{\mathcal{D}}(\bX,H)}(M,\wideparen{\mathcal{D}}(\bX,H)) \cong \varprojlim_n Ext^i_{\mathcal{Q}_n(\bX,H)}(M_n,\mathcal{Q}_n(\bX,H)).\]
 \cite[Lemma $8.4$]{ST2003}. These isomorphisms give the definitions of the two vertical arrows of the diagram in the lemma.\\
For any $\wideparen{\mathcal{D}}(\bX,G)$-module $P$ (which is not necessary coadmissible), we have an isomorphism of $\mathcal{Q}_n(\bX,H)$-modules
\[\mathcal{Q}_n(\bX,G)\otimes_{\wideparen{\mathcal{D}}(\bX,G)} P \simeq (\mathcal{Q}_n(\bX,H)\otimes_{\wideparen{\mathcal{D}}(\bX,H)}\wideparen{\mathcal{D}}(\bX,G))\otimes_{\wideparen{\mathcal{D}}(\bX,G)} P \simeq \mathcal{Q}_n(\bX,H)\otimes_{\wideparen{\mathcal{D}}(\bX,H)}  P.\]

Now, let $P^{.}\rightarrow M \rightarrow 0$ be a projective resolution of $M$ by free $\wideparen{\mathcal{D}}(\bX,G)$-modules. Since $\wideparen{\mathcal{D}}(\bX,G)$ is free over $\wideparen{\mathcal{D}}(\bX,H)$ on both sides, $P^.$ is also a projective resolution of $M$ in $Mod(\wideparen{\mathcal{D}}(\bX,H))$. Moreover, it is proved \cite[Remark $3.2$]{ST2003} that the canonical maps $\wideparen{\mathcal{D}}(\bX,G) \rightarrow \mathcal{Q}_n(\bX,G)$ and $\wideparen{\mathcal{D}}(\bX,H)\rightarrow \mathcal{Q}_n(\bX,H)$ are right flat, so that $\mathcal{Q}_n(\bX,G) \otimes P$ and $\mathcal{Q}_n(\bX,H) \otimes P$ are projective resolutions of $\mathcal{Q}_n(\bX,G) \otimes M$ and $\mathcal{Q}_n(\bX,H) \otimes M$, respectively. Thus, by definitions of $\wideparen{p}^i_{G,H}(\bX)$ and $p^i_{G,H,n}(\bX)$ it suffices to show that for any $\wideparen{\mathcal{D}}(\bX,G)$ -module $P$, the diagram 
 \[\begin{tikzcd}
Hom_{\wideparen{\mathcal{D}}(\bX,G)}(P,\wideparen{\mathcal{D}}(\bX,G)) \arrow{r}{p_{G,H}^{\bX} \circ} \arrow[swap]{d}{id \bar{\otimes} -} & Hom_{\wideparen{\mathcal{D}}(\bX,H)}(P, \wideparen{\mathcal{D}}(\bX,H))\arrow{d}{id \bar{\otimes} -} \\
Hom_{\mathcal{Q}_n(\bX,G)}(\mathcal{Q}_n(\bX,G)\otimes P,\mathcal{Q}_n(\bX,G)) \arrow{r}{p_{G,H,n}^{\bX}\circ} & Hom_{\mathcal{Q}_n(\bX,H)}(\mathcal{Q}_n(\bX,H)\otimes P,\mathcal{Q}_n(\bX,H))
\end{tikzcd}
\]
is commutative. (Note that, for every $f \in Hom_{\wideparen{\mathcal{D}}(\bX,G)}(P,\wideparen{\mathcal{D}}(\bX,G))$, the map  
$$id\bar{\otimes}f \in Hom_{\mathcal{Q}_n(\bX,G)}(\mathcal{Q}_n(\bX,G)\otimes P,\mathcal{Q}_n(\bX,G)) $$ 
is defined by $(id \bar{\otimes}f) (a\otimes m)=af(m)$ with $a \in \mathcal{Q}_n(\bX,G)$, $m \in P$).
This reduces to show that the diagram
\[\begin{tikzcd}
\wideparen{\mathcal{D}}(\bX,G) \arrow{r}{p_{G,H}^{\bX}} \arrow[swap]{d}{} & \wideparen{\mathcal{D}}(\bX,H)\arrow{d}{} \\
\mathcal{Q}_n(\bX,G)\arrow{r}{p_{G,H,n}^{\bX}} & \mathcal{Q}_n(\bX,H)\\
\end{tikzcd}
\]
is commutative. Now the proof can be completed by applying Corollary \ref{coro3}(i).
\end{proof}

Now let $\bX$ be a smooth rigid analytic space, let $G$ be a $p$-adic Lie group which acts continuously on $\bX$. Let $\mathcal{M} \in \mathcal{C}_{\bX/G}$ be a coadmissible $G$-equivariant left $\mathcal{D}_{\bX}$-module. Fix an open affinoid subset $\bU \in \bX_{w}(\mathcal{T})$. Recall that for any $\bU$-small subgroup $H \leq G$, one has an isomorphism of coadmissible $H$-equivariant $\mathcal{D}_{\bU}$-modules:
 \[\mathcal{M}\vert_{\bU} \simeq Loc^{\wideparen{\mathcal{D}}(\bU,H)}_{\bU}(\mathcal{M}(\bU)).\]

\begin{definition} If  $(\bU,H)$ is small, we define for all $i \geq 0$: $$E^i(\mathcal{M})(\bU,H):=  Ext^i_{\wideparen{\mathcal{D}}(\bU,H)}(\mathcal{M}(\bU), \wideparen{\mathcal{D}}(\bU,H)).$$
\end{definition}
 This is, in fact, a coadmissible right $\wideparen{\mathcal{D}}(\bU,H)$-module. 
 
 \begin{pro} \label{aa} Let $H' \leq H$ be  $\bU$-small open subgroups  of $G$. There is an isomorphism of right $\mathcal{D}(\bU)$-modules:
\[\wideparen{p}^i_{H',H}(\bU): \,\, Ext^i_{\wideparen{\mathcal{D}}(\bU,H')}(\mathcal{M}(\bU),\wideparen{\mathcal{D}}(\bU,H'))\tilde{\longrightarrow} Ext^i_{\wideparen{\mathcal{D}}(\bU,H)}(\mathcal{M}(\bU),\wideparen{\mathcal{D}}(\bU,H)).\] 
The pairs $(E^i(\mathcal{M})(\bU,H), \wideparen{p}_{H',H}^i(\bU))$ form an inverse system when $H', H$ run over the (partially ordered) set of all $\bU$-small subgroups of $G$.
\end{pro}
\begin{proof}
Since $H' \leq G$ is open compact in $H$, there is an open normal subgroup $N$ of $H$ which is contained in $H'$ (\cite[Lemma $3.2.1$]{AW}). Hence following Proposition \ref{isoD}, one has the following isomorphism:
\[\wideparen{p}^i_{H',N}(\bU): \,\, Ext^i_{\wideparen{\mathcal{D}}(\bU,H')}(\mathcal{M}(\bU),\wideparen{\mathcal{D}}(U,H'))\tilde{\longrightarrow} Ext^i_{\wideparen{\mathcal{D}}(\bU,N)}(\mathcal{M}(\bU),\wideparen{\mathcal{D}}(\bU,N))\]  
and
\[\wideparen{p}^i_{H,N}(\bU): \,\, Ext^i_{\wideparen{\mathcal{D}}(\bU,H)}(\mathcal{M}(\bU),\wideparen{\mathcal{D}}(\bU,H))\tilde{\longrightarrow} Ext^i_{\wideparen{\mathcal{D}}(\bU,N)}(\mathcal{M}(\bU),\wideparen{\mathcal{D}}(\bU,N)).\] 
Now, we define $$\wideparen{p}^i_{H',H}(\bU):= (\wideparen{p}^i_{H,N}(\bU))^{-1}\circ \wideparen{p}^i_{H',N}(\bU).$$
 By definition $\wideparen{p}^i_{H',H}(\bU)$ is an isomorphism of $\mD(\bU)$-modules. Furthermore, this is independent from the choice of an open normal subgroup $N$ of $H$. Indeed, if $N^{'}\leq N$ is an other normal subgroup of $H$, then $N'$ is also normal in $N$, thus Proposition \ref{isoD} gives
 \begin{center}
 $\wideparen{p}^i_{H',N'}(\bU)=\wideparen{p}^i_{H',N}(\bU)\circ \wideparen{p}^i_{N,N'}(\bU)$ and $\wideparen{p}^i_{H,N'}(\bU)=\wideparen{p}^i_{H,N}(\bU)\circ \wideparen{p}^i_{N,N'}(\bU)$.
 \end{center}
 Consequently
 \begin{align*}
(\wideparen{p}^i_{H,N'}(\bU))^{-1}\circ \wideparen{p}^i_{H',N'}(\bU)&=(\wideparen{p}^i_{H,N}(\bU)\circ \wideparen{p}^i_{N,N'}(\bU))^{-1}\circ \wideparen{p}^i_{H',N}(\bU)\circ \wideparen{p}^i_{N,N'}(\bU)\\
&=(\wideparen{p}^i_{H,N}(\bU))^{-1}\circ \wideparen{p}^i_{H',N}(\bU).
\end{align*}
\end{proof}

\begin{re}\label{rep} If $H'$ is normal in $G$, then we may choose $N=H'$ in the proof of the above proposition and we have
\[\wideparen{p}^i_{H',H}(\bU)= (\wideparen{p}^i_{H,H'}(\bU))^{-1}.\] 
\end{re}
Thanks to Proposition \ref{aa}, we are ready to give the following definition:
\begin{definition}\label{keydef} For every open affinoid subset $\bU \in \bX_{w}(\mathcal{T})$, we define:
\[E^i (\mathcal{M})(\bU):= \varprojlim_H E^i(\mathcal{M})(\bU,H)= \varprojlim_H Ext^i_{\wideparen{\mathcal{D}}(\bU,H)}(\mathcal{M}(\bU),\wideparen{\mathcal{D}}(\bU,H)),\] 
where the inverse limit is taken over the set of all $\bU$-small subgroups $H$ of $G$.
\end{definition}
\begin{re}\label{proj} $E^i (\mathcal{M})(\bU)$ obviously has a structure of right $\mathcal{D}(\bU)$-module. Furthermore, we obtain from Proposition $\ref{isoD}$ that the natural map
\[E^i (\mathcal{M})(\bU) \longrightarrow E^i(\mathcal{M})(\bU,H)\]
is a bijection for every $\bU$-small subgroup $H$ of $G$.
\end{re}

\begin{lemma}\label{lem} Let $U\cong\varprojlim_n U_n$, $V\cong\varprojlim_n V_n$ be Fr\'{e}chet-Stein $K$-algebras and $U \rightarrow V$ be a continuous morphism of Fr\'{e}chet-Stein algebras. Suppose that for each $n$, the induced morphism of rings $ U_n \rightarrow V_n$ is flat. Then for any coadmissible $U$-module $M$, there is an isomorphism of right $V$-modules
\[Ext^i_{U}(M,U) \wideparen{\otimes}_U V \longrightarrow Ext^i_V( V \wideparen{\otimes}_U M ,V).\]
\end{lemma} 
\begin{proof}
Since $M$ is coadmissible as $U$-module, we have the following isomorphism:
\[M \cong \varprojlim_n U_n \otimes_U M = \varprojlim_n M_n\]
with $M_n:= U_n \otimes_U M$ for every $n$. Hence $V \wideparen{\otimes}_UM\cong \varprojlim_n V_n \otimes_{U_n}M_n$ and this implies:
\[Ext^i_{U}(M,U) \wideparen{\otimes}_U V \cong \varprojlim_n Ext^i_{U_n}(M_n,U_n) {\otimes}_{U_n} V_n \]
and 
\[Ext^i_V( V \wideparen{\otimes}_U M ,V) \cong \varprojlim_n Ext^i_{V_n}( V_n {\otimes}_{U_n} M_n ,V_n).\]
So it suffices to prove, for every $n$, the existence of compatible isomorphisms of right $V_n$-modules
\[Ext^i_{U_n}(M_n,U_n) {\otimes}_{U_n} V_n \tilde{\longrightarrow} Ext^i_{V_n}( V_n {\otimes}_{U_n} M_n ,V_n).\]
This follows from Lemma \ref{lemma4}.
\end{proof} 
\begin{pro}\label{res2} Suppose that $(\bU,H)$ is small and $\bV \subset \bU $ is an open affinoid subset in $\bX_{w}(\mathcal{T})/ H$, then there is a morphism of right $\wideparen{\mathcal{D}}(\bU,H)$-modules
\[\wideparen{\tau}_{\bU,\bV, H}^i: \,\,E^i(\mathcal{M})(\bU,H)\rightarrow E^i(\mathcal{M})(\bV,H).\]
If $\bW \subset \bV\subset \bU$ are open subsets in $\bX_{w}(\mathcal{T})/H$, then the diagram
\[\begin{tikzcd}[column sep=small]
E^i(\mathcal{M})(\bU,H) \arrow{r}{\wideparen{\tau}_{\bU,\bV,H}^i}  \arrow{rd}{\wideparen{\tau}_{\bU,\bW,H}^i} 
  & E^i(\mathcal{M})(\bV,H) \arrow{d}{\wideparen{\tau}_{\bV,\bW,H}^i} \\
    & E^i(\mathcal{M})(\bW,H)
\end{tikzcd}\]
is commutative.
\end{pro}
\begin{proof} We choose a free $\mathcal{A}$-Lie lattice $\mathcal{L}$ of $\mathcal{T}(\bU)$ for some $H$-stable affine formal model $\mathcal{A}$ of $\mathcal{O}(\bU)$ and a good chain $(J_n)$ for $\mathcal{L}$. By rescaling $\mathcal{L}$, we may assume that $\bV$ is $\mathcal{L}$-accessible.\\ Recall the sheaves $\mathcal{Q}_n(-,H)$ on $\bU_{ac}(\mathcal{L},H)$. Under these assumptions, the morphism  $$ \mathcal{Q}_n(\bU,H) \longrightarrow \mathcal{Q}_n(\bV,H)$$ is  flat.
Thus we can apply  Lemma \ref{lem} and obtain:
\[E^i(\mathcal{M})(\bU,H)\wideparen{\otimes}_{\wideparen{\mathcal{D}}(\bU,H)}\wideparen{\mathcal{D}}(\bV,H) \simeq E^i(\mathcal{M})(\bV,H).\]
 This provides a natural map of right $\wideparen{\mathcal{D}}(\bU,H)$-modules
\begin{align*}
 E^i(\mathcal{M})(\bU,H) & \longrightarrow E^i(\mathcal{M})(\bU,H)\wideparen{\otimes}_{\wideparen{\mathcal{D}}(\bU,H)}\wideparen{\mathcal{D}}(\bV,H) \simeq E^i(\mathcal{M})(\bV,H)\\
 m &\longmapsto m \wideparen{\otimes} 1.
\end{align*} 
 If $\bW \subset \bV \subset \bU$ are open subsets in $\bX_{w}(\mathcal{T})/ H$, then following \cite[Corollary $7.4$]{AWI}
\begin{align*}
E^i(\mathcal{M})(\bU,H)\wideparen{\otimes}_{\wideparen{\mathcal{D}}(\bU,H)} \wideparen{\mathcal{D}}(\bW,H) & \simeq E^i(\mathcal{M})(\bU,H)\wideparen{\otimes}_{\wideparen{\mathcal{D}}(\bU,H)} \wideparen{\mathcal{D}}(\bV,H)  \wideparen{\otimes}_{\wideparen{\mathcal{D}}(\bV,H)} \wideparen{\mathcal{D}}(\bW,H)\\
&\simeq E^i(\mathcal{M})(\bV,H)\wideparen{\otimes}_{\wideparen{\mathcal{D}}(\bV,H)} \wideparen{\mathcal{D}}(\bW,H)\,\,\ ( \simeq E^i(\mathcal{M})(\bW,H)).
\end{align*}
Hence the commutative diagram follows.
\end{proof}

\begin{pro} \label{commutativeE} Let $H$ be an open compact subgroup of $G$ and $\bU, \bV \in \bX_{w}(\mathcal{T})/H$ such that $ \bV \subset \bU$. Suppose that $N \leq H$ is another open compact subgroup of $G$. Then the following diagram is commutative: 
\begin{equation}\label{diag}
 \begin{tikzcd}
E^i(\mathcal{M})(\bU,N) \arrow{r}{\wideparen{p}_{N,H}^i(\bU)} \arrow[swap]{d}{\wideparen{\tau}_{\bU,\bV,N}^{i}} & E^i(\mathcal{M})(\bU,H) \arrow{d}{\wideparen{\tau}_{\bU,\bV,H}^i} \\%
E^i(\mathcal{M})(\bV,N) \arrow{r}{\wideparen{p}_{N,H}^i(\bV)}& E^i(\mathcal{M})(\bV,H).
\end{tikzcd}
\end{equation}

\end{pro}
\begin{proof}
Firstly, suppose that $N$ is normal in $H$. Then following Remark \ref{rep}
\begin{center}
$\wideparen{p}^i_{N,H}(\bU)= (\wideparen{p}^i_{H,N}(\bU))^{-1}$ and $\wideparen{p}^i_{N,H}(\bV)= (\wideparen{p}^i_{H,N}(\bV))^{-1}.$
\end{center}
We need to prove that:
 \[\wideparen{\tau}_{\bU,\bV,N}^i \circ \wideparen{p}_{H,N}^i (\bU) = \wideparen{p}_{H,N}^i (\bV) \circ \wideparen{\tau}_{\bU,\bV,H}^i.\]
For this we choose a $H$-stable free $\mathcal{A}$-Lie lattice $\mathcal{L}$ in $\mathcal{T}(\bU)$ for some $H$-stable affine formal model $\mathcal{A}$ of $\mathcal{O}(\bU)$ and a good chain $(J_n)$ for $\mathcal{L}$ such that $J_n \leq N$ for any $n$. By rescaling $\mathcal{L}$ if necessary, we may suppose in addition that $\bV$ is $\mathcal{L}$-accessible, which means that $\bV \in \bU_{ac}(\mathcal{L},H)$. Consider the sheaves of rings 
\begin{center}
$ \mathcal{Q}_n(-,H)=\widehat{\mathcal{U}(\pi^n\mathcal{L})}_K\rtimes_{J_n}H$ and $\mathcal{Q}_n(-,N)=\widehat{\mathcal{U}(\pi^n\mathcal{L})}_K\rtimes_{J_n}N$ 
\end{center}
on $\bU_{ac}(\mathcal{L},H)$ and $\bU_{ac}(\mathcal{L},N)$ respectively. Since $\bV \in \bU_{ac}(\mathcal{L},H)$, then 
\begin{center}
$\wideparen{\mathcal{D}}(\bU,H)=\varprojlim_n\mathcal{Q}_n(\bU,H)$ and $\wideparen{\mathcal{D}}(\bU,N) =\varprojlim_n\mathcal{Q}_n(\bU,N)$,\\
$\wideparen{\mathcal{D}}(\bV,H)=\varprojlim_n\mathcal{Q}_n(\bV,H)$ and $\wideparen{\mathcal{D}}(\bV,N) =\varprojlim_n\mathcal{Q}_n(\bV,N).$
\end{center}
 Since all modules appearing in the diagram (\ref{diag}) are coadmissible, following Lemma \ref{plim}, it suffices to prove that:
\[\begin{tikzcd}
Ext^i_{\mathcal{Q}_n(\bU,H)}(\mathcal{Q}_n(\bU,H)\otimes_{} \mathcal{M}(\bU),\mathcal{Q}_n(\bU,H)) \arrow{r}{{p}^i_{H,N,n}(\bU)} \arrow[swap]{d}{\tau_{\bU,\bV,H}^i} & Ext^i_{\mathcal{Q}_n(\bU,N)}(\mathcal{Q}_n(\bU,N)\otimes_{} \mathcal{M}(\bU),\mathcal{Q}_n(\bU,N))\arrow{d}{\tau_{\bU,\bV,N}^i} \\
Ext^i_{\mathcal{Q}_n(\bV,H)}(\mathcal{Q}_n(\bV,H)\otimes_{} \mathcal{M}(\bV),\mathcal{Q}_{n}(\bV,H)) \arrow{r}{{p}^i_{H,N,n}(\bV)} & Ext^i_{\mathcal{Q}_n(\bV,N)}(\mathcal{Q}_n(\bV,N)\otimes_{}\mathcal{M}(\bV),
\mathcal{Q}_n(\bV,N))
\end{tikzcd}\]
is commutative. Now, Proposition \ref{commutativeQ} gives the result for the case $N$ is normal in $H$.\\

When $N$ is not longer necessarily normal in $H$, then there is an open normal subgroup $N'$ of $H$ in $N$ (as $H$ is compact and $N$ is open). Then
\begin{align*}
\wideparen{\tau}_{\bU,\bV,H}^i \circ \wideparen{p}_{N,H}^i (\bU)&= \wideparen{\tau}_{\bU,\bV,H}^i \circ( \wideparen{p}_{N',H}^i (\bU)\circ (\wideparen{p}_{N',N}^i (\bU))^{-1})\\
&=\wideparen{p}_{N',H}^i (\bV)\circ \wideparen{\tau}_{\bU,\bV,N}^i\circ (\wideparen{p}_{N',N}^i (\bU))^{-1}\\
&=\wideparen{p}_{N,H}^i (\bV) \circ \wideparen{p}_{N',N}^i (\bV) \circ \wideparen{\tau}_{\bU,\bV,N}^i\circ (\wideparen{p}_{N',N}^i (\bU))^{-1}\\
&= \wideparen{p}_{N,H}^i (\bV) \circ \wideparen{\tau}_{\bU,\bV,N}^i \circ \wideparen{p}_{N',N}^i (\bU)\circ (\wideparen{p}_{N',N}^i (\bU))^{-1}\\
&=\wideparen{p}_{N,H}^i (\bV) \circ \wideparen{\tau}_{\bU,\bV,N}^i.
\end{align*}
Hence the commutativity of (\ref{diag}) holds for $N$.
\end{proof}

\begin{pro} \label{res}  For every $\bU, \bV \in \bX_{w}(\mathcal{T})$ such that $\bV \subset \bU$, there is a right  $\mathcal{D}(\bU)$-linear restriction map
\[\wideparen{\tau}_{\bU,\bV}^i: E^i(\mathcal{M})(\bU) \longrightarrow E^i (\mathcal{M})(\bV).\] 

\end{pro}
\begin{proof}
Let $N$ be a $\bV$-small subgroup of $G$ and $N_{\bU}$ the stabilizer of $\bU$ in $N$. Then there exists a $\bU$-small subgroup $H$ inside $N_{\bU}$, which is normal in $N$ \cite[Lemma $3.2.1$]{AW}. By Proposition \ref{res2}, one has a morphism of right $\mathcal{D}(\bU)$-modules
\[\wideparen{\tau}^i_{\bU,\bV,H}: E^i(\mathcal{M})(\bU,H) \longrightarrow E^i (\mathcal{M})(\bV,H).\]
Then we can define a right $\mathcal{D}(\bU)$-linear morphism 

$$E^i(\mathcal{M})(\bU)\longrightarrow E^i(\mathcal{M})(\bV,N)$$ 

as the composition
\begin{equation*}
 E^i(\mathcal{M})(\bU)=\varprojlim_H E^i(\mathcal{M})(\bU,H) \longrightarrow E^i(\mathcal{M})(\bU,H)\xrightarrow{\wideparen{\tau}_{\bU,\bV,H}^i} E^i(\mathcal{M})(\bV,H) \xrightarrow{\wideparen{p}^i_{H,N}(\bV)} E^i(\mathcal{M})(\bV,N).
 \end{equation*}
 If $H'$ is another open  $\bU$-small subgroup of $H$ in $N_{\bU}$, then Proposition \ref{aa} and Proposition \ref{commutativeE} ensure that this map is independent of the choice of $H$. It amounts to showing that if $N' \leq N$ is another $\bV$-small subgroup in $G$, then the following diagram is commutative:
 \begin{equation}\label{diag1}
 \begin{tikzcd}
 E^i(\mathcal{M})(\bU) \arrow{r}{} \arrow{rd}{} & E^i(\mathcal{M})(\bV,N')  \arrow{d}{\wideparen{p}^i_{N',N}(\bV)} \\& E^i(\mathcal{M})(\bV,N).
 \end{tikzcd}
 \end{equation}
 
 If we take $H':= N'_{\bU} \cap H$, then $H'$ is a $\bU$-small subgroup of $N'_{\bU}$. Again by Proposition \ref{commutativeE} and Proposition \ref{aa}, it follows that the diagram 
 \[
  \setlength{\arraycolsep}{2pt}
  \begin{array}{*{9}c}
     E^i(\mathcal{M})(\bU,H') & \xrightarrow{\wideparen{\tau}_{\bU,\bV,H'}^{i}} & E^i(\mathcal{M})(\bV,H') & \xrightarrow{\wideparen{p}^i_{H',N'}(\bV)} & E^i(\mathcal{M})(\bV,N') \\
    \Ldownarrow & & \Ldownarrow & & \Ldownarrow & & \\
     E^i(\mathcal{M})(\bU,H) & \xrightarrow{\wideparen{\tau}_{\bU,\bV,H}^{i}} & E^i(\mathcal{M})(\bV,H) & \xrightarrow{\wideparen{p}^i_{H,N}(\bV)} & E^i(\mathcal{M})(\bV,N)
  \end{array}
\]
 is commutative, so that the triangle (\ref{diag1}) is commutative.
 Now, by the universal property of the inverse limit, this induces a right $\mathcal{D}(\bU)$-linear map 
 \begin{center}
 $E^i(\mathcal{M})(\bU)=\varprojlim_H E^i(\mathcal{M})(\bU,H) \longrightarrow E^i(\mathcal{M})(\bV)=\varprojlim_N E^i(\mathcal{M})(\bV,N)$.
 \end{center}
\end{proof}

\begin{re} Thanks to Proposition $\ref{res}$, we see that $E^i(\mathcal{M})$ is a presheaf of $\mathcal{D}_{\bX}$-modules on the set $\bX_{w}(\mathcal{T})$. Furthermore, $E^i(\mathcal{M})\vert_{\bU} = E^i(\mathcal{M}\vert_{\bU})$, if $\bU \in \bX_w{(\mathcal{T})}$ is an open affinoid subset of $\bX$.
 \end{re}
 
Let us now define a $G$-equivariant structure on the presheaf $E^i(\mathcal{M})$ of right $\mathcal{D}_{\bX}$-modules on $\bX_{w}(\mathcal{T})$. Let  $g \in G$ and $ \bU \in \bX_{w}(\mathcal{T})$. Recall that $g$ defines a morphism
\begin{align*}
g=g^\mathcal{O}(\bU): \,\, \mathcal{O}(\bU) & \longrightarrow \mathcal{O}(g\bU)\\
f &\longmapsto g.f.
\end{align*}  
Here, for any function $f \in \mathcal{O}(\bU)$, the function $g.f \in \mathcal{O}(g\bU)$ is defined as
\begin{center}
$(g.f)(y):= f(g^{-1}y), \, \forall y \in g\bU$
\end{center}

This induces an isomorphism of $K$-Lie algebras
\begin{align*}
g^{\mathcal{T}}:={g}^{\mathcal{T}}(\bU)  : \,\, \mathcal{T}(\bU) &\longrightarrow \mathcal{T}(g\bU)\\
v &\longmapsto g\circ v\circ g^{-1}
\end{align*}
which is linear relative to $g^\mathcal{O}(\bU)$. \\
Let $H$ be a $\bU$-small subgroup of $G$. Suppose that  $\mathcal{A}$ is an $H$-stable formal model in $\mathcal{O}(\bU)$ and $\mathcal{L}$ is an $H$-stable $\mathcal{A}$-Lie lattice in $\mathcal{T}(\bU)$. 
\begin{lemma}\label{Glemma}
\begin{itemize}
\item[(i)] $g(\mathcal{A})$ is a $gHg^{-1}$-stable formal model of $\mathcal{O}(g\bU)$ and ${g}^\mathcal{T}(\mathcal{L})$ is a $gHg^{-1}$-stable $g(\mathcal{A})$-Lie lattice in $\mathcal{T}(g\bU)$. If $\mathcal{L}$ is smooth (resp. free ) over $\mathcal{A}$, then ${g}^\mathcal{T}(\mathcal{L})$ is smooth (resp. free) over $g(\mathcal{A})$.
\item[(ii)] Let $\bV \in \bX_{w}(\mathcal{T})$ be an open affinoid subset in $\bU$. If  $\bV$ is a $\mathcal{L}$-accessible subdomain of $\bU$ then $g\bV$ is a ${g}^\mathcal{T}(\mathcal{L})$-accessible subdomain of $g\bU$.
\end{itemize}
\end{lemma}
\begin{proof}
\begin{itemize}
\item[(i)]Let $g \in G$ and $f \in \mathcal{O}(\bU)$.
Since the morphism $g: \mathcal{O}(\bU) \longrightarrow \mathcal{O}(g\bU)$ is $K$-linear, then 
\item[$\centerdot$]$Kg(\mathcal{A})= g(K\mathcal{A})=g(\mathcal{O}(\bU))=\mathcal{O}(g\bU).$
\item[$\centerdot$] if $h \in H$ then $ghg^{-1}(g(\mathcal{A}))=g({h}\mathcal{A}) \subset g(\mathcal{A})$, so that $g(\mathcal{A})$ is $gHg^{-1}$-stable.\\

Similarly,
\item[$\centerdot$] $K{g}^\mathcal{T}(\mathcal{L})= {g}^\mathcal{T}(K\mathcal{L})={g}^{\mathcal{T}}(\mathcal{T}(\bU))=\mathcal{T}(g\bU)$.
\item[$\centerdot$]$(ghg^{-1})^\mathcal{T}({g}^\mathcal{T}(\mathcal{L}))=(gh)^\mathcal{T}(\mathcal{L}) \subset {g}^\mathcal{T}(\mathcal{L}).$ Hence $\mathcal{L}$ is a $gHg^{-1}$-stable Lie lattice in $\mathcal{T}(gU).$
It remains to prove that if $\mathcal{L}$ is smooth (resp. free ) over $\mathcal{A}$, then ${g}^\mathcal{T}(\mathcal{L})$ is smooth (resp. free) over $g(\mathcal{A})$. But this is straighforward in view of the fact that we have the bijection
\[g^\mathcal{T}\vert_\mathcal{L}: \mathcal{L} \tilde{\longrightarrow} g^\mathcal{T}(\mathcal{L})\]
which is linear with respect to the (iso)morphism of rings $g\vert_\mathcal{A}: \mathcal{A} \longrightarrow g(\mathcal{A})$. 
\item[(ii)]
 Without loss of generality, we may suppose that $\bU=\bX$ and $\bV$ is a rational subset of $\bX$. We prove $(ii)$ by induction on $n$. If $\bV$ is $\mathcal{L}$-accessible in zero steps, then $\bV=\bX$ and $\bV$ is ${g}^\mathcal{T}(\mathcal{L})$-accessible in zero steps. Now, suppose that the statement is true for $n-1$. Let $\bV$ be $\mathcal{L}$-accessible in $n$ steps. We may assume that there is a chain $\bV \subset \bZ \subset \bX$ such that $\bZ$ is $\mathcal{L}-$accessible in $(n-1)$ steps, $\bV=\bZ(f)$  for some non zero $f \in \mathcal{O}(\bZ)$ and there is a $\mathcal{L}$-stable formal model $\mathcal{C} \subset \mathcal{O}(\bZ)$ such that $\mathcal{L}.f \subset \pi \mathcal{C}$. Then $g\bV= \lbrace g y : y \in \bV \rbrace$ and 
 \begin{center}
 $(g\bZ)(g.f)= \lbrace gy : \vert (g.f)(gy)\vert \leq 1, \forall y \in \bZ \rbrace = \lbrace gy: \vert f(g^{-1}gy) \vert = \vert f(y) \vert \leq 1, \forall y \in \bV  \rbrace$. 
 \end{center}
 Hence $g\bV=(g\bZ)(gf)$. By assumption $g\bZ \subset \bX$ is ${g}^\mathcal{T}(\mathcal{L})$-accessible in $(n-1)$ steps. Furthermore, by $(i)$, $g( \mathcal{C})$ is a $gHg^{-1}$-stable formal model of $\mathcal{O}(g\bZ)$ and it is straightforward that ${g}^\mathcal{T}(\mathcal{L}) .(gf) \subset \pi.(g(\mathcal{C}))$. This shows that $g\bU$ is also $g\mathcal{L}$-accessible in $n$-steps. \end{itemize}
\end{proof}

Let $(\bU,H)$ be small. Recall the isomorphism of $K-$algebras
\[\wideparen{g}_{\bU,H}: \wideparen{\mathcal{D}}(\bU,H)\tilde{\longrightarrow} \wideparen{\mathcal{D}}(g\bU,gHg^{-1})\]
and the isomorphism
\[g^\mathcal{M}_{\bU,H}: \,\, \mathcal{M}(\bU) \longrightarrow \mathcal{M}(g\bU)\]
which is linear with respect to $\wideparen{g}_{\bU,H}$ (since $\mathcal{M}\in \mathcal{C}_{\bX/G}$).\\

\begin{pro}\label{Giso1}
Suppose that $(\bU,H)$ is small and $g \in G$. 
There exists a $K$-linear map
\begin{align*}
g^{E^i(\mathcal{M})}_{\bU,H}\,\,\,: E^i(\mathcal{M})(\bU,H)  \longrightarrow E^i(\mathcal{M})(g\bU, gHg^{-1})
\end{align*}
such that
for every $ a \in \wideparen{\mathcal{D}}(\bU,H), m \in E^i(\mathcal{M})(\bU,H)$, we have:
\begin{equation}\label{equationg}
g^{E^i(\mathcal{M})}_{\bU,H}(ma)= g^{E^i(\mathcal{M})}_{\bU,H}(m). \wideparen{g}_{\bU,H}(a).
\end{equation}
\end{pro}
\begin{proof}

Denote ${}^gH:=gHg^{-1}$. We construct a map
\[g^{E^i(\mathcal{M})}_{\bU,H}: \,\, Ext^i_{\wideparen{\mathcal{D}}(\bU,H)}(\mathcal{M}(\bU), \wideparen{\mathcal{D}}(\bU,H))\longrightarrow Ext^i_{\wideparen{\mathcal{D}}(g\bU,{}^gH)}(\mathcal{M}(g\bU), \wideparen{\mathcal{D}}(g\bU,{}^gH))\]
as follows. Let $P^.\rightarrow \mathcal{M}(\bU)\rightarrow 0$ be a free resolution of $\mathcal{M}(\bU)$ as a $\wideparen{\mathcal{D}}(\bU,H)$-module. Then by regarding each term of the complex $P^.$ as a $\wideparen{\mathcal{D}}(g\bU,{}^gH)$-module via the isomorphism of rings $\wideparen{g}_{\bU,H}^{-1} : \,\,\wideparen{\mathcal{D}}(g\bU,{}^g H)\tilde{\longrightarrow} \wideparen{\mathcal{D}}(\bU,H)$, we can also view $P^.$ as a free resolution of $\mathcal{M}(g\bU) \simeq \mathcal{M}(\bU)$ by $\wideparen{\mathcal{D}}(g\bU,{}^gH)$-modules and denote it by ${}^gP^.$. Thus, the map $g^{E^i(\mathcal{M})}_{\bU,H}$ can be defined by applying the i-th cohomology functor to the morphism of complexes whose components are morphisms of the form:
\begin{align*}
\phi_{\bU,H}^g: \,\, Hom_{\wideparen{\mathcal{D}}(\bU,H)}(P^k, \wideparen{\mathcal{D}}(\bU,H) )& \longrightarrow  Hom_{\wideparen{\mathcal{D}}(g\bU,{}^gH)}({}^g(P^k), \wideparen{\mathcal{D}}(g\bU,{}^gH))\\
f &\longmapsto \,\,\,\,\, \wideparen{g}_{\bU,H} \circ f,
\end{align*}
where ${}^g(P^k)$ denotes the component $P^k$ of the complex $P^.$ viewed as a $\wideparen{\mathcal{D}}(g\bU,{}^gH)$-module  via the morphism $\wideparen{g}_{\bU,H}^{-1}$. We need to check the following facts:
\begin{itemize}
\item[$1.$] $\wideparen{g}_{\bU,H} \circ f \in Hom_{\wideparen{\mathcal{D}}(g\bU,{}^gH)}({}^g(P^k), \wideparen{\mathcal{D}}(g\bU,{}^gH))$, which means $\wideparen{g}_{\bU,H} \circ f$ is $\wideparen{\mathcal{D}}(g\bU,{}^gH)$-linear. Indeed, if $b \in \wideparen{\mathcal{D}}(g\bU,{}^gH)$ and $m \in {}^g(P^k)$, then:
\begin{align*}
(\wideparen{g}_{\bU,H} \circ f) (b.m)=\wideparen{g}_{\bU,H}(f(\wideparen{g^{-1}}_{\bU,H}(b)m)= \wideparen{g}_{\bU,H}(\wideparen{g^{-1}}_{\bU,H}(b)f(m))=b(\wideparen{g}_{\bU,H} \circ f) (m).
\end{align*}
Here the second equality follows from the fact that $f$ is $\wideparen{\mathcal{D}}(\bU,H)$-linear and the third one is based on the fact that $\wideparen{g}_{\bU,H} $ is a morphism of $K$-algebras.
\item[$2.$] For any $a \in \wideparen{\mathcal{D}}(\bU,H)$ and $f \in Hom_{\wideparen{\mathcal{D}}(\bU,H)}(P^k, \wideparen{\mathcal{D}}(\bU,H) )$, we check that: \[\phi_{\bU,H}^g(fa)=\phi_{\bU,H}^g(f)\wideparen{g}_{\bU,H}(a).\]
Let $ m \in {}^g(P)^k$. We compute:
\[\phi_{\bU,H}^g(fa)(m)= \wideparen{g}_{\bU,H}( f(m)a)= \wideparen{g}_{\bU,H}(f(m))\wideparen{g}_{\bU,H}(a)= \phi_{\bU,H}^g(f)(m)\wideparen{g}_{\bU,H}(a).\]
\end{itemize}
Finally, by definition of $g^{E^i(\mathcal{M})}_{\bU,H}$, this implies (\ref{equationg}).
\end{proof}

Next, we study some properties of the morphisms $g^{E^i(\mathcal{M})}_{\bU,H}$ with $g \in G$. Let $\mathcal{L}$ be a $H$-stable free $\mathcal{A}$-Lie lattice of $\mathcal{T}(\bU)$ for some $H$-stable affine formal model $\mathcal{A}$ in $\mathcal{O}(\bU)$.
Write $\mathcal{A}':= g(\mathcal{A})$  and $\mathcal{L'}:=g^\mathcal{T}(\mathcal{L})$. Lemma \ref{Glemma} shows us that there is a bijection between the following Grothendieck topologies:
\begin{align*}
\bU_{ac}(\mathcal{L},H) &\longrightarrow (g\bU)_{ac}(\mathcal{L}', gHg^{-1})\\
\bV &\longmapsto g\bV.
\end{align*}
Furthermore, if $J \leq G_{\mathcal{L}}$ is an open normal subgroup of $G$ such that $(J,\mathcal{L})$ is an $\mathcal{A}$-trivialising pair in $H$, then $(gJg^{-1},\mathcal{L}')$ is also an $\mathcal{A}'$-trivialising pair in $gHg^{-1}$. Let $(J_n)_n$ be a good chain for $\mathcal{L}$ and recall the sheaves $\mathcal{Q}_n$ from (\ref{Q1}). If $\bV \in \bU_{ac}(\mathcal{L},H)$, there is an isomorphism of $K$-algebras:
\[g^{\mathcal{Q}_n}_{\bV,H}: \mathcal{Q}_n(\bV,H) \longrightarrow \mathcal{Q}_n(g\bV,gHg^{-1}).\]
These maps satisfy
\[\wideparen{g}_{\bV,H}=\varprojlim_n g^{\mathcal{Q}_n}_{\bV,H}.\]
Let $\mathcal{M}$ be a coadmissible $G$-equivariant $\mathcal{D}_{\bX}$-module. For each $n$, we define the following presheaves. Let $\bV \in \bU_w(\mathcal{L},H)$, then:
 \begin{equation}\label{M1}
 \mathcal{M}_n(\bV):= \mathcal{Q}_n(\bV,H) \otimes_{\wideparen{\mathcal{D}}(\bU,H)}\mathcal{M}(\bU)
\end{equation}
and
\begin{equation}\label{M2}
\mathcal{M}_n(g\bV):= \mathcal{Q}_n(g\bV, gHg^{-1}) \otimes_{\wideparen{\mathcal{D}}(g\bU,gHg^{-1})}\mathcal{M}(g\bU).
 \end{equation}
 Note that they defined sheaves of modules on $\bU_w(\mathcal{L},H)$ and on $(g\bU)_w(\mathcal{L'},gHg^{-1})$, respectively. If $\bV \in \bU_w(\mathcal{L},H)$, the isomorphism 
\[g^\mathcal{M}_{\bV,H}: \mathcal{M}(\bV) \longrightarrow \mathcal{M}(g\bV)\]
induces an isomorphism
\begin{align*}
g^{\mathcal{M}_n}_{\bV,H}: \,\,\mathcal{M}_n(\bV) &\longrightarrow \mathcal{M}_n(g\bV)\\
s \otimes m &\longmapsto g_{\bV,H}^{\mathcal{Q}_n}(s) \otimes g^\mathcal{M}_{\bV,H}(m).
\end{align*}

We have the following result:
\begin{pro} \label{Giso} Let $g \in G$. There is an isomorphism
\[g^{E^i(\mathcal{M})}_{\bU,H,n}: \,\,  Ext_{\mathcal{Q}_n(\bU,H)}^i(\mathcal{M}_n(\bU), \mathcal{Q}_n(\bU,H)) \longrightarrow Ext_{\mathcal{Q}_n(g\bU,{}^gH)}^i(\mathcal{M}_n(g\bU), \mathcal{Q}_n(g\bU,{}^gH)) \]
such that 
\begin{itemize}
\item[$1.$] For any $s \in \mathcal{Q}_n(\bU,H)$ and $m \in Ext_{\mathcal{Q}_n(\bU,H)}^i(\mathcal{M}_n(\bU), \mathcal{Q}_n(\bU,H))$, one has 
\[g^{E^i(\mathcal{M})}_{\bU,H,n}(ms)= g^{E^i(\mathcal{M})}_{\bU,H,n}(m).g^{\mathcal{Q}_n}_{\bU,H}(s).\]
\item[$2.$] Let $\bV \in \bU_{ac}(\mathcal{L},H)$. Then the following diagram is commutative:
\[
  \setlength{\arraycolsep}{2pt}
  \begin{array}{*{9}c}
        Ext_{\mathcal{Q}_n(\bU,H)}^i(\mathcal{M}_n(\bU), \mathcal{Q}_n(\bU,H)) &  \xrightarrow{g_{\bU,H, n}^{E^i(\mathcal{M})}} & Ext_{\mathcal{Q}_n(g\bU,{}^gH)}^i(\mathcal{M}_n(g\bU), \mathcal{Q}_n(g\bU,{}^gH))
    \\
    \Ldownarrow{\tau_{H,n}^i} & & \Ldownarrow{\tau_{gH,n}^{i}} & & 
    \\
    Ext_{\mathcal{Q}_n(\bV,H)}^i(\mathcal{M}_n(\bV), \mathcal{Q}_n(\bV,H)) &  \xrightarrow{g_{\bV,H,n}^{E^i(\mathcal{M})}} & Ext_{\mathcal{Q}_n(g\bV,{}^gH)}^i(\mathcal{M}_n(g\bV), \mathcal{Q}_n(g\bV,{}^gH))
  \end{array}
.\]
Here $\tau_{H,n}^i$ and $\tau_{gH,n}^i$ are restriction maps which are defined in Proposition $\ref{commutativeQ}$.
\end{itemize}
\end{pro}
\begin{proof}
\begin{itemize}
\item[$(1.)$] We define $g^{E^i(\mathcal{M})}_{\bU,H,n}$ similarly as defining $g^{E^{i}(\mathcal{M})}_{\bU,H}$ in Proposition \ref{Giso}. Let $P_n^{.} \rightarrow \mathcal{M}_n(\bU) \rightarrow 0$ be a resolution of $\mathcal{M}_n(\bU)$ by free $\mathcal{Q}_n(\bU,H)$-modules. Then by considering each term of this resolution as a $\mathcal{Q}_n(g\bU,{}^gH)$-module via the isomorphism of $K$-algebras $g^{\mathcal{Q}_n}_{\bU,H}: \mathcal{Q}_n(\bU,H) \longrightarrow \mathcal{Q}_n(g\bU,{}^gH)$, we see that $P_n^{.}$ is also a resolution of $\mathcal{M}_n(g\bU)$ by free $\mathcal{Q}_n(g\bU,{}^gH)$-modules. Let us denote this by ${}^gP_n^{.}$. Then the morphism $g^{E^i(\mathcal{M})}_{\bU,H,n}$ is determined by taking the $i$-th cohomology of the following morphism of complexes:
\begin{align*}
Hom_{\mathcal{Q}_n(\bU,H)}(P_n^{.}, \mathcal{Q}_n(\bU,H)) &\longrightarrow Hom_{\mathcal{Q}_n(g\bU,^{g}H)}({}^gP_n^{.}, \mathcal{Q}_n(g\bU,{}^gH))\\
f &\longmapsto g^{\mathcal{Q}_n}_{\bU,H} \circ f.
\end{align*}
Now the required property can be proved similarly as for $g^{E^i(\mathcal{M})}_{\bU,H}$ in proposition \ref{Giso}.
\item[$(2.)$] Note that 
\[\mathcal{M}_n(\bV) = \mathcal{Q}_n(\bV,H) \otimes_{\wideparen{\mathcal{D}}(\bV,H)}\mathcal{M}(\bV)\cong \mathcal{Q}_n(\bV,H) \otimes_{\mathcal{Q}_n(\bU,H)}\mathcal{M}_n(\bU).\]
\[\mathcal{M}_n(g\bV) = \mathcal{Q}_n(g\bV,{}^gH) \otimes_{\wideparen{\mathcal{D}}(g\bV,{}^gH)}\mathcal{M}(g\bV)\cong \mathcal{Q}_n(g\bV,{}^gH) \otimes_{\mathcal{Q}_n(g\bU,{}^gH)}\mathcal{M}_n(g\bU).\]
By taking a projective resolution of $\mathcal{M}_n(\bU)$ by free $\mathcal{Q}_n(\bU,H)$-modules together with the flatness of the morphisms $\mathcal{Q}_n(\bU,H)\rightarrow \mathcal{Q}_n(\bV,H)$ and $\mathcal{Q}_n(g\bU,{}^gH)\rightarrow \mathcal{Q}_n(g\bV,{}^gH)$ (Proposition \ref{flat}), it reduces to show the assertion for $i=0$, which means that the diagram
\[
  \setlength{\arraycolsep}{2pt}
  \begin{array}{*{9}c}
        Hom_{\mathcal{Q}_n(\bU,H)}(\mathcal{M}_n(\bU), \mathcal{Q}_n(\bU,H)) & \Lrightarrow & Hom_{\mathcal{Q}_n(g\bU,{}^gH)}(\mathcal{M}_n(g\bU), \mathcal{Q}_n(g\bU,{}^gH))
    \\
    \Ldownarrow & & \Ldownarrow & & 
    \\
    Hom_{\mathcal{Q}_n(\bV,H)}(\mathcal{M}_n(\bV), \mathcal{Q}_n(\bV,H)) & \Lrightarrow & Hom_{\mathcal{Q}_n(g\bV,{}^gH)}(\mathcal{M}_n(g\bV), \mathcal{Q}_n(g\bV,{}^gH))
  \end{array}
\]
is commutative. \\
Let $f : \mathcal{M}_n(\bU) \rightarrow \mathcal{Q}_n(\bU,H)$ be a $\mathcal{Q}_n(\bU,H)$-linear morphism and write $r_{n}$, $r_{n}^g$ for the restrictions $\mathcal{Q}_n(\bU,H)\rightarrow \mathcal{Q}_n(\bV,H)$ and $\mathcal{Q}_n(g\bU,{}^gH)\rightarrow \mathcal{Q}_n(g\bV,{}^gH)$, respectively. For $ a \in \mathcal{Q}_n(g\bV,{}^gH)$, $m \in \mathcal{M}_n(g\bU)$, we have:
\begin{center}
$\left( 1\bar{\otimes}r_{n}^g\circ (g^{\mathcal{Q}_n}_{\bU,H}\circ f \circ (g^{\mathcal{M}_n}_{\bU,H})^{-1})\right)(a \otimes m)=a.r_{n}^g(g^{\mathcal{Q}_n}_{\bU,H}( f( (g^{\mathcal{M}_n}_{\bU,H})^{-1}(m)))) $
\end{center}
and
\begin{center}
$\left(g^{\mathcal{Q}_n}_{\bV,H} \circ (1\bar{\otimes} r_n \circ f) \circ (g^{\mathcal{M}_n}_{\bV,H})^{-1}\right)(a \otimes m)= a.g^{\mathcal{Q}_n}_{\bV,H}(r_n(f((g^{\mathcal{M}_n}_{\bV,H})^{-1}(m)))).$
\end{center}
So it reduces to prove that for any $b \in \mathcal{Q}_n(g\bV,{}^gH)$, one has that:
\[r_n^g(g^{\mathcal{Q}_n}_{\bU,H}((b))=g^{\mathcal{Q}_n}_{\bV,H}(r_n(b)),\]
which is a consequence of  \cite[Definition $3.4.9(c)$ and Proposition $3.4.10$]{AW}.
\end{itemize}
\end{proof}

\textbf{Notation:}  In the sequel, whenever $\bV, H$ are given and whenever there is no ambiguity, we simply write $\wideparen{g}$ for $\wideparen{g}_{\bV,H}$ and $g^{\mathcal{M}_n}$, $g^{\mathcal{Q}_n}, g^{E^i(\mathcal{M})}_n$ etc. instead of $g^{\mathcal{M}_n}_{\bV,H}$, $g^{\mathcal{Q}_n}_{\bV,H}$, $g^{E^i(\mathcal{M})}_{\bV,H,n}$ etc. respectively.\\

\begin{pro} The following diagram is commutative:
\[
  \setlength{\arraycolsep}{2pt}
  \begin{array}{*{9}c}
        Ext_{\wideparen{\mathcal{D}}(\bU,H)}^i(\mathcal{M}(\bU), \wideparen{\mathcal{D}}(\bU,H)) & \xrightarrow{g_{}^{E^i(\mathcal{M})}} & Ext_{\wideparen{\mathcal{D}}(g\bU,{}^gH)}^i(\mathcal{M}(g\bU), \wideparen{\mathcal{D}}(g\bU,{}^gH))
    \\
    \Ldownarrow & & \Ldownarrow & & 
    \\
    Ext_{\mathcal{Q}_n(\bU,H)}^i(\mathcal{M}_n(\bU), \mathcal{Q}_n(\bU,H)) & \xrightarrow{g_{n}^{E^i(\mathcal{M})}} & Ext_{\mathcal{Q}_n(g\bU,{}^gH)}^i(\mathcal{M}_n(g\bU), \mathcal{Q}_n(g\bU,{}^gH)).
  \end{array}
\]
\end{pro}
 
 \begin{proof}
  First, we note that the morphisms 
  \begin{center}
  $q_{\bU,n}: \wideparen{\mathcal{D}}(\bU,H)\longrightarrow \mathcal{Q}_n(\bU,H)$ and $q_{g\bU,n}: \wideparen{\mathcal{D}}(g\bU,{}^gH)\longrightarrow \mathcal{Q}_n(g\bU,{}^gH)$ 
  \end{center}
  are flat. By using a resolution $P^. \longrightarrow \mathcal{M}(\bU) \longrightarrow 0$ of $\mathcal{M}(\bU)$ by free $\wideparen{\mathcal{D}}(\bU,H)$-modules, it reduces to show the commutativity of the above diagram for the case where $i=0$. Let $f \in Hom_{\wideparen{\mathcal{D}}(\bU,H)}(\mathcal{M}(\bU),\wideparen{\mathcal{D}}(\bU,H))$, then by definition:
 \begin{center}
 $g^{E^0(\mathcal{M})}_{}(f)= \wideparen{g}\circ f \circ \left(g^\mathcal{M}\right)^{-1}. $
 \end{center}
 Let $s \in \mathcal{Q}_n(g\bU,{}^gH)$ and $m \in \mathcal{M}_n(g\bU)$. It follows that 
 \begin{center}
 $\left(id \bar{\otimes}q_{g\bU,n}(\wideparen{g}\circ f\circ \left(g^{\mathcal{M}}\right)^{-1}\right)(s \otimes m)= s.q_{g\bU,n}(\wideparen{g}(f((g^\mathcal{M})^{-1}(m))))$
 \end{center}
 and
 \begin{center}
 $\left(g^{\mathcal{Q}_n}\circ (1\bar{\otimes}q_{\bU,n}\circ f )\circ (g^{\mathcal{M}_n})^{-1}\right)(s \otimes m)=s.g^{\mathcal{Q}_n}(q_{\bU,n}(f((g^{\mathcal{M}})^{-1}(m)))).$
 \end{center}
 Now the result follows from the commutativity of the diagram
\[
 \begin{tikzcd}
\wideparen{\mathcal{D}}(\bU,H) \arrow{r}{\wideparen{g}} \arrow[swap]{d}{q_{\bU,n}} & \wideparen{\mathcal{D}}(g\bU,{}^gH)\arrow{d}{q_{g\bU,n}} \\
\mathcal{Q}_n(\bU,H) \arrow{r}{{g}^{\mathcal{Q}_n}} & \mathcal{Q}_n(g\bU,{}^gH) \\
\end{tikzcd}
\]
which is evident as $\wideparen{g}=\varprojlim_n g^{\mathcal{Q}_n}$.
 \end{proof}
 
\begin{re} The above proposition shows that for any $g \in G$ , $\bU \in X_{w}(\mathcal{T})$ and $H\leq G$ such that $(\bU,H)$ is small, the following equality holds:  
\[g^{E^i(\mathcal{M})}_{\bU,H} = \varprojlim_n g^{E^i(\mathcal{M})}_{\bU,H,n}.\]
\end{re} 
 
\begin{pro} \label{comm3}
If $N \leq H$ and $\bV$ is a $N$-stable subdomain of $\bU$ in $\bX_{w}(\mathcal{T})$, the diagram
\[
  \setlength{\arraycolsep}{2pt}
  \begin{array}{*{9}c}
    E^i(\mathcal{M})(\bV,N) & \xrightarrow{g^{E^i(\mathcal{M})}_{\bV,N}} & E^i(\mathcal{M})(g\bV,gNg^{-1}) 
    \\
    \Ldownarrow{\wideparen{p}^i_{N,H}(\bV)} & & \Ldownarrow{\wideparen{p}^i_{N,H}(\bV)} & & 
    \\
       E^i(\mathcal{M})(\bV,H) & \xrightarrow{g^{E^i(\mathcal{M})}_{\bV,H}} & E^i(\mathcal{M})(g\bV,gHg^{-1})
    \\
    \Big\uparrow{\wideparen{\tau}^i_{\bU,\bV,H}} & & \Big\uparrow{\wideparen{\tau}^i_{\bU,\bV,H}} & & 
    \\
    E^i(\mathcal{M})(\bU,H) & \xrightarrow{g^{E^i(\mathcal{M})}_{\bU,H}} & E^i(\mathcal{M})(g\bU,gHg^{-1})
  \end{array}
\]
is commutative.
\end{pro}
\begin{proof} It is enough to prove the proposition for the case where $N$  is normal in $H$, as the general case can be proved by choosing an open normal subgroup of $H$ which is contained in $N$.
\begin{itemize}
\item[$1.$] 
Let us prove the commutativity of the upper square.  Take a projective resolution of $\mathcal{M}(\bV)$ by free modules in $Mod(\wideparen{\mathcal{D}}(\bV,H))$. It is enough to show that for any (left) $\wideparen{\mathcal{D}}(\bV,H)$-module $P$,  the diagram
 \[\begin{tikzcd}
Hom_{\wideparen{\mathcal{D}}(\bV,H)}(P, \wideparen{\mathcal{D}}(\bV,H)) \arrow{r}{\phi_{\bV,H}^g} \arrow[swap]{d}{\wideparen{p}_{H,N}(\bV)} & Hom_{\wideparen{\mathcal{D}}(g\bV,{}^gH)}({}^gP, \wideparen{\mathcal{D}}(g\bV,{}^gH))\arrow{d}{\wideparen{p}_{{}^{g}H,{}^gN}(g\bV)} \\
Hom_{\wideparen{\mathcal{D}}(\bV,N)}(P, \wideparen{\mathcal{D}}(\bV,N))\arrow{r}{\phi_{\bV,N}^g} & Hom_{\wideparen{\mathcal{D}}(g\bV,gNg^{-1})}({}^gP, \wideparen{\mathcal{D}}(g\bV,{}^gN))\\
\end{tikzcd}
\]
is commutative. It means that if $f \in Hom_{\wideparen{\mathcal{D}}(\bV,H)}(P, \wideparen{\mathcal{D}}(\bV,H))$, then one has : 
$$\wideparen{p}_{{}^gH,{}^gN}(g\bV)( \wideparen{g}_{\bV,H} \circ f)= \wideparen{g}_{\bV,N}\circ \wideparen{p}_{H,N}(\bV) (f). $$
But this reduces to proving that the diagram 
\begin{equation}\label{diag3}
\begin{tikzcd}
\wideparen{\mathcal{D}}(\bV,H) \arrow{r}{\wideparen{g}_{\bV,H}} \arrow[swap]{d}{p_{H,N}^{\bV}} & \wideparen{\mathcal{D}}(g\bV,{}^gH)\arrow{d}{p_{{}^gH,{}^gN}^{\bV}} \\
\wideparen{\mathcal{D}}(\bV,N) \arrow{r}{\wideparen{g}_{\bV,N}} & \wideparen{\mathcal{D}}(g\bV,{}^gN) \\
\end{tikzcd}
\end{equation}
is commutative. For this, choose a $H$-stable free $\mathcal{A}$-Lie lattice $\mathcal{L}$ for some $H$-stable formal model $\mathcal{A}$ of $\mathcal{O}(\bV)$ and a good chain $(J_n)$ for $\mathcal{L}$. Recall from Lemma \ref{Glemma}(i) that $\mathcal{L'}={g}^\mathcal{T}(\mathcal{L})$ is a $gHg^{-1}$-stable free $g(\mathcal{A})$-Lie lattice in $\mathcal{T}(g\bU)$. For  a fixed natural integer $n \in \mathbb{N}$, we consider the following diagram:
\[\begin{tikzcd}
\widehat{U(\pi^n\mathcal{L})}_K \rtimes_{J_n}H \arrow{r}{g_{\bV,H}^{\mathcal{Q}_n}} \arrow[swap]{d}{p_{H,N,n}^{\bV}} & \widehat{U(\pi^{n}\mathcal{L'})}_K \rtimes_{gJ_ng^{-1}}gHg^{-1} \arrow{d}{p_{{}^gH,{}^gN,n}^{\bV}} \\
\widehat{U(\pi^n\mathcal{L})}_K \rtimes_{J_n}N \arrow{r}{g_{\bV,N}^{\mathcal{Q}_n}} & \widehat{U(\pi^n\mathcal{L'})}_K \rtimes_{gJ_ng^{-1}}gNg^{-1}. \\
\end{tikzcd}\]
Let  $\lbrace g_1=1,...,g_m,...,g_n\rbrace$  be a set of representatives of cosets of $G$ modulo $J_n$ such that $\lbrace \bar{g_1}=1,\bar{g_2},\bar{g_m},.., \bar{g}_n\rbrace$ is a basis of $\widehat{U(\pi^n\mathcal{L})}_K \rtimes_{J_n}H$ and $\lbrace \bar{g}_1,..., \bar{g}_m \rbrace$ is a basis of $\widehat{U(\pi^n\mathcal{L})}_K \rtimes_{J_n}N$ over the ring $\widehat{U(\pi^n\mathcal{L})_K}$ (as  left modules). Then we get a basis of $\widehat{U(\pi^{n}\mathcal{L'})}_K \rtimes_{gJ_ng^{-1}}gHg^{-1} $ (respectively, of $\widehat{U(\pi^n\mathcal{L'})}_K \rtimes_{J_n}gNg^{-1}$) over the ring $\widehat{U(\pi^n\mathcal{L'})_K}$ induced by classes of  $\lbrace gg_1g^{-1},...,gg_mg^{-1},...,g g_n g^{-1}\rbrace$ (respectively, of $\lbrace g g_1g^{-1},...,gg_mg^{-1}\rbrace $ )  modulo $gJ_ng^{-1}$. This implies, by definition of the projection maps $p_{H,N,n}^{\bV}$ and $p_{{}^gH,{}^gN,n}^{g\bV}$, that the above diagram is commutative for each $n$, which produces the commutativity of (\ref{diag3}).

\item[$2.$] It remains to show the commutativity of the lower square. We still fix a $H$-stable free $\mathcal{A}$-Lie lattice of $\mathcal{T}(\bU)$, a good chain $(J_n)$ for $\mathcal{L}$ and keep notations as above. Suppose in addition that $\bV$ is an $\mathcal{L}$- accessible subdomain of $\bU$ (by rescaling $\mathcal{L}$). Then $g\bV$ is an $\mathcal{L'}$- accessible subdomain of $g\bU$ by Lemma \ref{Glemma}(ii). 
Now, since all morphisms of the lower square are linear maps between coadmissible modules, it is enough to show that the diagram
\[\begin{tikzcd}
Ext^i_{\mathcal{Q}_n(\bU,H)}(\mathcal{M}_n(\bU),\mathcal{Q}_n(\bU,H)) \arrow{r}{} \arrow[swap]{d}{} & Ext^i_{\mathcal{Q}_n(g\bU,{}^gN)}\mathcal{M}_n(g\bU),\mathcal{Q}_n(gU,{}^gH))\arrow{d}{} \\
Ext^i_{\mathcal{Q}_n(\bV,H)}(\mathcal{M}_n(\bV),\mathcal{Q}_{n}(\bV,H)) \arrow{r}{} & Ext^i_{\mathcal{Q}_n(g\bV,{}^gH)},(\mathcal{M}_n(g\bV),\mathcal{Q}_{n}(g\bV, {}^gH))
\end{tikzcd}\]
is commutative. This is indeed Proposition \ref{Giso}(2).
\end{itemize}
\end{proof}
All our efforts so far culminate in the following theorem.
\begin{theorem}\label{thm1}
Let $\bX$ be a smooth rigid analytic space and $G$ be a $p$-adic Lie group acting continuously on $\bX$. Let $\mathcal{M} \in \mathcal{C}_{\bX/G}$, then for all $i \geq 0$, $E^i(\mathcal{M})$ is a $G$-equivariant presheaf of right $\mathcal{D}_{\bX}$-modules on $\bX_{w}(\mathcal{T})$.
\end{theorem}
\begin{proof}
Let $\bW \subset \bV \subset \bU$ be affinoid subdomains of $\bX$ in $\bX_{w}(\mathcal{T})$. By \cite[Lemma $3.4.7$]{AW} there exists an open compact subgroup $H \leq G$ such that the pairs $(\bW,H), (\bV,H), (\bU,H)$ are all small.  We consider the following diagram:
\begin{center}
\begin{tikzpicture}[thick]
  [scale=.15,auto=left,every node/.style={circle,fill=white!25}]
  \node (n1) at (0,0) {$E^i(\mathcal{M})(\bU,H)$};
  \node (n2) at (8,0)  {$E^i(\mathcal{M})(\bV,H)$};
  \node (n3) at (4,5)  {$E^i(\mathcal{M})(\bW,H)$};
  \node (n4) at (2,1) {$E^i(\mathcal{M})(\bU)$};
  \node (n5) at (6,1)  {$E^i(\mathcal{M})(\bV)$};
  \node (n6) at (4,3)  {$E^i(\mathcal{M})(\bW)$};

  \draw [->] (n1) -- (n2);
  \draw [->] (n1) -- (n3);
  \draw [->] (n3) -- (n2);

  \draw [->] (n4) -- (n5);
  \draw [->] (n4) -- (n6);
  \draw [->] (n5) -- (n6);

  \draw [->] (n4) -- (n1);
  \draw [->] (n5) -- (n2);
  \draw [->] (n6) -- (n3);
  
\end{tikzpicture}
\end{center}
The three quadrilaterals are commutative by definition. The outer triangle is commutative by Proposition \ref{res2} and the three arrows connecting the two triangles are bijections by Remark \ref{proj}. Hence the inner triangle is commutative and this proves that $E^i (\mathcal{M})$ is a presheaf.\\
\\
Next, fix $g \in G$ and  $\bU \in \bX_{w}(\mathcal{T})$. We define 
\[g^{E^i(\mathcal{M})}(\bU):\,\, E^i(\mathcal{M})(\bU)\longrightarrow E^i(\mathcal{M})(g\bU)\]
to be the inverse limit of the maps $g^{E^i(\mathcal{M})}_{\bU,H}$ in Proposition \ref{Giso1}. Then
\item[$\star$] By  (\ref{equationg}) (Proposition \ref{Giso1}), it is straightforward to see that $ g^{E^i(\mathcal{M})}(m.a)= g^{E^i(\mathcal{M})}(m).{g}^\mathcal{D}(a) $ for any $ a \in \mathcal{D}(\bU)$ and $m \in E^i(\mathcal{M})(\bU)$.\\
\item[$\star$] Assume that $\bV \subset \bU$ are in $\bX_{w}(\mathcal{T})$. Let $H$ be a $\bU$-small subgroup  of $G_{\bU} \cap G_{\bV}$. We consider the following  diagram:
\begin{center}
\begin{tikzpicture}[thick]
  [scale=.15,auto=left,every node/.style={circle,fill=white!25}]
  \node (n1) at (0,7) {$E^i(\mathcal{M})(\bU,H)$};
  \node (n5) at (2,5)  {$E^i(\mathcal{M})(\bU)$};
  \node (n6) at (5,5)  {$E^i(\mathcal{M})(g\bU)$};
  \node (n2) at (7,7) {$E^i(\mathcal{M})(g\bU,{}^gH)$};
  \node (n7) at (5,2)  {$E^i(\mathcal{M})(g\bV)$};
  \node (n8) at (2,2)  {$E^i(\mathcal{M})(\bV)$};
  \node (n4) at (0,0)  {$E^i(\mathcal{M})(\bV,H)$};
  \node (n3) at (7,0)  {$E^i(\mathcal{M})(g\bV,{}^gH)$};
    
  \draw [->] (n1) -- (n2);
  \draw [->] (n2) -- (n3);
  \draw [->] (n4) -- (n3);
  \draw [->] (n1) -- (n4);

  \draw [->] (n5) -- (n1);
  \draw [->] (n6) -- (n2);
  \draw [->] (n8) -- (n4);
  \draw [->] (n7) -- (n3);

  \draw [->] (n5) -- (n6);
  \draw [->] (n6) -- (n7);
  \draw [->] (n5) -- (n8);
  \draw [->] (n8) -- (n7);
\end{tikzpicture}
\end{center}
Note that the outer square is commutative by Proposition \ref{comm3}, the four trapezia are commutative by definition and the arrows connecting the two squares are bijections. This proves that the inner square is commutative. Hence $g^{E^i(\mathcal{M})} : \,\,  E^i( \mathcal{M}) \longrightarrow g^\ast (E^i (\mathcal{M}))$ is a morphism of presheaves on $\bX_{w}(\mathcal{T})$.
\item[$\star$] Finally, if $g, h \in G$, we need to show that $(gh)^{E^i(\mathcal{M})}=g^{E^i(\mathcal{M})}\circ h^{E^i(\mathcal{M})}$. By taking a free resolution of $\mathcal{M}(\bU)$ by free $\wideparen{\mathcal{D}}(\bU,H)$-modules, it is enough to show that for any $\wideparen{\mathcal{D}}(\bU,H)$-module $P$, the diagram 
\[\begin{tikzcd}
Hom_{\wideparen{\mathcal{D}}(\bU,H)}(P, \wideparen{\mathcal{D}}(\bU,H)) \arrow{r}{\phi_{\bU,H}^{h}}  \arrow{rd}{\phi_{\bU,H}^{gh}} 
  & Hom_{\wideparen{\mathcal{D}}(h\bU,hHh^{-1})}({}^hP,\wideparen{\mathcal{D}}(h\bU,hHh^{-1})) \arrow{d}{\phi_{h\bU,{}^hH}^{g}} \\ 
    & Hom_{\wideparen{\mathcal{D}}(gh\bU,ghHh^{-1}g^{-1})}({}^{gh}P,\wideparen{\mathcal{D}}(gh\bU,ghHh^{-1}g^{-1})) 
\end{tikzcd}\]
is commutative. Let $f \in Hom_{\wideparen{\mathcal{D}}(\bU,H)}(P, \wideparen{\mathcal{D}}(\bU,H))$, then 
\begin{center}
$\phi_{h\bU,{}^hH}^g \circ \phi_{\bU,H}^h(f)= \phi_{h\bU,{}^hH}^g( \wideparen{h}_{\bU,H}\circ f)= \wideparen{g}_{h\bU,{}^hH} \circ \wideparen{h}_{\bU,H}\circ f $ 
\end{center}
 while $\phi^{gh}_{\bU,H}= \wideparen{gh}_{\bU,H} \circ f$. Hence the commutativity of the diagram follows from the equality $\wideparen{gh}_{\bU,H}= \wideparen{g}_{g\bU, {}^hH} \circ \wideparen{h}_{\bU,H}$, which is from \cite[Lemma $3.4.3$]{AW}.

\end{proof}
\subsection{The Ext functors on the category $\mathcal{C}_{\bX/G}$} 
We keep all the notation of the preceding section. 
We now show that 
for any  $\mathcal{M} \in \mathcal{C}_{\bX/G}$ and any $i\geq 0$, the presheaf $E^i(\mathcal{M})$ on $\bX_{w}(\mathcal{T})$ constructed in the previous subsection is in fact a sheaf and extends therefore to a $G$-equivariant $\mathcal{D}_{\bX}$-module on $\bX$. It then turns out that it actually defines an object in $\mathcal{C}^r_{\bX/G}.$
\\

We first assume that $(\bX,G)$ is small and let $\mathcal{M} \in \mathcal{C}_{\bX/G}$ be a sheaf of  coadmissible $G$-equivariant left $\mathcal{D}_{\bX}$-modules. 
\begin{lemma}\label{l4} Let $\bU \in \bX_{w}(\mathcal{T})$ and $H$ be a $\bU$-small subgroup of $G$. Then there is an isomorphism of right $\wideparen{\mathcal{D}}(\bU,H)$-modules
\[\Phi^i_{\bU,H}: \,\,Ext^i_{\wideparen{\mathcal{D}}(\bX,G)}(\mathcal{M}(\bX), \wideparen{\mathcal{D}}(\bX,G))\wideparen{\otimes}_{\wideparen{\mathcal{D}}(\bX,H)}\wideparen{\mathcal{D}}(\bU,H)  \tilde{\longrightarrow} Ext^i_{\wideparen{\mathcal{D}}(\bU,H)}(\mathcal{M}(\bU), \wideparen{\mathcal{D}}(\bU,H)).\]
\end{lemma}
\begin{proof} Recall that $\mathcal{M} \cong Loc_{\bX}(\mathcal{M}(\bX))$, so that 
$$\mathcal{M}(\bU)\simeq \wideparen{\mathcal{D}}(\bU,H) \wideparen{\otimes}_{ \wideparen{\mathcal{D}}(\bX,H)} \mathcal{M}(\bX). $$
By applying Proposition \ref{isoD}, we obtain an isomorphism of right $\wideparen{\mathcal{D}}(\bX,H)$-modules
\begin{equation}\label{equation1}
\wideparen{p}^i_{G,H}(\bX): Ext^i_{\wideparen{\mathcal{D}}(\bX,G)}(\mathcal{M}(\bX), \wideparen{\mathcal{D}}(\bX,G))\tilde{\longrightarrow} Ext^i_{\wideparen{\mathcal{D}}(\bX,H)}(\mathcal{M}(\bX), \wideparen{\mathcal{D}}(\bX,H)).
\end{equation}
Hence
\begin{equation}\label{equation2}
Ext^i_{\wideparen{\mathcal{D}}(\bX,G)}(\mathcal{M}(\bX), \wideparen{\mathcal{D}}(\bX,G))\wideparen{\otimes}_{\wideparen{\mathcal{D}}(\bX,H)}\wideparen{\mathcal{D}}(\bU,H) \tilde{\longrightarrow} Ext^i_{\wideparen{\mathcal{D}}(\bX,H)}(\mathcal{M}(\bX), \wideparen{\mathcal{D}}(\bX,H))\wideparen{\otimes}_{\wideparen{\mathcal{D}}(\bX,H)}\wideparen{\mathcal{D}}(\bU,H).
\end{equation}
Finally, apply Lemma \ref{lem} gives:
\begin{equation}\label{equation3}
Ext^i_{\wideparen{\mathcal{D}}(\bX,H)}(\mathcal{M}(\bX), \wideparen{\mathcal{D}}(\bX,H))\wideparen{\otimes}_{\wideparen{\mathcal{D}}(\bX,H)}\wideparen{\mathcal{D}}(\bU,H) \tilde{\longrightarrow} Ext^i_{\wideparen{\mathcal{D}}(\bU,H)}(\mathcal{M}(\bU), \wideparen{\mathcal{D}}(\bU,H)).
\end{equation}
\end{proof}
 Let us explain how the isomorphism $\Phi^0_{\bU,H}$ looks like, when $H$ is an open normal subgroup in $G$. Write $\Phi_{\bU,H}:= \Phi^0_{\bU,H}$. Then
\[\Phi_{\bU,H}: \,\,Hom_{\wideparen{\mathcal{D}}(\bX,G)}(\mathcal{M}(\bX), \wideparen{\mathcal{D}}(\bX,G))\wideparen{\otimes}_{\wideparen{\mathcal{D}}(\bX,H)}\wideparen{\mathcal{D}}(\bU,H)  \tilde{\longrightarrow} Hom_{\wideparen{\mathcal{D}}(\bU,H)}(\mathcal{M}(\bU), \wideparen{\mathcal{D}}(\bU,H)).\]
Let us choose a $G$-stable Lie lattice $\mathcal{L}$ of $\mathcal{T}(\bX)$ such that $\bU$ is $\mathcal{L}$-accessible. Let $(J_n)$ be a good chain for $\mathcal{L}$.  Then we can take the sheaves $\mathcal{Q}_n$ into account and obtain that:
\begin{center}
 $\wideparen{\mathcal{D}}(\bX,G)= \varprojlim_n\mathcal{Q}_n(\bX,G)$, $\wideparen{\mathcal{D}}(\bX,H)= \varprojlim_n\mathcal{Q}_n(\bX,H)$ and $\wideparen{\mathcal{D}}(\bU,H)= \varprojlim_n\mathcal{Q}_n(\bU,H).$
 \end{center} 
 Write $M:= \mathcal{M}(\bX)\cong \varprojlim_n M_n$. Then $\mathcal{M}_n(\bU)= \mathcal{Q}_n(\bU,H)\otimes_{\mathcal{Q}_n(\bX,H)}M_n$ .The morphism $\Phi_{\bU,H}$ is defined as the inverse limit of an inverse system $(\Phi_{\bU,H,n})_n$ of morphisms, where
\[\Phi_{\bU,H,n}: Hom_{\mathcal{Q}_n(\bX,G)}(M_n, \mathcal{Q}_n(\bX,G)){\otimes}_{\mathcal{Q}_n(\bX,H)}\mathcal{Q}_n(\bU,H)  \tilde{\longrightarrow} Hom_{\mathcal{Q}_n(\bU,H)}(\mathcal{M}_n(\bU), \mathcal{Q}_n(\bU,H))\]
  is defined as follows. If $f_n: M_n \longrightarrow \mathcal{Q}_n(\bX,G)$ is a $\mathcal{Q}_n(\bX,G)$-linear morphism and $ a \in \mathcal{Q}_n(\bU,H)$, then applying (\ref{equation1}), we obtain the $\mathcal{Q}_n(\bX,H)$-linear morphism $$p_{G,H,n}^{\bX} \circ f_n: \,\, M_n \longrightarrow \mathcal{Q}_n(\bX,H),$$ 
  
where $p_{G,H,n}^{\bX}$ is defined in (\ref{projection1}). Next, $(p_{G,H,n}^{\bX} \circ f_n) \otimes a$ is the image of $f_n \otimes a$ via the isomorphism (\ref{equation2}). Finally, by applying the isomorphism (\ref{equation3}), we get the map 
\begin{align*}
1 \bar{\otimes} ((p_{G,H,n}^{\bX} \circ f_n).a): \,\, \mathcal{Q}_n(\bU,H)\otimes M_n &\longrightarrow \mathcal{Q}_n(\bU,H)\\
 b \otimes m &\longmapsto b. p_{G,H}(f_n(m)).a.
\end{align*}
Note that in the above formula, we identify  \,$p_{G,H,n}^{\bX} (f_n(m))\in \mathcal{Q}_n(\bX,H)$ with its image in $\mathcal{Q}_n(\bU,H)$ via the canonical morphism $\mathcal{Q}_n(\bX,H) \longrightarrow \mathcal{Q}_n(\bU,H)$.  Therefore
\begin{equation}\label{equality}
\Phi_{U,H,n}(f_n)= id \bar{\otimes} ((p_{G,H,n}^{\bX} \circ f_n).a ) \in Hom_{\mathcal{Q}_n(\bU,H)}(\mathcal{M}_n(\bU), \mathcal{Q}_n(\bU,H)).
\end{equation}
\\
Recall that ${}^rLoc_{\bX}(-)$ denotes the localisation functor on the category $\mathcal{C}_{\wideparen{\mathcal{D}}(\bX,G)}^r$ of coadmissible right $\wideparen{\mathcal{D}}(\bX,G)$-modules.
\begin{pro}\label{iso2}  Suppose that $(\bX,G)$ is small. There is an isomorphism of presheaves of right $\mathcal{D}_{\bX}$-modules on $\bX_{w}(\mathcal{T})$
\[\Phi: \,\, {}^rLoc_{\bX}( Ext^i_{\wideparen{\mathcal{D}}(\bX,G)}(\mathcal{M}(\bX), \wideparen{\mathcal{D}}(\bX,G)) \tilde{\longrightarrow} E^i(\mathcal{M}).\]
\end{pro}
\begin{proof}
Write $M:=\mathcal{M}(\bX)$ and fix an open affinoid subset $\bU \in \bX_{w}(\mathcal{T})$. By Lemma \ref{l4}, for any $\bU$-small subgroup $H$ of $G$, there is an isomorphism of right $ \wideparen{\mathcal{D}}(\bU,H)$-modules
\[\Phi^i_{U,H}: \,\, Ext^i_{\wideparen{\mathcal{D}}(\bX,G)}(\mathcal{M}(\bX), \wideparen{\mathcal{D}}(\bX,G))\wideparen{\otimes}_{\wideparen{\mathcal{D}}(\bX,H)}\wideparen{\mathcal{D}}(\bU,H)  \tilde{\longrightarrow} Ext^i_{\wideparen{\mathcal{D}}(\bU,H)}(\mathcal{M}(\bU), \wideparen{\mathcal{D}}(\bU,H)).\]
If  $H' \leq H$ is another $\bU$-small subgroup of $G$, we need to show that 
\begin{equation}\label{iso1}
\begin{tikzcd}
Ext^i_{\wideparen{\mathcal{D}}(\bX,G)}(M, \wideparen{\mathcal{D}}(\bX,G))\wideparen{\otimes}_{\wideparen{\mathcal{D}}(X,H')}\wideparen{\mathcal{D}}(\bU,H') \arrow{r}{\Phi^i_{\bU,H'}} \arrow[swap]{d}{} & Ext^i_{\wideparen{\mathcal{D}}(\bU,H')}(\mathcal{M}(\bU), \wideparen{\mathcal{D}}(\bU,H'))\arrow{d}{\wideparen{p}^i_{H',H}} \\
Ext^i_{\wideparen{\mathcal{D}}(\bX,G)}(M, \wideparen{\mathcal{D}}(\bX,G))\wideparen{\otimes}_{\wideparen{\mathcal{D}}(\bX,H)}\wideparen{\mathcal{D}}(\bU,H)\arrow{r}{\Phi^i_{\bU,H}} & Ext^i_{\wideparen{\mathcal{D}}(\bU,H)}(\mathcal{M}(\bU), \wideparen{\mathcal{D}}(\bU,H))\\
\end{tikzcd}
\end{equation}
is commutative. 
It suffices to assume that $H'$ and $H$ are normal in $G$. Then $\wideparen{p}^i_{H'H}(\bU)$ is the inverse of the map $\wideparen{p}^i_{H,H'}(\bU)$ (which is defined in Proposition \ref{isoD}). So it is equivalent to show that the diagram
\begin{center}
\begin{tikzpicture}[thick]
  [scale=.15,auto=left,every node/.style={circle,fill=white!25}]
  \node (n1) at (0,0) {$Ext^i_{\wideparen{\mathcal{D}}(\bX,G)}(M, \wideparen{\mathcal{D}}(\bX,G))\wideparen{\otimes}_{\wideparen{\mathcal{D}}(\bX,H)}\wideparen{\mathcal{D}}(\bU,H)$};
  \node (n2) at (8,0)  {$Ext^i_{\wideparen{\mathcal{D}}(\bU,H)}(\mathcal{M}(\bU), \wideparen{\mathcal{D}}(\bU,H))$};
  \node (n3) at (0,2)  {$Ext^i_{\wideparen{\mathcal{D}}(\bX,G)}(M, \wideparen{\mathcal{D}}(\bX,G))\wideparen{\otimes}_{\wideparen{\mathcal{D}}(\bX,H')}\wideparen{\mathcal{D}}(\bU,H')$};
  \node (n4) at (8,2) {$Ext^i_{\wideparen{\mathcal{D}}(\bU,H')}(\mathcal{M}(\bU), \wideparen{\mathcal{D}}(\bU,H'))$};
  
  \draw [->] (n1) -- (n2);
  \draw [->] (n3) -- (n1);
  \draw [->] (n3) -- (n4);
  \draw [->] (n2) -- (n4);
\end{tikzpicture}
\end{center}
is commutative.\\
\\
Fix a $G$-stable free $\mathcal{A}$-Lie lattice $\mathcal{L}$ in $\mathcal{T}(\bX)$ for some $G$-stable affine formal model $\mathcal{A}$ of $\mathcal{O}(\bX)$ and a good chain $(J_n)$ for $\mathcal{L}$. By rescaling $\mathcal{L}$ if necessary, we may suppose that $\bU$ is $\mathcal{L}$-accessible \cite[Prop. 7.6]{AWI}. Recall the sheaves $\mathcal{Q}_n$ and $\mathcal{M}_n$ in (\ref{Q1}), (\ref{M1}), and (\ref{M2}).
Then
\begin{center}
$\wideparen{\mathcal{D}}(\bX,G)= \varprojlim_n \mathcal{Q}_n(\bX,G)$, $\wideparen{\mathcal{D}}(\bU,H)= \varprojlim_n \mathcal{Q}_n(\bU,H)$ and $\wideparen{\mathcal{D}}(\bU,H')= \varprojlim_n \mathcal{Q}_n(\bU,H')$.
\end{center}
Thus $M \cong \varprojlim_n M_n$, with $M_n:= \mathcal{Q}_n(\bX,G)\otimes_{\wideparen{\mathcal{D}}(\bX,G)}M.$
Since the morphisms in the above square are linear between coadmissible modules, it is enough to prove that the diagram
\begin{center}
\begin{tikzpicture}[thick]
  [scale=.15,auto=left,every node/.style={circle,fill=white!25}]
  \node (n1) at (0,0) {$Ext^i_{\mathcal{Q}_n(\bX,G)}({M}_n, \mathcal{Q}_n(\bX,G))\otimes_{\mathcal{Q}_n(\bX,H)}\mathcal{Q}_n(\bU,H)$};
  \node (n2) at (8,0)  {$Ext^i_{\mathcal{Q}_n(\bU,H)}(\mathcal{M}_n(\bU), \mathcal{Q}_n(\bU,H))$};
  \node (n3) at (0,2)  {$Ext^i_{\mathcal{Q}_n(\bX,G)}({M}_n, \mathcal{Q}_n(\bX,G))\otimes_{\mathcal{Q}_n(\bX,H')}\mathcal{Q}_n(\bU,H')$};
  \node (n4) at (8,2) {$Ext^i_{\mathcal{Q}_n(\bU,H')}(\mathcal{M}_n(\bU), \mathcal{Q}_n(\bU,H'))$};
  
  \draw [->] (n1) -- (n2);
  \draw [->] (n3) -- (n1);
  \draw [->] (n3) -- (n4);
  \draw [->] (n2) -- (n4);
\end{tikzpicture}
\end{center}
is commutative. \\
Now, by taking a free resolution of $M_n$ as a $\mathcal{Q}_n(\bX,G)$-module and by using the flatness of the morphisms $\mathcal{Q}_n(\bX,H')\longrightarrow \mathcal{Q}_n(\bU,H')$ and $\mathcal{Q}_n(\bX,H)\longrightarrow \mathcal{Q}_n(\bU,H)$ (Proposition \ref{flat}), it remains to prove that, for any $\mathcal{Q}_n(\bX,G)$-module $P$, the diagram
\begin{center}
\begin{tikzpicture}[thick]
  [scale=.15,auto=left,every node/.style={circle,fill=white!25}]
  \node (n1) at (0,0) {$Hom_{\mathcal{Q}_n(\bX,G)}(P, \mathcal{Q}_n(\bX,G))\otimes\mathcal{Q}_n(\bU,H)$};
  \node (n2) at (8,0)  {$Hom_{\mathcal{Q}_n(\bU,H)}(\mathcal{Q}_n(\bU,H) \otimes P, \mathcal{Q}_n(\bU,H))$};
  \node (n3) at (0,2)  {$Hom_{\mathcal{Q}_n(\bX,G)}(P, \mathcal{Q}_n(\bX,G))\otimes\mathcal{Q}_n(\bU,H')$};
  \node (n4) at (8,2) {$Hom_{\mathcal{Q}_n(\bU,H')}(\mathcal{Q}_n(\bU,H') \otimes P, \mathcal{Q}_n(\bU,H'))$};
  
  \draw [->] (n1) -- (n2);
  \draw [->] (n3) -- (n1);
  \draw [->] (n3) -- (n4);
  \draw [->] (n2) -- (n4);
\end{tikzpicture}
\end{center}
is commutative.  \\
Let $f \in Hom_{\mathcal{Q}_n(\bX,G)}(P, \mathcal{Q}_n(\bX,G))$ and $a \in \mathcal{Q}_n(\bU,H)$, then we need to show that:
\begin{equation}\label{eq}
p_{H,H',n}^{\bU}\circ ( 1 \bar{\otimes}({p}_{G,H,n}^{\bX}\circ f )i(a)) = 1\bar{\otimes}((p_{ G,H',n}^{\bX} \circ f)a).
\end{equation}
Where, $i : \mathcal{Q}_n(\bU,H') \longrightarrow \mathcal{Q}_n(\bU,H)$ is the natural inclusion. Let $b \in \mathcal{Q}_n(\bU,H')$ and $m \in P$, then we compute by using (\ref{equality})
\begin{align*}
& p_{H,H',n}^{\bU}\circ ( 1 \bar{\otimes}({p}_{G,H,n}^{\bX}\circ f)i(a))(b \otimes m) = p_{H,H',n}^{\bU} ( b {p}_{G,H,n}^{\bX}(f(m))i(a))\\
&=bp_{H,H',n}^{\bU}({p}_{G,H,n}^{\bX}(f(m)))a =b p_{H,H',n}^{\bU} \circ p_{G,H,n}^{\bX}(f(m)) a = b p_{G,H',n}^{\bX}(f(m)) a.
\end{align*}
Thus, the equality (\ref{eq}) is proved and so the commutativity of the diagram (\ref{iso1}) follows. As a consequence of this, by taking the inverse limit of the maps $\Phi^i_{\bU,H}$, we obtain a right  $\mathcal{D}(\bU)$-linear isomorphism 
\[\Phi^i(\bU): \,\, {}^rLoc_{\bX}( Ext^i_{\wideparen{\mathcal{D}}(\bX,G)}(\mathcal{M}(\bX), \wideparen{\mathcal{D}}(\bX,G))(\bU) \tilde{\longrightarrow} E^i(\mathcal{M})(\bU).\]
Finally, $\Phi^i$  being a morphism of presheaves amounts to showing that if $\bV \subset \bU$ are open subsets in $\bX_{w}(\mathcal{T})$ and $H$ is an open normal subgroup of $G$ which stabilizes $\bU$ and $\bV$,  the following diagram is commutative:
\[\begin{tikzcd}
Ext^i_{\wideparen{\mathcal{D}}(\bX,G)}(M, \wideparen{\mathcal{D}}(\bX,G))\wideparen{\otimes}_{\wideparen{\mathcal{D}}(\bX,H)}\wideparen{\mathcal{D}}(\bU,H) \arrow{r}{\phi^i_{\bU,H}} \arrow[swap]{d}{} & Ext^i_{\wideparen{\mathcal{D}}(\bU,H)}(\mathcal{M}(\bU), \wideparen{\mathcal{D}}(\bU,H))\arrow{d}{} \\
Ext^i_{\wideparen{\mathcal{D}}(\bX,G)}(M, \wideparen{\mathcal{D}}(\bX,G))\wideparen{\otimes}_{\wideparen{\mathcal{D}}(\bX,H)}\wideparen{\mathcal{D}}(\bV,H)\arrow{r}{\phi^i_{\bV,H}} & Ext^i_{\wideparen{\mathcal{D}}(\bV,H)}(\mathcal{M}(\bV), \wideparen{\mathcal{D}}(\bV,H)).\\
\end{tikzcd}
\]
 This is indeed a consequence of Proposition \ref{res2}, where it is proved that: 
 \[Ext^i_{\wideparen{\mathcal{D}}(\bV,H)}(\mathcal{M}(\bV), \wideparen{\mathcal{D}}(\bV,H))\cong Ext^i_{\wideparen{\mathcal{D}}(\bU,H)}(\mathcal{M}(\bU), \wideparen{\mathcal{D}}(\bU,H))\wideparen{\otimes}_{\wideparen{\mathcal{D}}(\bU,H)}\wideparen{\mathcal{D}}(\bV,H).\]
\end{proof}

\begin{coro}\label{coro1}
Let $\mathcal{M} \in \mathcal{C}_{\bX/G}$ be a coadmissible $G$-equivariant $\mathcal{D}_{\bX}$-module on $\bX$. The presheaf $E^i(\mathcal{M})$ is a sheaf on the basis $\bX_{w}(\mathcal{T})$ of the Grothendieck topology on $\bX$. In particular, this can be extended to a sheaf on $\bX_{rig}$, which is still denoted by  $E^i(\mathcal{M})$. 
\end{coro}

\begin{proof}
Fix  $\bU \in \bX_{w}(\mathcal{T})$ and let $H$ be a $\bU$-small open subgroup of $G$. Then following Proposition \ref{iso2}
\[E^i(\mathcal{M})\vert_{\bU} \simeq {}^rLoc_{\bU}(Ext^i_{\wideparen{\mathcal{D}}(\bU,H)}(\mathcal{M}(\bU), \wideparen{\mathcal{D}}(\bU,H)).\]
Since the right hand side is a sheaf on $\bU_w(\mathcal{T}\vert_{\bU})$, one has that $E^i(\mathcal{M})\mid_{\bU}$ is also a sheaf on $\bU_w(\mathcal{T}\vert_{\bU})$. It follows that the presheaf $E^i(\mathcal{M})$ is actually a sheaf on $\bX_{w}(\mathcal{T})$ as claimed.
\end{proof}
\begin{theorem}\label{thm2}
Let $\mathcal{M}$ be a coadmissible $G$-equivariant left $\mD_{\bX}$-module. Then $E^i(\mathcal{M})$ is a coadmissible $G$-equivariant right $\mathcal{D}_{\bX}$-module for every  $\mathcal{M} \in \mathcal{C}_{\bX/G}$ and every $i \geq 0$.
\end{theorem}
\begin{proof}
First, let us show that $E^i(\mathcal{M})$ is a sheaf of $G$- equivariant locally Fr\'{e}chet right $\mathcal{D}_{\bX}$-modules. Let $\bU \in \bX_{w}(\mathcal{T})$ and $H$ be a $\bU$-small subgroup of $G$. Then the bijection 
$$E^i(\mathcal{M})(\bU) \simeq E^i(\mathcal{M})(\bU,H)= Ext^i_{\wideparen{\mD}(\bU,H)}(\mathcal{M}(\bU), \wideparen{\mathcal{D}}(\bU,H))$$ 
from Remark \ref{proj} tells us that $E^i(\mathcal{M})(\bU)$ can be equipped with a canonical Fr\'{e}chet topology transferred from the canonical topology on $Ext^i_{\wideparen{\mD}(\bU,H)}(\mathcal{M}(\bU), \wideparen{\mathcal{D}}(\bU,H))$. This topology does not depend on the choice of $H$, so that $E^i(\mathcal{M})(\bU)$ becomes a coadmissible (right) $\wideparen{\mathcal{D}}(\bU,H)$-module. It remains to check that if $g \in G$ then each map $g^{E^i(\mathcal{M})}(\bU): E^i(\mathcal{M})(\bU) \longrightarrow E^i(\mathcal{M})(g\bU)$ is continuous for any $\bU \in \bX_{w}(\mathcal{T})$. Indeed, note that the map $g^{E^i(\mathcal{M})}(\bU)$ is a linear isomorphism with respect to the $K$- algebras isomorphism $\wideparen{g}_{\bU,H}: \wideparen{\mathcal{D}}(\bU,H) \longrightarrow  \wideparen{\mathcal{D}}(g\bU,gHg^{-1})$. We obtain that $g^{E^i(\mathcal{M})}(\bU)$ is continuous by \cite[Lemma $3.6.5$]{AW}. Thus $E^i(\mathcal{M})$ is in $Frech^r(G-\mathcal{D}_{\bX})$.\\
\\
Next,  write $M := \mathcal{M}(\bX)$. In view of Theorem \ref{thm1}, Proposition \ref{iso2} and Corollary \ref{coro1}, it remains to prove that when $(\bX,G)$ is small, the morphism $$  \Phi^i : \,\, {}^rLoc_{\bX}(  Ext^i_{\wideparen{\mathcal{D}}(\bX,G)}(M, \wideparen{\mathcal{D}}(\bX,G))) \longrightarrow E^i(\mathcal{M})$$ is indeed a $G$-equivariant morphism. \\
\\
In the sequel, to simplify the notations, we write 
\begin{center}
$\mathcal{N}:= E^i(\mathcal{M})$ and $\mathcal{N'}:= {}^rLoc_{\bX}(  Ext^i_{\wideparen{\mathcal{D}}(\bX,G)}(M, \wideparen{\mathcal{D}}(\bX,G))).$
\end{center}
 Let $ \bU \in \bX_{w}(\mathcal{T})$ and $g\in G$. Then by definition, $\Phi^i(\bU)= \varprojlim_H \Phi^i_{\bU,H}$, it reduces to prove that for any $\bU$-small subgroup $H$ of $G$ which is normal, the  diagram
\[\begin{tikzcd}
Ext^i_{\wideparen{\mathcal{D}}(\bX,G)}(M, \wideparen{\mathcal{D}}(\bX,G)) \wideparen{\otimes}_{\wideparen{\mathcal{D}}(\bX,H)}\wideparen{\mathcal{D}}(\bU,H) \arrow{r}{g^{\mathcal{N'}}_{\bU,H}} \arrow[swap]{d}{\Phi^i_{\bU,H}} & Ext^i_{\wideparen{\mathcal{D}}(\bX,G)}(M, \wideparen{\mathcal{D}}(\bX,G)) \wideparen{\otimes}_{\wideparen{\mathcal{D}}(\bX,{}^gH)}\wideparen{\mathcal{D}}(g\bU,{}^gH)\arrow{d}{\Phi^i_{g\bU,{}^gH}} \\
Ext^i_{\wideparen{\mathcal{D}}(\bU,H)}(\wideparen{\mathcal{D}}(\bU,H) \wideparen{\otimes}_{\wideparen{\mathcal{D}}(\bX,H)}M, \wideparen{\mathcal{D}}(\bU,H))\arrow{r}{g^\mathcal{N}_{\bU,H}} & Ext^i_{\wideparen{\mathcal{D}}(g\bU,{}^gH)}(\wideparen{\mathcal{D}}(g\bU,{}^gH) \wideparen{\otimes}_{\wideparen{\mathcal{D}}(\bX,{}^gH)}M, \wideparen{\mathcal{D}}(g\bU,{}^{g}H))\\
\end{tikzcd}\]
is commutative. Here recall that $g^{\mathcal{N}}_{\bU,H}$ and  $g^{\mathcal{N}'}_{\bU,H}$ correspond to the $G$-equivariant structures on the sheaf $\mathcal{N}$ and on $\mathcal{N'}$ respectively.\\
Choose a Lie lattice $\mathcal{L}$ in $\mathcal{T}(\bX)$ and a good chain $(J_n)$ for $\mathcal{L}$ such that 
$$\wideparen{\mathcal{D}}(\bX,G)= \varprojlim_n \mathcal{Q}_n(\bX,G).$$ 
By rescaling $\mathcal{L}$, we may suppose that $\bU$ is $\mathcal{L}$-accessible. This implies 
  $$\wideparen{\mathcal{D}}(\bU,H)= \varprojlim_n \mathcal{Q}_n(\bU,H).$$ Now, following Lemma \ref{Glemma}, $g\bU$ is also $\mathcal{L'}$-accessible with $\mathcal{L}':={g}^\mathcal{T}(\mathcal{L})$. Thus $\mathcal{L}'$ together with the good chain $(gJ_ng^{-1})$ defines the  Frechet-Stein structures 
  \begin{center}
  $\wideparen{\mathcal{D}}(\bX,{}^gH)=\varprojlim_n \mathcal{Q}_n(\bX,{}^gH)$ and $\wideparen{\mathcal{D}}(g\bU,{}^g H)= \varprojlim_n \mathcal{Q}_n(g\bU,{}^gH).$
  \end{center}
 Since each map of the above diagram is a linear map between coadmissible modules, they can be regarded as the inverse limits of systems of morphisms:
\begin{center}
$\Phi^i_{\bU,H}=\varprojlim_n \Phi^i_{\bU,H,n}$, \, $\Phi^i_{g\bU,{}^gH}=\varprojlim_n \Phi^i_{g\bU,{}^gH,n}$ \\
$g^\mathcal{N}_{\bU,H}=\varprojlim_n g^\mathcal{N}_{\bU,H,n}$  $g^\mathcal{N'}_{\bU,H}=\varprojlim_n g^\mathcal{N'}_{\bU,H,n}.$
\end{center}
As a consequence, it is enough to prove that the diagram
\[\begin{tikzcd}
Ext^i_{\mathcal{Q}_n(\bX,G)}(M_n, \mathcal{Q}_n(\bX,G)) \otimes \mathcal{Q}_n(\bU,H) \arrow{r}{g^{\mathcal{N'}}_{\bU,H,n}} \arrow[swap]{d}{\Phi^i_{\bU,H,n}} & Ext^i_{\mathcal{Q}_n(\bX,G)}(M_n, \mathcal{Q}_n(\bX,G)) \otimes \mathcal{Q}_n(g\bU,{}^gH)\arrow{d}{\Phi^i_{g\bU,{}^gH,n}}\\
Ext^i_{\mathcal{Q}_n(\bU,H)}(\mathcal{Q}_n(\bU,H) \otimes M_n, \mathcal{Q}_n(\bU,H))\arrow{r}{g ^{\mathcal{N}}_{\bU,H,n}} & Ext^i_{\mathcal{Q}_n(g\bU,{}^gH)}(\mathcal{Q}_n(g\bU,{}^gH) \otimes M_n,\mathcal{Q}_n(g\bU,{}^{g}H))\\
\end{tikzcd}\]
is commutative. Here we assume that $M=\varprojlim_n M_n$, with respect to the given Frechet-Stein structure on $\wideparen{\mathcal{D}}(\bX,G)$. After taking a resolution of $M_n$ by free $\mathcal{Q}_n(\bX,G)$-modules, it amounts to proving the commutativity of the above diagram for the case $i=0$, which means that the following diagram is commutative:
\[\begin{tikzcd}
Hom_{\mathcal{Q}_n(\bX,G)}(M_n, \mathcal{Q}_n(\bX,G)) \otimes \mathcal{Q}_n(\bU,H) \arrow{r}{g^{\mathcal{N}'}_{\bU,H,n}} \arrow[swap]{d}{\Phi_{\bU,H,n}} & Hom_{\mathcal{Q}_n(\bX,G)}(M_n, \mathcal{Q}_n(\bX,G)) \otimes \mathcal{Q}_n(g\bU,{}^gH)\arrow{d}{\Phi_{g\bU,{}^gH,n}} \\
Hom_{\mathcal{Q}_n(\bU,H)}(\mathcal{Q}_n(\bU,H) \otimes M_n, \mathcal{Q}_n(\bU,H))\arrow{r}{g ^{\mathcal{N}}_{U,H,n}} & Hom_{\mathcal{Q}_n(g\bU,{}^gH)}(\mathcal{Q}_n(g\bU,{}^gH) \otimes M_n,\mathcal{Q}_n(g\bU,{}^{g}H))\\
\end{tikzcd}\]
Let  $f \in Hom_{\mathcal{Q}_n(\bX,G)}(M_n, \mathcal{Q}_n(\bX,G))$ and $a \in \mathcal{Q}_n(\bU,H)$. It is enough to show that:
\[\Phi_{g\bU,{}^gH}\left( g^\mathcal{N'}_{\bU,H,n}\left( f\otimes a\right)\right)= g^{\mathcal{N}}_{\bU,H,n} \left( \Phi_{\bU,H,n}\left(f\otimes a\right)\right) \]
Since $g^\mathcal{N'}_{\bU,H,n}\left( f\otimes a\right)=(f\gamma_n(g)).g^{\mathcal{Q}_n}_{\bU,H}(a)$ and $\Phi_{\bU,H,n}(f \otimes a)= 1 \bar{\otimes} ( {p}_{G,H,n} (f) ). a$ (which are morphisms in  $Hom_{\mathcal{Q}_n(gU,{}^gH)}(\mathcal{Q}_n(g\bU,{}^gH) \otimes M_n,\mathcal{Q}_n(g\bU,{}^{g}H))$), it is equivalent to show that:
\begin{center}
$1 \bar{\otimes}({p}_{G,{}^gH,n}( (f \gamma_n(g^{-1})).{g}^{\mathcal{Q}_n}_{\bU, H}(a))) = {g}^{\mathcal{Q}_n}_{\bU,H} \circ ( 1 \bar{\otimes} ( {p}_{G,H,n} (f) ). a ) \circ ({g^{-1}})^{\mathcal{M}_n}_{\bU,H} $
\end{center}
where 
$$\gamma_n : G \longrightarrow \mathcal{Q}_n(\bX,G)^\times=\left(\widehat{U(\pi^n \mathcal{L})}_K \rtimes_{J_n} G\right)^\times $$ is the canonical group homomorphism from Remark \ref{gamma}.\\
Let $m \in M_n, b \in \mathcal{Q}_n(g\bU,{}^gH)$, we compute
\begin{align*}
(1 \bar{\otimes}(p_{G,{}^gH,n}\circ (f \gamma_n(g^{-1})).{g}_{\bU,H}^{\mathcal{Q}_n}(a)))(b \otimes m)= b.\, p_{G,{}^gH,n}(f(m)\gamma_n(g^{-1})){g}^{\mathcal{Q}_n}_{\bU,H}(a)
\end{align*}
and
\begin{align*}
&({g}_{\bU,H,}^{\mathcal{Q}_n} \circ ( 1 \bar{\otimes} ( p_{G,H,n}^{\bX} \circ f ). a ) \circ ( {g^{-1}})^{\mathcal{M}_n}_{\bU,H})(b \otimes m )\\
&= {g}_{\bU,H}^{\mathcal{Q}_n}\circ ( 1 \bar{\otimes} ( p_{G,H,n}^{\bX} \circ f ). a )({g^{-1}}_{\bU,H}^{\mathcal{Q}_n}(b)\otimes \gamma_n(g^{-1})m)\\
&={g}_{\bU,H}^{\mathcal{Q}_n}({g^{-1}}_{\bU,H}^{\mathcal{Q}_n}(b)).{g}_{\bU,H}^{\mathcal{Q}_n}(p_{G,H,n}^{\bX}(f(\gamma_n(g^{-1})m)a)\\
&= b.{g}_{\bU,H}^{\mathcal{Q}_n}(p_{G,H,n}^{\bX}(f(\gamma_n(g^{-1})m)){g}^{\mathcal{Q}_n}_{\bU,H}(a).\\
\end{align*}
Here, we identify the element $ p_{G,H,n}^{\bX}(f(\gamma_n(g^{-1})m) \in \mathcal{Q}_n(\bX,H)$ with its image in $\mathcal{Q}_n(\bU,H)$ via the natural restriction $\mathcal{Q}_n(\bX,H) \longrightarrow \mathcal{Q}_n(\bU,H)$ and the element $p_{G,{}^gH,n}(f(m)\gamma_n(g^{-1}) )\in \mathcal{Q}_n(\bX,{}^gH)$ with its image in $\mathcal{Q}_n(g\bU,{}^gH)$ via  $\mathcal{Q}_n(\bX,{}^gH) \longrightarrow \mathcal{Q}_n(g\bU,{}^gH).$
Thus, it remains to show that for any $m \in M_n$, one has
\begin{equation}\label{comm1}
p_{G,{}^gH,n}(f(m)\gamma_n({g^{-1}}))={g}_{\bU,H}^{\mathcal{Q}_n}(p_{G,H,n}^{\bX}(\gamma_n(g^{-1})f(m)).
\end{equation}
Consider the following diagram:
\begin{equation}\label{comm2}
\begin{tikzcd}
 \mathcal{Q}_n(\bX,G) \arrow{r}{Ad_{\gamma_n(g)}} \arrow[swap]{d}{p_{G,H,n}^{\bX}} & \mathcal{Q}_n(\bX,G)\arrow{d}{p_{G,gHg^{-1},n}} \\
\mathcal{Q}_n(\bX,H)\arrow{r}{{g}_{X,H}^{\mathcal{Q}_n}}\arrow{d}{} & \mathcal{Q}_n(\bX,gHg^{-1}) \arrow{d}{}\\
\mathcal{Q}_n(\bU,H)\arrow{r}{{g}^{\mathcal{Q}_n}_{\bU,H}} & \mathcal{Q}_n(g\bU,gHg^{-1}).
\end{tikzcd}
\end{equation}
By \cite[Definition $3.4.9(c)$ and Propostion $3.4.10$]{AW}, we see that $Ad_{\gamma_n(g)}={g}^{\mathcal{Q}_n}_{\bX,H}$ on $\mathcal{Q}_n(\bX,H) \subset \mathcal{Q}_n(\bX,G)$ and the commutativity of the lower diagram of the diagram (\ref{comm2}) follows from loc.cit. On the other hand, it is proved in the proof of Proposition \ref{comm3} that the upper diagram of (\ref{comm2}) is commutative. 
 Hence we may compute as follows:
 \begin{align*}
 p_{G,{}^gH,n}(f(m) \gamma_n(g^{-1}))&=p_{G,{}^gH,n}(\gamma_n(g^{})\gamma_n(g^{-1})f(m)\gamma_n (g^{-1}))\\
 &=p_{G,{}^gH,n}(g_{X,H}^{\mathcal{Q}_n}(\gamma_n(g^{-1})f(m)))\\
 &= {g}^{\mathcal{Q}_n}_{\bU,H}(p_{G,H}(\gamma_n(g^{-1})f(m)).
 \end{align*}
Hence we obtain the commutativity of (\ref{comm2}) and so the theorem follows.

\end{proof}
\begin{definition} \label{e-def} Let $\mathcal{M} \in \mathcal{C}_{\bX/G}$, then we define for any non-negative integer $i \geq 0$:
\[\mathcal{E}^i(\mathcal{M}):= \mathcal{H}om_{\mathcal{O}_{\bX}}(\Omega_{\bX}, E^i(\mathcal{M})).\]
\end{definition}
\begin{pro}
For every $i \geq 0$, $\mathcal{E}^i$ is an endofunctor on the category $\mathcal{C}_{\bX/G}$.\end{pro}
\begin{proof}
Following Theorem \ref{chanfunc} and Theorem \ref{thm2}, the sheaf $\mathcal{E}^i(\mathcal{M})$ is a coadmissible $G$-equivariant left $\mathcal{D}_{\bX}$-module. Now if $f: \mathcal{M} \longrightarrow \mathcal{M'}$ is a morphism of coadmissible $G$-equivariant left $\mathcal{D}_{\bX}$-modules, then for any $\bU \in \bX_{w}(\mathcal{T})$ and any $\bU$-small subgroup $H$ of $G$, it follows that the $\wideparen{\mathcal{D}}(\bU,H)$-linear map $f(\bU): \mathcal{M}(\bU) \longrightarrow \mathcal{M'}(\bU)$ induces a morphism 
\[Ext^i_{\wideparen{\mathcal{D}}(\bU,H)}(\mathcal{M'}(\bU), \wideparen{\mathcal{D}}(\bU,H))\longrightarrow Ext^i_{\wideparen{\mathcal{D}}(\bU,H)}(\mathcal{M}(\bU), \wideparen{\mathcal{D}}(\bU,H)),\]
which is right $\wideparen{\mathcal{D}}(\bU,H)$-linear. Hence by \cite[Lemma $3.6.5$]{AW}, this is a continuous map with respect to the natural Fr\'{e}chet topologies on both sides. In this way we obtain a morphism of $G$-equivariant locally Fr\'{e}chet $\mathcal{D}_{\bX}$modules
$$E^i(f): E^i(\mathcal{M'}) \longrightarrow E^i(\mathcal{M})$$
whose local sections are continuous. Now, if $g: \mathcal{M'} \longrightarrow \mathcal{M''}$ is another morphism in $\mathcal{C}_{\bX/G}$, then it is straightforward to show that $E^i(id)=id$ and $E^i(g \circ f)=E^i(f) \circ E^i(g)$, which ensures that $E^i$ is a functor from $\mathcal{C}_{\bX/G}$ into $\mathcal{C}^r_{\bX/G}$. Finally $\mathcal{E}^i$ is a composition of two functors, so it is a functor from $\mathcal{C}_{\bX/G}$ into itself, as claimed.
\end{proof}

\section{Dimension and weakly holonomic equivariant $\mathcal{D}$-modules}\label{section_five}

\subsection{Dimension theory for coadmissible equivariant $\mathcal{D}$-modules}\label{subsec_dim}

In this section, we fix a smooth rigid analytic $K$-variety $\bX$ of dimension $d$  and a $p$-adic Lie group $G$ acting continuously on $\bX$. We are now ready to introduce the notion of dimension for coadmissible $G$-equivariant $\mathcal{D}_{\bX}$-modules. Recall that the set $\bX_{w}(\mathcal{T})$ is a basis for the Grothendieck topology on $\bX$.
\begin{definition} Let $\mathcal{M}\in \mathcal{C}_{\bX/G}$. Let $\mathcal{U}$ be an admissible covering of $\bX$ by affinoid subdomains in $\bX_{w}(\mathcal{T})$. The \textit{dimension} of $\mathcal{M}$ with respect to $\mathcal{U}$ is defined as follows:
\[d_{\mathcal{U}}(\mathcal{M}):= \sup \left\lbrace d(\mathcal{M}(\bU)) \vert \bU \in \mathcal{U}\right\rbrace,\]
where $d(\mathcal{M}(\bU))$ is the dimension of the coadmissible $\wideparen{\mathcal{D}}(\bU,H)$-module $\mathcal{M}(\bU)$ for some $\bU$-small subgroup $H$ of $G$. 
\end{definition}
Recall that $d(\mathcal{M}(\bU))$ does not depend on the choice of $H$, cf. remark \ref{redim}. 
 
\begin{pro}  Let $\mathcal{M}\in \mathcal{C}_{\bX/G}$ and let $\mathcal{U}$ and $\mathcal{V}$ be two admissible coverings of $\bX$ by elements in $\bX_{w}(\mathcal{T})$. Then $d_{\mathcal{U}}(\mathcal{M})= d_{\mathcal{V}}(\mathcal{M})$.
\end{pro}
\begin{proof}
We may assume that $\mathcal{V}$ is a refinement of $\mathcal{U}$ and every element of $\mathcal{U}$ has an admissible covering by elements of $\mathcal{V}$. Let $\bU_1,...,\bU_k \in \mathcal{V}$ be a cover of $\bU_0 \in \mathcal{U}$ (which is quasi-compact!).  We fix an open compact subgroup $H$ of $G$ such that $(\bU_0,H)$ is small and choose a $H$-stable affine formal model $\mathcal{A}$ in $\mathcal{O}(\bU_0)$ and a $H$-stable smooth $\mathcal{A}$-Lie lattice $\mathcal{L}$ in $\mathcal{T}(\bU_0)$. Then by \cite[ Lemma $4.4.1$]{AW}, we may assume that $H$ stabilizes $\mathcal{A}$, $\mathcal{L}$ and each member $\bU_i$ in $\mathcal{V}$. By replacing $\mathcal{L}$ by a sufficiently large $\pi$-power multiple, we may also assume that each $ \bU_i$ is a $\mathcal{L}$-accessible affinoid subspace in $\bU_0$ so that $\bU_1,...,\bU_k \in (\bU_0)_{ac}(\mathcal{L},H)$ and they form an $(\bU_0)_{ac}(\mathcal{L},H)$-covering. Recall from section \ref{section_Q} the sheaf of rings $\mathcal{Q}_n(-,H)$ and the sheaf of modules $\mathcal{M}_n$ on the Grothendieck topology $\bX_{ac}(\mathcal{L},H)$, which are induced by $\mathcal{M}$. Then
\begin{center}
$\wideparen{\mathcal{D}}(\bU_i,H)\simeq \varprojlim_n \mathcal{Q}_n(\bU_i, H)$ \, and \, $\mathcal{M}(\bU_i)\simeq \varprojlim_n \mathcal{M}_n(\bU_i)$ \, for all $i=0,1,..,k$.
\end{center} 
Each $\mathcal{M}(\bU_i)$ is a coadmissible $\wideparen{\mathcal{D}}(\bU_i,H)$-module and  by Definition \ref{defdim}, $d(\mathcal{M}(\bU_i))= 2d-j_H(\mathcal{M}(\bU_i))$ for each $i$.\\
Now by \cite[Theorem $4.3.14$]{AW}, one has that $\bigoplus_{i=1}^{k}\mathcal{Q}_n(\bU_i,H)$ is a faithfully flat right $\mathcal{Q}_n(\bU_0,H)$-module. Thus applying \cite[Proposition $7.5(c)$]{AWI} gives that $\bigoplus_{i=1}^k \wideparen{\mathcal{D}}(\bU_i,H)$ is c-faithfully flat over $\wideparen{\mathcal{D}}(\bU_0,H)$. On the other hand, the completed tensor product commutes with finite direct sum, so that:
\begin{align*}
& Ext^i_{\wideparen{\mathcal{D}}(\bU_0,H)}(\mathcal{M}(\bU_0), \wideparen{\mathcal{D}}(\bU_0,H)) \wideparen{\otimes}_{\wideparen{\mathcal{D}}(\bU_0,H)}\oplus_{i=1}^k\wideparen{\mathcal{D}}(\bU_i,H)\\
 &\simeq \oplus_{i=1}^kExt^i_{\wideparen{\mathcal{D}}(\bU_0,H)}(\mathcal{M}(\bU_0), \wideparen{\mathcal{D}}(\bU_0,H))\wideparen{\otimes}_{\wideparen{\mathcal{D}}(\bU_0,H)}\wideparen{\mathcal{D}}(\bU_i,H)\\
 &\simeq \oplus_{i=1}^kExt^i_{\wideparen{\mathcal{D}}(\bU_i,H)}(\mathcal{M}(\bU_i), \wideparen{\mathcal{D}}(\bU_i,H)).
\end{align*}

By consequence, one has
\[j_H(\mathcal{M}(\bU_0))=\inf\lbrace j_H(\mathcal{M}(\bU_i))\,: \mathcal{M}(\bU_i)\neq 0, i=1, 2, ..., k \rbrace,\]
so the proposition follows immediately.
\end{proof}

The above proposition means that the dimension of a module $\mathcal{M}\in \mathcal{C}_{\bX/G}$ does not depend on the choice of an admissible covering $\mathcal{U}$ of $\bX$. Hence we we will from now on denote it by $d_{\bX}(\mathcal{M})$ or simply $d(\mathcal{M})$ if the space $\bX$ is clear. By definition, $$0 \leq d_{\bX}(\mathcal{M}) \leq 2d.$$

\subsection{Bernstein's inequality}

Let $\bX$ be a smooth rigid analytic variety and $G$ be a $p$-adic Lie group acting continuously on $\bX$. We say that 
{\it Bernstein's inequality holds in $\mathcal{C}_{\bX/G}$}, if $d(\mathcal{M})\geq \dim \bX$
for all nonzero $\mathcal{M}\in \mathcal{C}_{\bX/G}$. In this section, we show that Bernstein's inequality holds in $\mathcal{C}_{\bX/G}$ whenever $\bX$ has good reduction, i.e. admits a formal model which is smooth. We do not know whether this result generalizes to all $\bX$ (with continuous $G$-action). In the case $G=1$, we know that Bernstein's inequality holds in $\mathcal{C}_{\bX}$ for any $\bX$, cf. \cite[Thm. 6.2]{AWIII}.
\vskip5pt
We recall the following basic notion. 
Let $R$ be a commutative ring and $A$ a commutative $R$-algebra. Let $L$ be a $(R,A)$-Lie algebra. Let $I\subset A$ be an ideal. A finite set $\lbrace x_1,...,x_d\rbrace$ of elements in $L$ is an \textit{$I$-standard basis} if it satisfies the following conditions:
\begin{itemize}
\item[(i)] $\lbrace x_1,...,x_d\rbrace$ is a basis of $L$ as an $A$-module (which implies that $L$ is free over $A$),
\item[(ii)] There exists a set $F= \lbrace f_1,...,f_r\rbrace \subset I$ with $r \leq d$ which generates $I$ such that
\begin{center}
$x_i.f_j=\delta_{ij}$ for all $1 \leq i \leq d$ and $1 \leq j \leq r$.
\end{center}
\end{itemize}

\begin{lemma} \label{bernsteinlemma} Let $(\bX,G)$ be small and suppose that 
$\bX=Sp(K\langle x_1,...,x_d \rangle)$ is a polydisc. Then Bernstein's inequality holds in $\mathcal{C}_{\bX/G}$. 
\end{lemma}
\begin{proof}
Let $\mathcal{M}$ be a non-zero module in $\mathcal{C}_{\bX/G}$ and $M:=\mathcal{M}(\bX) \in \mathcal{C}_{\wideparen{\mathcal{D}}(\bX,G)}$. Denote by $\partial_1,...,\partial_d$ the partial derivations with respect to coordinates $x_1,...,x_d$. Write $\mathcal{A}:=\mathcal{R}\langle x_1,...,x_d\rangle$ and $\mathcal{L}:= Der_{\mathcal{R}}(\mathcal{A})= \partial_1\mathcal{A}\oplus...\oplus \partial_d\mathcal{A}$. Then $\mathcal{A}$ is an affine formal model of $\mathcal{O}(\bX)=K\langle x_1,...,x_d\rangle$ and $\mathcal{L}$ is a free $\mathcal{A}$-Lie lattice in $\mathcal{T}(\bX)$. Now, we can choose an open subgroup $H$ of $G$ which stabilises $\mathcal{A}$ and $\mathcal{L}$ \cite[Lemma 3.2.4/3.2.8]{AW}. Thus, $(\bX,H)$ is small and
\[\wideparen{\mathcal{D}}(\bX,H)\cong \varprojlim_n \widehat{U(\pi^n\mathcal{L})}_K\rtimes_{J_n}H\]
for any choice of a good chain ${(J_n)}_n$ for $\mathcal{L}$. Note that $d_{\wideparen{\mathcal{D}}(\bX,G)}(M)=d_{\wideparen{\mathcal{D}}(\bX,H)}(M)$ (Remark \ref{redim} (ii)) and there exist $n$ sufficiently large such that
\begin{center}
$j_{\wideparen{\mathcal{D}}(\bX,H)}(M)= j_{\widehat{U(\pi^n\mathcal{L})}_K\rtimes_{J_n}H}(M_n)$, with $M_n= (\widehat{U(\pi^n\mathcal{L})}_K\rtimes_{J_n}H) \otimes_{\wideparen{\mathcal{D}}(\bX,H)}M$.
\end{center}
On the other hand, Proposition \ref{keypro} and Lemma \ref{keylemma2} tell us that 
\[j_{\widehat{U(\pi^n\mathcal{L})}_K\rtimes_{J_n}H}(M_n)=j_{\widehat{U(\pi^n\mathcal{L})}_K}(M_n).\]
Now applying \cite[Corollary $7.4$]{AWirr} gives $j_{\widehat{U(\pi^n\mathcal{L})}_K}(M_n)\leq d$ and so $d(\mathcal{M})\geq d$ as claimed.
\end{proof}

\begin{lemma}\label{lem_cov1} Let $\bX$ be a smooth affinoid of dimension $d$.
 Let $\mathcal{A}$ be an affine formal model and $\mathcal{L}$ a smooth Lie lattice in $\mathcal{T}(\bX)$. Suppose that $\mathcal{U}$ is a finite $\bX_{ac}(\mathcal{L})$-covering by 
 affinoid subdomains $\bX_i$ with $\mathcal{L}$-stable affine formal models $\mathcal{A}_i$ of $\mathcal{O}(\bX_i)$. Let $\mathcal{L}_i=\mathcal{A}_i\otimes_{\mathcal{A}}\mathcal{L}$. Suppose that $\widehat{U(\mathcal{L})}_K$ and each $\widehat{U(\mathcal{L}_i)}_K$ is Auslander-Gorenstein of dimension at most $2d$. Then $$d(\mathcal{M})=\sup d(\mathcal{M}_i)$$
for any non-zero finitely generated $\widehat{U(\mathcal{L})}_K$-module $\mathcal{M}$, where $
\mathcal{M}_i=\mathcal{O}(X_i)\otimes\mathcal{M}$.
\end{lemma}
\begin{proof} The map $\widehat{U(\mathcal{L})}_K\rightarrow\oplus_i \widehat{U(\mathcal{L}_i)}_K$ is faithfully flat by \cite[Thm. 4.9]{AWII}. By definition of the grade and faithfully flat descent for Ext-groups, this implies $ j(\mathcal{M})=\inf \{ j(\mathcal {M}_i) \mid \mathcal{M}_i \neq 0\}.$
\end{proof}
We prepare the next result with an auxiliary lemma. 
Let $\bX$ be a smooth affinoid with $A=\mO(\bX)$. Let $\mathcal{A}$ be an affine formal model and $\mathcal{I}\subset \mathcal{A}$ an ideal. Let $\mathcal{L}$ be an $\mA$-Lie lattice in $\mathcal{T}(\bX)$. Recall the normalizer 
$N_\mathcal{L}(\mathcal{I}):=\{x\in\mL: x\cdot \mI \subseteq\mI \}$
of $\mathcal{I}$ in $\mathcal{A}$ and let
$$\mathcal{N}:=N_\mathcal{L}(\mathcal{I})/\mI\mL.$$
Let 
$W:=\widehat{U(\mathcal{N})_K}$ and 
 $U:=\widehat{U(\mathcal{L})_K}$ and set $I=\mathcal{I}A$. The $(W,U)$-bimodule $U/I U$
gives rise to the functor $$i_{+}: M\mapsto M\otimes_{W} U/I U$$ from finitely generated right $W$-modules to finitely generated right 
 $U$-modules. For a set $F\subseteq A$, we let $C_{\mL}(F):=\{x\in\mL: x\cdot f =0 \text{ for all }f\in F\}$ be the centralizer 
of $F$ in $\mL.$ 
\begin{lemma}\label{lem_int1}
Suppose there are generators $F=\{f_1,...,f_r\}$ for the ideal $I$ such that 
$\mL\cdot f_i \subseteq\mA$ for each $1\leq i \leq r$ and elements $x_1,...,x_r\in\mL$ such that $x_i\cdot f_j=\delta_{ij}$. Then $j_{U}(i_+(M))=j_{W}(M)+r.$
\end{lemma}
\begin{proof}
Let $\mC=C_{\mL}(F)$. According to \cite[Lem. 4.1]{AWII} one has 
$ \mL= \big(\oplus_{i=1,...,r}\mA x_i \big) \oplus \mC.$
In particular, $\mC$ is a smooth $(\mR,\mA)$-Lie algebra. Moreover, the equality implies that $$N_{\mL}(\mI)=\big(\oplus_{i=1,...,r}\mI x_i \big) \oplus \mC.$$
Indeed, any $z=\sum_{i=1}^{r} a_ix_i +c$ with $c\in\mC$ such that $a_k\notin\mI$ for some $k$ does not belong to $N_{\mL}(\mI)$, because of $z\cdot f_k=a_k\notin\mI$. This shows the forward inclusion. For the reverse inclusion, note that $\oplus_{i=1,...,r}\mI x_i\subseteq N_{\mL}(\mI)$ is clear, since $x_i\cdot \mA\subseteq\mA$ for any $i$. Moreover, if $x\in \mC$ and $f\in\mI$, say $f=\sum_j a_jf_j$ with $a_j\in A$, then $$x\cdot f=\sum_j a_j(x\cdot f_j)+ (x\cdot a_j)f_j= \sum_j (x\cdot a_j)f_j\in I\cap \mA=\mI,$$ 
which proves $\mC\subseteq N_{\mL}(\mI)$. The statement implies that the natural map $\mC\rightarrow N_{\mL}(\mI)/\mI\mL$ is surjective.
Its kernel equals $\mI\mL\cap\mC=\mI\mC$, so that $\mC/\mI\mC\simeq \mN$ as $(R,\mA)$-Lie algebras. Let $g_1,...,g_s$ generate the ideal $\mI$ in $\mA$ and let $\overline{\mA}=\mA/\mI$.
As in the proof of \cite[5.8]{AWII}, this induces a complex of $U(\mC)$-modules 
$$U(\mC)^s\longrightarrow U(\mC)\longrightarrow U(\overline{\mA}\otimes_{\mA}\mC)\longrightarrow 0,$$
which is exact, according to \cite[2.3]{AWI}, since $\mC$ is a smooth $(\mR,\mA)$-Lie algebra.
It stays exact after $\pi$-adic completion, since $U(\mC)$ is noetherian and all modules are of finite type. This yields an isomorphism of Banach algebras $V/FV\simeq W$, where 
$V:=\widehat{U(\mathcal{C})_K}$. Hence as right $U$-modules
$ i_+(M)=M \otimes_{W} U/I U\simeq M\otimes_{V} U.$
The claim now follows from \cite[6.2]{AWIII}.
\end{proof}
For example, the hypothesis of the preceding lemma are satisfied, if $\mL$ admits an $\mI$-standard basis. As to the existence of such bases, we mention the following integral version of 
\cite[6.2]{AWII}.
\begin{lemma}\label{lem_cov2}
Suppose that $\mathcal{A}$ is an affine formal model, $\mathcal{I}\subset \mathcal{A}$ an ideal with $\mathcal{I}\neq\mathcal{I}^2$. Let $\mathcal{L}$ be an $(\mR, \mA)$-Lie algebra which is free as $\mathcal{A}$-module. Suppose that $N_\mathcal{L}(\mathcal{I})/\mathcal{I}\mathcal{L}$ and $\mathcal{I}/\mathcal{I}^2$ are free as $\overline{\mathcal{A}}=\mA /\mI$-modules and that the map $\mathcal{L}/ \mathcal{I}\mathcal{L} \rightarrow {\rm Hom}_{\overline{\mathcal{A}}}(\mathcal{I}/\mathcal{I}^2,\overline{\mathcal{A}})$ is surjective. Then there is an element $g\in 1+\mI$, such that with $\mA':=\mA\langle1/g\rangle, \mL':=\mA'\otimes_{\mA}\mL$ and
$\mI':=\mA'\otimes_{\mA}\mI$,
the $(\mR, \mA')$-Lie algebra $\mL'$ admits an $\mI'$-standard basis. The formal open subscheme $D(1/g)={\rm Spf} \mA'$ of ${\rm Spf} \mA$ contains the closed subspace ${\rm Spf} (\overline{\mathcal{A}}).$
\end{lemma}
\begin{proof} As in \cite[6.2]{AWII}, one obtains an element $abc\in 1+\mI\subset \mA$ (in the notation of loc.cit.) such that with $g:=abc$, $\mA':=\mA\langle1/g\rangle$ and
$\mI':=\mA'\langle1/g\rangle\otimes_{\mA}\mI$ the $(\mR, \mA')$-Lie algebra $\mL'$ admits an $\mI'$-standard basis. Moreover, $D(1/g)$ contains the closed subspace ${\rm Spf} (\mA/\mI).$ Indeed, 
suppose there is $\frak{p}\in {\rm Spf}(\mA/\mI)$, an open prime ideal of $\mA$ which contains $\mI$, such that $g(\frak{p})=0$, i.e. $g\in \frak{p}$. Since $g\in 1+\mI$, this implies $1\in \frak{p}$, a contradiction. 
\end{proof}

\begin{pro}\label{bernsteinlem}
Let $(\bX,G)$ be small and suppose that $\bX$ has good reduction. Then Bernstein's inequality holds in $\mathcal{C}_{\bX/G}$.
\end{pro}
\begin{proof}
By Prop. \ref{dimcom} it suffices to consider right modules. 
If $\bX$ is a polydisc, we are done by lemma \ref{bernsteinlemma}. So assume that $\bX$ is not a polydisc. Let a nonzero $\mathcal{M}\in \mathcal{C}^r_{\bX/G}$ be given and let $M:=\mathcal{M}(\bX) \in \mathcal{C}^r_{\wideparen{\mathcal{D}}(\bX,G)}$. By hypothesis, there is a formal polydisc $\mathcal{Y}={\rm Spf}(\mathcal{A})$ with $\mathcal{A}=\mathcal{R}\langle y_1,...,y_d \rangle$ and an ideal $\mI\subset\mA$ such that $\overline{\mathcal{A}}:=\mathcal{A}/\mathcal{I}$ is a smooth affine formal model for $\mathcal{O}(\bX)$. Let $\bY=\mathcal{Y}_K$ be the generic fibre of $\mathcal{Y}$.
Let $\mL:=Der_{\mathcal{R}}(\mathcal{A})$. Then $\mN:=N_\mathcal{L}(\mathcal{I})/\mathcal{I}\mathcal{L}\simeq Der_{\mathcal{R}}(\overline{\mathcal{A}})$ is a $\overline{\mathcal{A}}$-Lie lattice in $\mathcal{T}(\bX)$ and 
there is a compact open subgroup $H$ of $G$ which stablizes both $\overline{\mathcal{A}}$ and 
$\mN$, cf. \cite[3.2.4 and 3.2.8]{AW}. Hence
\[\wideparen{\mathcal{D}}(\bX,H)\cong \varprojlim_n \widehat{U(\pi^n\mathcal{N})}_K\rtimes_{J_n}H\]
for any choice of a good chain $(J_\bullet)$ for $\mathcal{N}$, cf.  \cite[3.3.4]{AW}.
Note that $d_{\wideparen{\mathcal{D}}(\bX,G)}(M)=d_{\wideparen{\mathcal{D}}(\bX,H)}(M)$ (Remark \ref{redim} (ii)) and there exist $n$ sufficiently large such that
\begin{center}
$j_{\wideparen{\mathcal{D}}(\bX,H)}(M)= j_{\widehat{U(\pi^n\mathcal{N})}_K\rtimes_{J_n}H}(M_n)$, with $M_n= (\widehat{U(\pi^n\mathcal{N})}_K\rtimes_{J_n}H) \otimes_{\wideparen{\mathcal{D}}(\bX,H)}M$.
\end{center}
On the other hand, Proposition \ref{keypro} and Lemma \ref{keylemma2} tell us that 
$j_{\widehat{U(\pi^n\mathcal{N})}_K\rtimes_{J_n}H}(M_n)=j_{\widehat{U(\pi^n\mathcal{N})}_K}(M_n).$ It therefore suffices to prove for sufficiently large $n$ the main assertion:
 $$d_{\bX}\leq  d_{\widehat{U(\pi^n\mathcal{N})}_K}(M_n).$$
 
\vskip5pt

Our proof relies on Bernstein's inequality for completed deformed Weyl algebras as proved in 
\cite[7.4]{AWirr}: for any $n\geq 1$, the Banach algebra $\widehat{U(\pi^n\mathcal{L})_K}$ is of this type, whence 
$$ d\leq d_{\widehat{U(\pi^n\mathcal{L})_K}}(N)$$
for any finitely generated nonzero right $\widehat{U(\pi^n\mathcal{L})_K}$-module $N$. 

\vskip5pt

Let $\mathcal{X}={\rm Spf}(\overline{\mathcal{A}})$ and let $i:\mX\rightarrow\mY$ be the induced embedding. The second fundamental sequence for differentials associated to the smooth embedding $i$ is exact. Dualizing it yields the exact sequence of 
$\overline{\mathcal{A}}$-modules
$$ 0 \rightarrow \mN=N_\mathcal{L}(\mathcal{I})/\mathcal{I}\mathcal{L} \rightarrow \mathcal{L}/ \mathcal{I}\mathcal{L} \rightarrow {\rm Hom}_{\overline{\mathcal{A}}}(\mathcal{I}/\mathcal{I}^2,\overline{\mathcal{A}})\rightarrow 0.$$
Note that 
multiplication by $\pi^n$ gives $\pi^n\mN=N_{\pi^n\mathcal{L}}(\mathcal{I})/\mathcal{I}(\pi^n\mathcal{L}).$ Now let 
$W:=\widehat{U(\pi^n\mathcal{N})_K}$ and 
 $U:=\widehat{U(\pi^n\mathcal{L})_K}$ and consider the functor $$i_{+}: M\mapsto M\otimes_{W} U/I U$$ from finitely generated right $W$-modules to finitely generated right 
 $U$-modules. Here $I=\mathcal{I}A$.
 The functor is compatible with localization in the following sense. Let
 $\mY'={\rm Spf}(\mA')\subset \mY$ be an open affine subscheme with generic fibre 
 $\bY'=\mY'_K$ inducing the closed embedding
 $$i': \mathcal{X}'=\mathcal{X}\cap\mathcal{Y}'\rightarrow\mathcal{Y}'.$$ 
  Let $\mathcal{I}'=\mI\mA'$ and $\mathcal{L}'=\mA'\otimes_{\mA} \mL$ and 
$\mN'=N_\mathcal{L'}(\mathcal{I}')/\mathcal{I}'\mathcal{L}'$. 
The above exact sequence tensored over $\overline{\mathcal{A}}$ with 
$\overline{\mathcal{A}'}=\mA'/\mI'$ remains exact by flatness and proves that 
$\mN'\simeq \mA'\otimes_{\mA} \mN.$ In particular, $\pi^n\mN'\simeq  \mA'\otimes_{\mA} (\pi^n\mN)$. Let $W'=\widehat{U(\pi^n\mathcal{N}')_K}$ and $U'=\widehat{U(\pi^n\mathcal{L}')_K}$ giving rise, as above, to a functor $(i')_+$. Evaluating this functor on the finitely generated 
$W'$-module $M'=M\otimes_W W'$ yields
$$(i')_+(M')=(M\otimes_W W') \otimes_{W'} U'/I' U'\simeq (M\otimes_W U/I U) \otimes_U U'=i_+(M)\otimes_U U'. $$
 In other words, the sheaf $\mathrm{Loc}(i_+(M))$ on 
 $\bY_{w}(\mathcal{L})$ has local sections over $\bY'$ equal to $(i')_+(M')$.

\vskip5pt

Since $i$ is a smooth embedding, there is a finite covering of $\mY$ by affine open formal subschemes $\mY'$ such that 
$$ 0 \rightarrow \mN'=N_{\mathcal{L}'}(\mathcal{I}')/\mathcal{I}'\mathcal{L}' \rightarrow \mathcal{L}'/ \mathcal{I}'\mathcal{L}' \rightarrow {\rm Hom}_{\overline{\mathcal{A}'}}(\mathcal{I}'/\mathcal{I}'^2,\overline{\mathcal{A}'})\rightarrow 0$$
is a sequence of finite free $\overline{\mathcal{A}'}$-modules. Passing to connected components, we may suppose that each $\mY'$ is connected.
There are two cases. 

Suppose first that $\mI'=\mI'^2$. Then either $\mI'=0$ or $\mI'=\mA'$ (compare \cite[6.3 ]{AWII}) and so either $\mY'\subset\mX$ or $\mY'\cap \mX=\emptyset$. Note that $\mY'\subset\mX$ implies $d={\rm dim} \bY \leq {\rm dim} \bX$, whence $\bX=\bY$, a contradiction.

Suppose secondly that $\mI'\neq \mI'^2$. According to \ref{lem_cov2}, there is an affine open formal subscheme $\mY''={\rm Spf} \mA''\subset\mY'$ containing the 
closed subspace $\mX':=\mY'\cap \mX$ and such that with $\mL'':=\mA''\otimes_{\mA'}\mL'$ and
$\mI'':=\mA''\otimes_{\mA'}\mI'$ the following holds:
the $(\mR, \mA'')$-Lie algebra $\mL''$ admits an $\mI''$-standard basis.
By replacing $\mY'$ by $\mY''$ and by adding finitely many connected affine open formal subschemes in the open complement $\mY\setminus \mX$, we may therefore assume that each member of our finite
covering $\mY'$ of $\mY$ satisfies exactly one of the following conditions: either
$\mY'\cap\mX=\emptyset$ or else $\mL'$ admits an $\mI'$-standard basis. 
\vskip5pt
Let $\bY'=\mY'_K$ and $\bX'=\mX'_K$. Note that $U$ and $U'$ is Auslander-Gorenstein of dimension at most $2d$ according to \cite[4.3]{AWIII} and its proof (this is the case $m=0$ in the notation of loc.cit., since we have Gorenstein models and smooth Lie lattices). Of course, any $\bY'$ is $\mL$-admissible, since $\mA'$ is an $\mL$-stable formal model of $\bY'$. According to \ref{lem_cov1}, we therefore have 
$ d(N)=\sup_{\bY'} d(N')$
for any nonzero finitely generated $\widehat{U(\pi^n\mathcal{L})}_K$-module $N$. For similar reasons, we have that $W$ and $W'$ is Auslander-Gorenstein of dimension at most $2d_{\bX}$. Any $\bX'$ is $\mN$-admissible, since $\overline{\mA}'$ is an $\mN$-stable formal model of $\bX'$. According to \ref{lem_cov1}, we have 
$ d(M)=\sup_{\bX'} d(M')$
for any nonzero finitely generated $\widehat{U(\pi^n\mathcal{N})}_K$-module $M$.

We now finish the proof of the main assertion on the $\widehat{U(\pi^n\mathcal{N})}_K$-module $M_n$. Let $N:=i_+(M_n)$. Its dimension is then computed over those $N'$ such that 
$\mY'\cap\mX\neq\emptyset$. Hence, there is $\mX'$ with $d(N)=d_{\bY'}(N')$ and such that $\mL'$ admits an $\mI'$-standard basis. Let $f_1,...,f_r$ be the corresponding set of generators for $\mI'$. The hypothesis in Lem. \ref{lem_int1} for the Lie lattice $\mL'$ and the ideal $I'$ are therefore satisfied. Hence, the same is true for the Lie lattice $\pi^n\mL'$ relative to the generators $\frac{f_j}{\pi^n}$ for $I'$ and we obtain 
$j_{\widehat{U(\pi^n\mathcal{L}')}_K}((i')_+(M'_n))=j_{\widehat{U(\pi^n\mathcal{N}')}_K}(M_n')+r$. The compatibility of $i_+$ with localization shows that $N'=(i')_+(M'_n)$, so this yields $$ d\leq d_{\widehat{U(\pi^n\mathcal{L}')}_K}(N')=2d-j_{\widehat{U(\pi^n\mathcal{L}')}_K}(N')=2d-j_{\widehat{U(\pi^n\mathcal{N}')}_K}(M_n')-r.$$
Because of $d=d_{\bX}+r$, the right-hand side equals 

$$2r+(2d_{\bX}-j_{\widehat{U(\pi^n\mathcal{N}')}_K}(M_n'))-r=d_{\widehat{U(\pi^n\mathcal{N}')}_K}(M_n')+r$$
and so $d_{\bX}=d-r\leq d_{\widehat{U(\pi^n\mathcal{N}')}_K}(M_n')$. Hence 
$$d_{\bX}\leq \sup_{\bX'} d_{\widehat{U(\pi^n\mathcal{N}')}_K}(M_n') = d_{\widehat{U(\pi^n\mathcal{N})}_K}(M_n).$$ 
\end{proof}

The following result is a direct consequence of the preceding proposition, given the local nature of the dimension function, cf. \ref{subsec_dim}.
\begin{coro}\label{bernsteinlem_global}
Let $\bX$ be a smooth rigid variety and $G$ be a $p$-adic Lie group which acts continuously on $\bX$. Suppose that $\bX$ has good reduction.
Then Bernstein's inequality holds in $\mathcal{C}_{\bX/G}$.
\end{coro}



\subsection{Weakly holonomic equivariant $\mathcal{D}$-modules and duality}

Let $\bX$ be a smooth rigid analytic variety of dimension $d$ and $G$ be a $p$-adic Lie group which acts continuously on $\bX$. 
\begin{definition} A module $\mathcal{M}\in \mathcal{C}_{\bX/G}$ is called \textit{weakly holonomic} if $d(\mathcal{M}) \leq \dim \bX$. The category of  weakly holonomic equivariant  $\mathcal{D}_{\bX}$-modules is denoted by $\mathcal{C}^{wh}_{\bX/G}$. 
\end{definition}
There is an analogous version of the preceding definition for right modules in $\mathcal{C}^r_{\bX/G}$.
\vskip5pt
Here is a first example in dimension one.
\begin{ex} Let $(\bX,G)$ be small with $\dim \bX=1$. Let $P\in \mathcal{D}(\bX)$ be a regular differential operator and $M =\wideparen{\mathcal{D}}(\bX,G)/\wideparen{\mathcal{D}}(\bX,G)P$. 
Then $Loc_{\bX}(M)\in \mathcal{C}^{wh}_{\bX/G}$. Indeed, this follows from \ref{example}.
\end{ex}

\begin{pro}\label{exact1}
Let 
\[0 \longrightarrow \mathcal{M}_1\longrightarrow \mathcal{M}_0\longrightarrow \mathcal{M}_2 \longrightarrow 0\]
be an exact sequence in $\mathcal{C}_{\bX/G}$. Then $\mathcal{M}_0$ is weakly holonomic   if and only if $\mathcal{M}_1, \mathcal{M}_2$ are weakly holonomic. The category $\mathcal{C}^{wh}_{\bX/G}$ is a full abelian subcategory of $\mathcal{C}_{\bX/G}$ closed under extensions.
\end{pro}
\begin{proof}
Let $\mathcal{U}$ be an admissible covering of $\bX$ by affinoid subdomains in $\bX_{w}(\mathcal{T})$. For every $\bU \in \mathcal{U}$, it follows from Proposition \ref{exact} that
\[d(\mathcal{M}_0(\bU))=\max\lbrace d(\mathcal{M}_1(\bU)), d(\mathcal{M}_2(\bU))\rbrace.\]
This implies all statements. \end{proof}
Recall the family of endofunctors $\mathcal{E}^{i}$ on $\mathcal{C}_{\bX/G}$, cf. Definition \ref{e-def}, for $i\geq 0$.
\begin{pro}\label{thm3}
The functor $\mathcal{E}^{d}$ takes values in the subcategory $\mathcal{C}^{wh}_{\bX/G}$.
\end{pro}
\begin{proof}
Since this is a local problem, we may assume that $(\bX,G)$ is small. Let $A=\mO(\bX)$. Let $\mathcal{M} \in \mathcal{C}_{\bX/G}$. Then $\mathcal{M} \simeq Loc_{\bX}(M)$, with $M=\mathcal{M}(\bX)\in \mathcal{C}_{\wideparen{\mathcal{D}}(\bX,G)}$. By Proposition \ref{iso2} and Theorem \ref{chanfunc}, one has that:
\[\mathcal{E}^d(\mathcal{M}) \simeq Loc_{\wideparen{\mathcal{D}}(\bX,G)}(Hom_A(\Omega(\bX), Ext^d_{\wideparen{\mathcal{D}}(\bX,G)}(M, \wideparen{\mathcal{D}}(\bX,G)))).\]
The grade of the coadmissible (right) module $Ext^d_{\wideparen{\mathcal{D}}(\bX,G)}(M, \wideparen{\mathcal{D}}(\bX,G))$  is $\geq d$, by the c-Auslander condition, so that its dimension is less than $d$. Now Proposition \ref{dimcom} implies
\[d(Hom_A(\Omega(\bX), Ext^d_{\wideparen{\mathcal{D}}(\bX,G)}(M, \wideparen{\mathcal{D}}(\bX,G)))) \leq d.\]
This shows that $\mathcal{E}^d(\mathcal{M})\in \mathcal{C}^{wh}_{\bX/G}$.
\end{proof}

\begin{lemma}\label{CM} Suppose that Bernstein's inequality holds in $\mathcal{C}_{\bX/G}$ and let 
$\mathcal{M} \in \mathcal{C}_{\bX/G}^{wh}$. Then $\mathcal{E}^i(\mathcal{M})=0$ for any $i \neq d$.  
\end{lemma}
\begin{proof} We may suppose $\mathcal{M} \neq 0$. Then 
$j(\mathcal{M})\geq d$ by Bernstein's inequality, whence $\mathcal{E}^i(\mathcal{M})=0$ for all $i<d$. So let $i>d$. Then $j(\mathcal{E}^i(\mathcal{M}))\geq i>d$ by the c-Auslander condition, so that 
$d(\mathcal{E}^i(\mathcal{M}))<d$. So $\mathcal{E}^i(\mathcal{M})=0$, by Bernstein's inequality. \end{proof}
The preceding lemma motivates the following definition. 
\begin{definition} Suppose that Bernstein's inequality holds in $\mathcal{C}_{\bX/G}$. The \textit{duality functor} $\mathbb{D}$ on $\mathcal{C}_{\bX/G}^{wh}$ into itself is defined as follows:
\[\mathbb{D}(\mathcal{M}):= \mathcal{E}^{d}(\mathcal{M})= \mathcal{H}om_{\mathcal{O}_{\bX}}(\Omega_{\bX}, E^{d}(\mathcal{M}))\]
for any $\mathcal{M} \in \mathcal{C}_{\bX/G}^{wh}$.
\end{definition}

\begin{pro}\label{main_2} Suppose that Bernstein's inequality holds in $\mathcal{C}_{\bX/G}$. There is a natural isomorphism of functors $\mathbb{D}^2\simeq {\rm id}$ on $\mathcal{C}^{wh}_{\bX/G}$.
\end{pro}
\begin{proof}
This can be proved along the lines of the proof of \cite[Proposition $7.3$]{AWIII} for weakly holonomic $\mathcal{D}_{\bX}$-modules. By the local nature of the functor $\mathbb{D}$ and since $\bX_{w}(\mathcal{T})$ is a basis for the Grothendieck topology on $\bX$, we may assume that $(\bX,G)$ is small and and need to show the existence of an isomorphism \begin{center}
$\Gamma(\bX, \mathbb{D}^2 (\mathcal{M})) \simeq \Gamma(\bX,\mathcal{M})$ 
\end{center}
compatible with restriction maps whenever $\bU\subset\bX$ in 
$\bX_{w}(\mathcal{T})$. We may assume that there are $G$-stable formal models $\mathcal{A}$ and $\mathcal{B}$ for $A=\mathcal{O}(\bX)$ and $B=\mathcal{O}(\bU)$ respectively, such that 
$\mathcal{O}(\bX)\rightarrow \mathcal{O}(\bU)$ maps $\mathcal{A}$ into $\mathcal{B}$. We choose a $G$-stable free $\mathcal{A}$-Lie lattice $\mathcal{L}$ for
$\mathcal{T}(\bX)$ such that $\mathcal{B}$ exhibits $\bU$ as $\mathcal{L}$-accessible.
We may assume that $[\mathcal{L},\mathcal{L}]\subseteq \pi\mathcal{L}$ and that 
$\mathcal{L}\cdot\mathcal{A}\subseteq\pi\mathcal{A}$.
 Choose a good chain $J_n$ for $\mathcal{L}$. Then 
\begin{center}
$S:=\wideparen{\mathcal{D}}(\bX,G)\simeq \varprojlim_n S_n$ with $S_n:=\widehat{U(\pi^n\mathcal{L})}_K \rtimes_{J_n} G$
\end{center}
and
\begin{center}
$T:=\wideparen{\mathcal{D}}(\bU,G)\simeq \varprojlim_n T_n$ with 
$T_n:=\widehat{U(\pi^n(\mathcal{B}\otimes_{\mathcal{A}}\mathcal{L}))}_K \rtimes_{J_n} G$
\end{center}
are realizations of the corresponding Fréchet-Stein structures of $S$ and $T$.
According to Cor. \ref{keycoro}, we may assume that $D_n$ and $T_n$ are Auslander-Gorenstein rings of dimension at most $2d$ for each $n\geq 0$. Write $M:= \Gamma(\bX,\mathcal{M}) \cong \varprojlim_n M_n$ with $M_n:=  S_n \otimes_S M $ and $N:= \Gamma(\bU,\mathcal{M}) \cong \varprojlim_n N_n$ with $N_n:=  T_n \otimes_T N.$ Note that $N_n=T_n \otimes_{S_n} M_n$.
Let $\Omega:=\Omega(\bX)$.
By Proposition \ref{dimcom}, one has:
\begin{align*}
\Gamma(\bX,\mathbb{D}^2 (\mathcal{M}))&= Hom_A(\Omega, Ext^d_S(Hom_A(\Omega,Ext^d_{S}(M,S)),S)\\
&\simeq Ext^d_S(\Omega \otimes_A Hom_A(\Omega, Ext^d_S(M,S)),S)\\
&\simeq Ext^d_S(Ext^d_S(M,S),S).
\end{align*}
According to \cite[Lemma 8.4]{ST2003}, one has $Ext^d_S(M,S) \cong \varprojlim_n Ext^d_{S_n}(M_n, S_n)$. Moreover, Lemma \ref{CM} implies that 
$Ext^{i}_{S_n}(M_n, S_n)=0$ for any $i\neq d$. The classical duality over Auslander-Gorenstein rings  \cite[Theorem $4$]{Iwa} therefore gives the $S_n$-linear isomorphism

$$Ext^d_{S_n}(Ext^d_{S_n}(M_n,{S}_n),{S}_n))\simeq M_n.$$
Recall here that the duality morphism comes from the usual convergent spectral sequence
with $E_2$-term 
$$E^{l,m}_2:=Ext^l_{S_n}(Ext^{-m}_{S_n}(M_n,S_n),S_n)$$
and abutment $E^{l+m}:=M_n$ for $l+m=0$, resp. $:=0$ for $l+m\neq 0$,
appearing in the global duality $M_n\rightarrow R\Hom_{S_n}(M_n,S_n)$. We have the commutative diagram
\[\begin{tikzcd}
S_{n} \arrow{r}{} \arrow[swap]{d}{} & T_{n}\arrow{d}{} \\
S_{n-1} \arrow{r}{} & T_{n-1}\\
\end{tikzcd}\]
in which all ring homomorphism are flat: the horizontal ones according to \cite[Thm. 4.3.14]{AW}. The base extension $S_{n-1}\otimes_{S_{n}} (\cdot)$ transforms the whole spectral sequence and hence the duality morphism for the $S_n$-module $M_n$ into the corresponding spectral sequence and duality morphism for the $S_{n-1}$-module $M_{n-1}=S_{n-1}\otimes_{S_{n}} M_n$. The base extension $T_{n}\otimes_{S_{n}} (\cdot)$ transforms the whole spectral sequence and hence the duality morphism for the $S_n$-module $M_n$ into the corresponding spectral sequence and duality morphism for the $T_{n}$-module $N_n=T_{n}\otimes_{S_{n}} M_n$. This means that we may pass the above $S_n$-linear isomorphism to the projective limit in $n$ to obtain the isomorphism
$$\Gamma(\bX, \mathbb{D}^2 (\mathcal{M})) = Ext^d_S(Ext^d_S(M,S),S)\simeq M= \Gamma(\bX,\mathcal{M})$$
and that this latter isomorphism is indeed compatible with the restriction maps to $\bU\subset\bX$. 
\end{proof}

\section{Preservation of weak holonomicity and examples}\label{section_six}
\subsection{Extension.}
We generalize the extension functor \cite[Section $7.2$]{AWIII} to the equivariant setting. This functor is defined on the category of $G$-equivariant coherent $\mathcal{D}_{\bX}$-module and takes values in the category $\mathcal{C}_{\bX/G}$. We will show that it preserves weak holonomicity in a suitable sense.\\
\\
Let $\bX$ be a smooth rigid analytic space and $G$  a $p$-adic Lie group acting continuously on $\bX$.
\begin{lemma}\label{extension} Suppose that $(\bX,G)$ is small and that $H$ is an open subgroup of $G$. The natural map
\[\wideparen{\mathcal{D}}(\bX,H) \otimes_{\mathcal{D}(\bX)\rtimes H} (\mathcal{D}(\bX) \rtimes G)\longrightarrow \wideparen{\mathcal{D}}(\bX,G) \]
is an isomorphism. If $M$ is a $\mathcal{D}(\bX)\rtimes G$-module, the natural morphism
\[\wideparen{\mathcal{D}}(\bX,H) \otimes_{\mathcal{D}(\bX)\rtimes H} M \tilde{\longrightarrow} \wideparen{\mathcal{D}}(\bX,G) \otimes_{\mathcal{D}(\bX)\rtimes G} M \]
is bijective.
\end{lemma}
\begin{proof}
Following \cite[Proposition $3.4.10$]{AW} there is a bijection
$\wideparen{\mathcal{D}}(\bX,H) \otimes_{K[H]}K[G]\simeq \wideparen{\mathcal{D}}(\bX,G),$
which factors into 
\[\wideparen{\mathcal{D}}(\bX,H) \otimes_{K[H]}K[G]\longrightarrow \wideparen{\mathcal{D}}(\bX,H) \otimes_{\mathcal{D}(\bX)\rtimes H} (\mathcal{D}(\bX) \rtimes G)\longrightarrow \wideparen{\mathcal{D}}(\bX,G).\]
The first morphism is surjective, which implies that both morphisms are in fact bijective. This implies the first claim of the lemma. The second claim follows from this. 
\end{proof}
\begin{pro} Suppose that $(\bX,G)$ is small. Let $M$ be a $\mathcal{D}(\bX)\rtimes G$-module which is coherent as a $\mathcal{D}(\bX)$-module. Then the tensor product
\[ \wideparen{M}:= \wideparen{\mathcal{D}}(\bX,G) \otimes_{\mathcal{D}(\bX)\rtimes G}M\]
is a coadmissible $\wideparen{\mathcal{D}}(\bX,G)$-module. 
\end{pro}
\begin{proof}
Since $G$ is compact $p$-adic Lie group, it is topologically finitely generated. As $M$ is finitely presented as a $\mathcal{D}(\bX)$-module, it follows that the $\wideparen{\mathcal{D}}(\bX,G)$-module $ \wideparen{\mathcal{D}}(\bX,G) \otimes_{\mathcal{D}(\bX)}M$ is coadmissible \cite[Corollary $3.4(v)$]{ST2003}. Now, let $g_1, g_2,...,g_r$ be a set of topological generators for $G$ and $m_1,...,m_s\in M$ which generate $M$ as a $\mathcal{D}(\bX)$-module and let $I$ be the $\wideparen{\mathcal{D}}(\bX,G)$-submodule generated by the finite set $\lbrace g_i \otimes m_j - 1\otimes g_im_j \rbrace$. Then $I$ is a coadmissible $\wideparen{\mathcal{D}}(\bX,G)$-module by \cite[Corollary $3.4(iv)$]{ST2003}. There is a surjective map
$$f: \wideparen{\mathcal{D}}(\bX,G) \otimes_{\mathcal{D}(\bX)}M \longrightarrow \wideparen{\mathcal{D}}(\bX,G) \otimes_{\mathcal{D}(\bX)\rtimes G}M.$$
We will show that $I=\ker(f)$, which proves the proposition. The inclusion $I\subseteq \ker(f)$ is clear. Moreover, $\ker(f)$ is generated over $\mathcal{D}(\bX,G)$ by the elements $g\otimes m-1\otimes gm$ with $g\in G, m\in M$.
Let $ x \in L=\mathcal{T}(\bX) $. Then $g_ixg_i^{-1}=g_i.x$ in $\mathcal{D}(\bX)\rtimes G=U(L)\rtimes G$, so that we can compute as follows:
\begin{center}
$g_i\otimes xm_j-1\otimes g_ixm_j=(g_ixg_i^{-1})g_i\otimes m_j -1\otimes (g_ixg_i^{-1})g_im_j= (g_i.x)g_i\otimes m_j-1\otimes (g_i.x)g_im_j=(g_i.x)(g_i\otimes m_j-1\otimes g_im_j)$
\end{center}
Hence $I$ contains all elements of the form $g_i\otimes m-1\otimes g_im$ with $m \in M$. Now, let $g \in G$ and $(g_n) \in \langle g_1,...,g_r\rangle$ such that $\lim g_n=g$. Note that the coadmissible module $\wideparen{\mathcal{D}}(\bX,G)\otimes_{\mathcal{D}(\bX)}M$ has a natural Fr\'{e}chet topology such that the map $\wideparen{\mathcal{D}}(\bX,G) \longrightarrow \wideparen{\mathcal{D}}(\bX,G)\otimes_{\mathcal{D}(\bX)}M$ is continuous. this implies:
$$\lim (g_n \otimes m-1\otimes g_nm )=g\otimes m -1\otimes g. $$
Here we note that $G \subset \wideparen{\mathcal{D}}(\bX,G)$. Since $I$ is a closed subspace of $\wideparen{\mathcal{D}}(\bX,G)\otimes_{\mathcal{D}(\bX)}M$ \cite[Lemma $3.6$]{ST2003}, we obtain $g\otimes m-1\otimes gm\in I$ for any $g \in G$ and $m\in M$. Thus $I=\ker(f)$.
\end{proof}
Let $\mathcal{M}$ be a $G$-equivariant $\mathcal{D}_{\bX}$-module which is coherent as a $\mathcal{D}_{\bX}$-module. Then we define the presheaf $E_{\bX/G}(\mathcal{M})$ on $\bX_{w}(\mathcal{T})$ as follows. Let $\bU \in \bX_{w}(\mathcal{T})$. Then define
\[E_{\bX/G}(\mathcal{M})(\bU):= \varprojlim_H\wideparen{\mathcal{D}}(\bU,H)\otimes_{\mathcal{D}(\bU) \rtimes H}\mathcal{M}(\bU)\]
where the inverse limit is taken over the set of all $\bU$-small subgroups $H$ of $G$. 
\begin{pro}
The presheaf $E_{\bX/G}(\mathcal{M})$ extends to a coadmissible $G$-equivariant  $\mathcal{D}_{\bX}$-module (still denoted by $E_{\bX/G}(\mathcal{M})$).
\end{pro}
\begin{proof}
For every $\bU\in \bX_{w}(\mathcal{T})$, there is a $\bU$-small open subgroup $H$ of $G$.
To verify the sheaf property on $\bX_{w}(\mathcal{T})$, we may therefore assume that $(\bX,G)$ is small. Denote 
\begin{center}
$M:= \mathcal{M}(\bX)$ and $\wideparen{M}=\wideparen{\mathcal{D}}(\bX,G)\otimes_{\mathcal{D}(\bX)\rtimes G}M$. 
\end{center}
Let $\bU \in \bX_{w}(\mathcal{T})$. Then
\[\wideparen{\mathcal{D}}(\bU,H)\wideparen{\otimes}_{\wideparen{\mathcal{D}}(\bX,H)}\wideparen{M}\cong \wideparen{\mathcal{D}}(\bU,H)\wideparen{\otimes}_{\wideparen{\mathcal{D}}(\bX,H)} \wideparen{\mathcal{D}}(\bX,H) \otimes_{\mathcal{D}(\bX)\rtimes H}M\cong \wideparen{\mathcal{D}}(\bU,H)\otimes_{\mathcal{D}(\bX)\rtimes H}M.\]
Furthermore, since  $\mathcal{M}$ is a coherent $\mathcal{D}_{\bX}$-module, one has  
\[\mathcal{M}(\bU)\cong \mathcal{D}(\bU)\otimes_{\mathcal{D}(\bX)}M.\]
Consequently, $\mathcal{M}(\bU) \cong \mathcal{D}(\bU) \otimes_{\mathcal{D}(\bX)}M \cong \mathcal{D}(\bU)\rtimes H \otimes_{\mathcal{D}(\bX)\rtimes H}M$, whence
\[E_{\bX/G}(\mathcal{M})(\bU)\cong \wideparen{\mathcal{D}}(\bU,H)\otimes_{\mathcal{D}(\bU)\rtimes H}\mathcal{M}(\bU)\cong \wideparen{\mathcal{D}}(\bU,H)\otimes_{\mathcal{D}(\bX)\rtimes H}M \cong \wideparen{\mathcal{D}}(\bU,H)\wideparen{\otimes}_{\wideparen{\mathcal{D}}(\bX,H)}\wideparen{M}.\]
This proves  $E_{\bX/G}(\mathcal{M}) \cong Loc_{\bX}(\wideparen{M}),$ which implies the proposition. \end{proof}
Let ${\rm Coh}(G-\mathcal{D}_{\bX})$ be the category of $G$-equivariant coherent $\mathcal{D}_{\bX}$-modules. 
\begin{coro} The formation of $E_{\bX/G}(\mathcal{M})$ is a functor
$$E_{\bX/G}: {\rm Coh}(G-\mathcal{D}_{\bX})\longrightarrow \mathcal{C}_{\bX/G}.$$
\end{coro}

Recall the dimension of a coherent $\mathcal{D}_{\bX}$-module $\mathcal{M}$ on the smooth rigid analytic variety $\bX$, cf. \cite{MN}. The module $\mathcal{M}$ is said to be of {\it minimal dimension} if its dimension is not greater than $\dim \bX$.
\begin{pro}\label{min_ext}
Let $\mathcal{M}\in {\rm Coh}(G-\mathcal{D}_{\bX})$ be of minimal dimension. Then $E_{\bX/G}(\mathcal{M})\in  \mathcal{C}^{wh}_{\bX/G}.$ 
\end{pro}
\begin{proof}
Since the question is local, we may assume that $(\bX,G)$ is small. Choose a $G$-stable free Lie lattice $\mathcal{L}$ and a good chain $(J_n)_n$ for $\mathcal{L}$ such that $\wideparen{\mathcal{D}}(\bX,G)\cong \widehat{U(\pi^n\mathcal{L})}_K \rtimes_{J_n}G$. Write $D:= \mathcal{D}(\bX)$, $\wideparen{{D}}:= \wideparen{\mathcal{D}}(\bX,G)$, $D_n:= \widehat{U(\pi^n\mathcal{L})}_K \rtimes_{J_n}G$,  $M:= \mathcal{M}(\bX)$ and let $d:= \dim \bX$. Recall that by definition $\wideparen{M}= \wideparen{{D}}\otimes_{D \rtimes G}M$. By assumption $d(M)\leq d$. Now $\wideparen{{D}}$-module $\wideparen{{D}} \otimes_{D}M$ is coadmissible nd there is a $n$ sufficiently large, such that $j_{D_n}(D_n\otimes_D M)=j_{\wideparen{{D}}}(\wideparen{{D}}\otimes_D M)$. As we know that $D_n$ is flat over $D$, it follows that
\[Ext^i_{D}(M, D) \otimes_{D} D_n \cong Ext^i_{D_n}(D_n \otimes_D M,D_n).\]
Thus $j_{D_n}(D_n\otimes_D M) \geq j_D(M)$, which implies
 $$d(\wideparen{{D}}\otimes_DM)=d(D_n\otimes_DM) \leq d(M)\leq d.$$
Since $\wideparen{{D}}\otimes_DM$ surjects onto $\wideparen{M}$, this proves
$d(\wideparen{M}) \leq d$.
\end{proof}
Here is the link to classical holonomic algebraic $D$-modules when $\bX$ is algebraic. Let for the rest of this subsection $\mathbb{X}$ be a smooth $K$-scheme of locally finite type and $\bX=\mathbb{X}^{an}$ be its rigid analytification. Consider the induced morphism of locally ringed $\mathit{G}$-spaces $\rho: \bX \rightarrow \mathbb{X}$ and the pull-back functor  
$$\rho^\ast\mathcal{M}=\mathcal{O}_{\bX} \otimes_{\rho^{-1}\mathcal{O}_\mathbb{X}}\rho^{-1}\mathcal{M}$$ on $\mathcal{O}_\mathbb{X}$-modules. 
Since $\rho^\ast \mathcal{D}_\mathbb{X}=\mathcal{D}_{\bX}$, this restricts to a functor from (coherent) $\mathcal{D}_\mathbb{X}$-modules to (coherent) $\mathcal{D}_{\bX}$-modules.\\
\begin{lemma} \label{indlemma} Let $\mathcal{M}$ be a holonomic $\mathcal{D}_\mathbb{X}$-module. Then $\rho^\ast\mathcal{M}$ is a $\mathcal{D}_{\bX}$-module of minimal dimension.
\begin{proof}
Let $\mathcal{M}$  be a nonzero holonomic $\mathcal{D}_\mathbb{X}$-module. Let $U$ be an open affine subdomain of $\mathbb{X}$ over which $\mathcal{M}$ is nonzero. Let $\mathcal{U}=\lbrace \bU_i , i \in I\rbrace$ be an admissible covering of $\rho^{-1}U$ by affinoid subdomains of $\bX$. As $\rho$ is flat, we may suppose that $\mathcal{O}_\mathbb{X}(U) \rightarrow \mathcal{O}_{\bX}(\bU_i)$ is flat, whence $\mathcal{D}_\mathbb{X}(U) \rightarrow \mathcal{D}_{\bX}(\bU_i)$ is flat. 
Since $\rho^\ast \mathcal{M}(\bU_i) =  \mathcal{D}_{\bX}(\bU_i)\otimes_{\mathcal{D}_\mathbb{X}(U)}\mathcal{M}(U),$ one obtains
\[Ext^n_{\mathcal{D}_{\bX}(\bU_i)}(\rho^\ast \mathcal{M}(\bU_i), \mathcal{D}_{\bX}(\bU_i))\cong Ext^n_{\mathcal{D}_\mathbb{X}(U)}(\mathcal{M}(U), \mathcal{D}_\mathbb{X}(U))\otimes_{\mathcal{D}_\mathbb{X}(U)}\mathcal{D}_{\bX}(\bU_i).\]
By consequence 
$$j_{\mathcal{D}_\mathbb{X}(U)}(\mathcal{M}(U))\leq j_{\mathcal{D}_{\bX}(\bU_i)}(\rho^\ast\mathcal{M}(\bU_i)).$$
Since $\mathcal{M}$ is holonomic and nonzero over $U$, we have $j_{\mathcal{D}_\mathbb{X}(U)}(\mathcal{M}(U))= \dim \mathbb{X} $, which implies that $ \dim \mathbb{X} \leq j_{\mathcal{D}_{\bX}(\bU_i)}(\rho^\ast\mathcal{M}(\bU_i))$ and $d_{\mathcal{D}_{\bX}(\bU_i)}(\rho^\ast\mathcal{M}(\bU_i))\leq \dim \mathbb{X}=\dim \bX$ for every $i$. Letting $U$ vary implies $d(\rho^\ast\mathcal{M})\leq \dim \bX$, as claimed.
\end{proof}
\end{lemma}

\subsection{Equivariant integrable connections}
Let $\bX$ be a smooth rigid analytic variety and $G$ be a $p$-adic Lie group which acts continuously on $\bX$. A $G$-equivariant $\mathcal{D}_{\bX}$-module, which is coherent as an 
$\mathcal{O}_{\bX}$-module will be called a \textit{$G$-equivariant integrable connection}. The subcategory of ${\rm Coh}(G-\mathcal{D}_{\bX})$ consisting of the 
$G$-equivariant integrable connections on $\bX$ is denoted by ${\rm Con}(G-\mathcal{D}_{\bX})$.
Of course, ${\rm Con}(G-\mathcal{D}_{\bX})$ contains the structure sheaf $\mathcal{O}_{\bX}$.
\begin{pro}\label{Frech} Let $\mathcal{M}$  be a $G$-equivariant integrable connection on $\bX$. Then $\mathcal{M}\in {\rm Frech}(G-\mathcal{D}_{\bX})$.
\end{pro}
\begin{proof}
Let $\bU \in \bX_{w}(\mathcal{T})$ be an affinoid subdomain. Then $\mathcal{M}\vert_{\bU}$ is a coherent $\mathcal{O}_{\bU}$-module, so that by Kiehl's theorem, $\mathcal{M}(\bU)$ is a coherent $\mathcal{O}(\bU)$-module. this implies $\mathcal{M}(\bU)$ has a canonical Banach topology by \cite[Chapter $3$, Proposition $3.7.3.3$]{BoRe}. For any $g \in G$, the map 
$$g^\mathcal{M}(\bU): \mathcal{M}(\bU)\longrightarrow \mathcal{M}(g\bU)$$
is a bijection which is linear with respect to the continuous morphism of $K$-Banach algebras $g^\mathcal{O}(\bU): \mathcal{O}(\bU) \longrightarrow \mathcal{O}(g\bU)$. If we consider the $\mathcal{O}(g\bU)$-module $\mathcal{M}(g\bU)$ as a $\mathcal{O}(\bU)$, then $\mathcal{M}(g\bU)$ is  coherent as a $\mathcal{O}(\bU)$-module such that the map $\mathcal{O}(\bU) \times \mathcal{M}(g\bU) \longrightarrow \mathcal{M}(g\bU)$ is continuous
and $g^\mathcal{M}(\bU)$ is a $\mathcal{O}(\bU)$-linear map. By \cite[Chapter $3$, Proposition $3.7.3.2$]{BoRe}, $g^\mathcal{M}(\bU)$ is continuous between Banach spaces. Since Banach spaces are Fr\'{e}chet spaces, this proves that $\mathcal{M} \in {\rm Frech}(G-\mathcal{D}_{\bX})$. 
\end{proof}

\begin{lemma}\label{ext} Suppose that $(\bX,G)$ is small. Let $M$ be a $\mathcal{D}(\bX)\rtimes G $-module which is coherent over $\mathcal{O}(\bX)$. Let $\mathcal{L}$ be a $G$-stable $\mathcal{A}$-Lie lattice $\mathcal{L}$ in $\mathcal{T}(\bX)$ for some $G$-stable affine formal model $\mathcal{A}$ of $\mathcal{O}(\bX)$. Then there exists $m\geq 0$ such that there is a structure of $\widehat{U(\pi^n\mathcal{L})_K}\rtimes G$-module on $M$ for all $n \geq m$ which extends the given $\mathcal{D}(\bX)\rtimes G $-action.
\end{lemma}
\begin{proof} 
Let $A=\mathcal{O}(\bX)$. 
By assumption $M$ is finitely generated as an $A$-module. Let $S$ be a finite generating set of $M$ on $A$. Then $\mathcal{M}:=\mathcal{A}S$ is an $\mathcal{A}$-submodule of $M$ which generates $M$ over $K$. Furthermore, since  $\mathcal{L}$ is an $\mathcal{A}$-Lie lattice by assumption, there exists $m \geq 0$ such that for all $n\geq m$, $\pi^n\mathcal{L}. \mathcal{M} \subset \mathcal{M}$, forcing $\mathcal{M}$ to be an $U(\pi^n\mathcal{L})$-module. Now, since $\mathcal{A}$ is $\pi$-adically complete, so is the finitely generated $\mathcal{A}$-module 
$\mathcal{M}$, so that $\mathcal{M}$ is also an $\widehat{U(\pi^n\mathcal{L})}$-module. Therefore, $M\cong K\otimes \mathcal{M}$ is a $\widehat{U(\pi^n\mathcal{L})}_K$-module. On the other hand, we see that the structure of $\widehat{U(\pi^n\mathcal{L})}_K$-module (which extends the given $\mathcal{D}(\bX)$-action) on $M$ is compatible with the $G$-action. 
\end{proof}
\begin{pro} \label{exam1} Suppose that $(\bX,G)$ is small. Let $M$ be a $\mathcal{D}(\bX)\rtimes G $-module which is coherent over $\mathcal{O}(\bX)$. Let $\mathcal{L}$ be a $G$-stable free $\mathcal{A}$-Lie lattice $\mathcal{L}$ for some $G$-stable affine formal model $\mathcal{A}$ of $\mathcal{O}(\bX)$. The $\mathcal{D}(\bX)\rtimes G$-action on $M$ extends to a $\wideparen{\mathcal{D}}(\bX,G)$-module structure if the following condition holds in $M$:
\begin{equation}\label{stronge}
g.m=\beta_{\pi^n\mathcal{ L}}(g)m \,\,\, \text{for all}\,\,\, m\in M  \,\,\text{and}\,\, g\in J_n.
\end{equation}
for some good chain $(J_\bullet)$ for $\mathcal{L}$.
In this case, $M\simeq \wideparen{\mathcal{D}}(\bX,G)\otimes_{\mathcal{D}(\bX)\rtimes G}M$ and $M$ is a coadmissible $\wideparen{\mathcal{D}}(\bX,G)$-module. 
\end{pro}
\begin{proof}
By Lemma \ref{ext}, the module $M$ is a 
$\widehat{U(\pi^n\mathcal{L})}_K\rtimes G$-module for sufficiently large $n$. Clearly, $M$ is a $\wideparen{\mathcal{D}}(\bX,G)$-module, if and only if the $\widehat{U(\pi^n\mathcal{L})}_K\rtimes G$-module structure factors through the quotient
$\widehat{U(\pi^n\mathcal{L})_K}\rtimes_{J_n} G$ for some good chain $(J_\bullet)$ for $\mathcal{L}$. Hence, if and only if the condition (\ref{stronge}) holds. In this case, the natural morphism $i: M \longrightarrow \widehat{U(\pi^n\mathcal{L})}_K \rtimes G \otimes_{U(L) \rtimes G}M$ is an isomorphism. Indeed, the $ \widehat{U(\pi^n\mathcal{L})}_K \rtimes G$-linear map 
\begin{align*}
j:\,\,\,\, \widehat{U(\pi^n\mathcal{L})}_K \rtimes G \otimes_{U(L) \rtimes G}M &\longrightarrow M\\
a \otimes m &\longmapsto am
\end{align*}
satisfies $j \circ i = id_M$ and $i$ is injective. For the surjectivity of $i$, we note that the ring $\widehat{U(\pi^n\mathcal{L})}_K \rtimes G$ (resp. $U(L) \rtimes G$) consists of elements of the form $\sum a_ig_i$, where the sum is finite  with $g_i \in G$ and $a_i \in \widehat{U(\pi^n\mathcal{L})}_K$ (resp. $a_i \in U(L)$). As a consequence, the map $ \widehat{U(\pi^n\mathcal{L})}_K \otimes_{U(L)}M \longrightarrow \widehat{U(\pi^n\mathcal{L})}_K \rtimes G \otimes_{U(L) \rtimes G}M$ is surjective. On the other hand, the map $i$ factors through
\begin{center}
$M \tilde{\longrightarrow} \widehat{U(\pi^n\mathcal{L})}_K \otimes_{U(L)}M \longrightarrow \widehat{U(\pi^n\mathcal{L})}_K \rtimes G \otimes_{U(L) \rtimes G}M$.
\end{center}
Where the first map is an isomorphism by \cite[Lemma $7.2$]{AWII}. Therefore $i$ is surjective, so it is an isomorphism as claimed. As a consequence, this proves that the canonical morphism 
\[M \longrightarrow \widehat{U(\pi^n\mathcal{L})}_K \rtimes_{J_n}G \otimes_{U(L)\rtimes G} M\]
is an isomorphism of $\widehat{U(\pi^n\mathcal{L})}_K \rtimes_{J_n}G$-modules, as $M$ is also an $\widehat{U(\pi^n\mathcal{L})}_K \rtimes_{J_n} G$-module. Passing to the projective limit over all $n$ shows that $M\simeq \wideparen{\mathcal{D}}(\bX,G)\otimes_{\mathcal{D}(\bX)\rtimes G}M$.
\end{proof}
\begin{re} In the situation of the preceding proposition, one has  $\mathcal{M}\simeq E_{\bX/G}(\mathcal{M}),$ for $\mathcal{M}$ the coadmissible equivariant module associated with $M$.
\end{re}
\begin{definition} An equivariant integrable connection $\mathcal{M}\in {\rm Con}(G-\mathcal{D}_{\bX})$ is called {\it strongly equivariant}
if the $\mathcal{D}(\bU)\rtimes H $-module $\mathcal{M}(\bU)$ satisfies the condition (\ref{stronge}) for every $(\bU,H)$ small.
\end{definition} 
\begin{re}The condition (\ref{stronge}) is the analogue in our context 
of the condition \cite[Prop.2.6]{VandenBerghED} appearing in the classical theory of equivariant algebraic $\mathcal{D}$-modules and marking there the difference between "weakly equivariant" and "equivariant". 
\end{re}

\begin{pro} Every strongly equivariant $\mathcal{M}\in {\rm Con}(G-\mathcal{D}_{\bX})$ 
lies in $\mathcal{C}^{wh}_{\bX/G}$. 
\end{pro}
\begin{proof}  Any $G$-equivariant integrable connection lies in 
${\rm Frech}(G-\mathcal{D}_{\bX})$, according to \ref{Frech}.
By \ref{exam1} and the subsequent remark, the module $\mathcal{M}$ coincides locally on small affinoids with its extension to 
$\mathcal{C}_{\bX/G}$. Since the extension takes modules of minimal dimension into $\mathcal{C}^{wh}_{\bX/G}$, cf. \ref{min_ext}, the proposition follows. 
\end{proof} 

\begin{coro}\label{structure_sheaf} The structure sheaf $\mathcal{O}_{\bX}$ is strongly equivariant and thus lies in $\mathcal{C}^{wh}_{\bX/G}.$
\end{coro}
\begin{proof}
Without loss of generality, we may suppose that $(\bX,G)$ is small. Let $\mathcal{A}$ be a $G$-stable affine formal model of $A=\mathcal{O}(\bX)$ and $\mathcal{L}$ is a $G$-stable  $\mathcal{A}$-Lie lattice in $Der_K(A)$. Recall that each $g \in G$ acts on $\mathcal{A}$ via the morphism of groups $\rho: G \longrightarrow Aut(\mathcal{A})$ and on $\mathcal{L}$ via
\begin{center}
$g.x:= \rho(g)\circ x \circ \rho(g^{-1})$ for all $x \in \mathcal{L}$.
\end{center}
Now if $g \in G_\mathcal{L}$, we can write $\rho(g)=\exp(p^\epsilon x)$ with $x \in \mathcal{L}$. Then for $a \in \mathcal{A}$,
\begin{center}
$\beta_{\mathcal{L}}(g).a= \exp(p^\epsilon \iota(x)).a= \sum_n \frac{p^{\epsilon n}}{n!}\iota(x)^n.a= \sum_n \frac{p^{\epsilon n}}{n!}x^n.a=\exp(p^\epsilon x)(a)=\rho(g)(a)=g.a.$
\end{center}
This proves that $\beta_{\mathcal{L}}(g)-g$ acts trivially on $\mathcal{A}$. Since 
$J_n\subseteq G_{\pi^n\mathcal{L}}$ for any good chain $(J_\bullet)$ for $\mathcal{L}$, we see that the condition (\ref{stronge}) holds for any such chain.
\end{proof}

 
 \subsection{Weak holonomicity and push-forward }\label{section_push}

In this section, we prove a dimension formula for the equivariant pushforward functor, generalizing 
\cite[Thm. 6.1]{AWIII} to the equivariant setting. It implies that the equivariant Kashiwara equivalence \cite{AWG} descends to weakly holonomic modules. 

\vskip5pt 

We start with the affinoid situation. Suppose that $i: \,\, \bY=Sp(A/I) \rightarrow \bX=Sp(A)$ is a closed embedding of {\it smooth} affinoid varieties and $G$ be a compact $p$-adic Lie group which acts continuously on $\bX$ and stabilizes $\bY$. We suppose that:
\begin{itemize}
\item[$(a)$] $\mathcal{T}(\bX)$ admits a free $G$-stable $\mathcal{A}$-Lie lattice $\mathcal{L}=\mathcal{A}\partial_1\oplus...\oplus \mathcal{A}\partial_d$ for some $G$-stable affine formal model $\mathcal{A}\subset \mathcal{O}(\bX)$ and such that $[\mathcal{L},\mathcal{L}]\subset \pi \mathcal{L}$, $\mathcal{L}.\mathcal{A}\subset \pi \mathcal{A}$,
\item[$(b)$] $\lbrace \partial_1,...,\partial_d\rbrace$ is an $I$-standard basis with respect to a generating set $\lbrace f_1,...,f_r \rbrace \subset I.$
\end{itemize}
Recall from \cite{AWG} the equivariant pushforward functor $i_+: \mathcal{C}_{\wideparen{\mathcal{D}}(\bY,G)}^r \rightarrow \mathcal{C}^r_{\wideparen{\mathcal{D}}(\bX,G)}$ defined by
\[i_{+} N:= N \wideparen{\otimes}_{\wideparen{\mathcal{D}}(\bY,G)}\wideparen{\mathcal{D}}(\bX,G)/I\wideparen{\mathcal{D}}(\bX,G).\]

 We explain why $i_+ N$ is indeed a  coadmissible (right) $\wideparen{\mathcal{D}}(\bX,G)$-module, thereby fixing some notation for future use. So let $\mathcal{I}:=I\cap \mathcal{A}$ and consider the integral normalizer 
 \[\mathcal{N}_{\mathcal{L}}(\mathcal{I}):= \lbrace x \in \mathcal{L} : x(\mathcal{I})\subset \mathcal{I} \rbrace.\]
 Then $\mathcal{N}:= \mathcal{N}_{\mathcal{L}}(\mathcal{I})/\mathcal{I}\mathcal{L}$ is a $G$-stable
  $\mathcal{A}/\mathcal{I}$-Lie lattice in $\mathcal{T}(\bY)= (A/I)\partial_{r+1}\oplus...\oplus (A/I)\partial_d.$
  Thus, for a good chain $(J_n)_n$ of $G$, we have 
  \begin{center}
  $\wideparen{\mathcal{D}}(\bX,G) \cong \varprojlim_n \widehat{U(\pi^n\mathcal{L})}_K \rtimes_{J_n}G$ and $\wideparen{\mathcal{D}}(\bY,G) \cong \varprojlim_n \widehat{U(\pi^n\mathcal{N})}_K \rtimes_{J_n}G$
  \end{center}
  (note that $G_\mathcal{N}\subset G_\mathcal{L}$ by \cite[Lemma $4.3.2$]{AWG} so we can choose a good chain of $G$ such that each $J_n$ is contained in $G_\mathcal{N}$).
  Write $T_n:=\widehat{U(\pi^n\mathcal{L})}_K \rtimes_{J_n}G$ and $S_n:=\widehat{U(\pi^n\mathcal{N})}_K \rtimes_{J_n}G$, then 
  \begin{center}
  $i_+ N\cong \varprojlim_n N_n\otimes_{S_n}T_n/IT_n$, with $N_n=N\otimes_{\wideparen{\mathcal{D}}(\bY,G)}S_n$.
\end{center}  


\begin{pro}\label{dimest} Given $N\in \mathcal{C}_{\wideparen{\mathcal{D}}(\bX,G)}^r$, one has

\[d_{\wideparen{\mathcal{D}}(\bX,G)}(i_+N)=d_{\wideparen{\mathcal{D}}(\bY,G)}(N) +\dim A -\dim A/I.\]

\end{pro}
\begin{proof}
Since $i_+N$ is a coadmissible $\wideparen{\mathcal{D}}(\bX,G)$-module, there exist $n$ sufficiently large such that
\[j_{\wideparen{\mathcal{D}}(\bX,G)}(i_+N)=j_{T_n}(i_+N\otimes_{\wideparen{\mathcal{D}}(\bX,G)}T_n)=j_{\widehat{U(\pi^n\mathcal{L})}_K }(i_+N\otimes_{\wideparen{\mathcal{D}}(\bX,G)}T_n).\]
Here, the last equality follows from Proposition \ref{keypro} and Lemma \ref{keylemma2}. Note that:
\begin{align*}
i_+N\otimes_{\wideparen{\mathcal{D}}(\bX,G)}T_n = N_n\otimes_{S_n}T_n/IT_n,
\end{align*}
where $N \cong \varprojlim_nN_n$ with  $N_n\otimes_{\wideparen{\mathcal{D}}(Y,G)}S_n$. Furthermore
\begin{align*}
N_n\otimes_{S_n}T_n/IT_n&= N_n\otimes_{\widehat{U(\pi^n\mathcal{N})}_K \rtimes_{J_n}G}\widehat{U(\pi^n\mathcal{L})}_K \rtimes_{J_n}G/I(\widehat{U(\pi^n\mathcal{L})}_K \rtimes_{J_n}G)\\
&\cong N_n\otimes_{\widehat{U(\pi^n\mathcal{N})}_K}\widehat{U(\pi^n\mathcal{L})}_K/I\widehat{U(\pi^n\mathcal{L})}_K.
\end{align*} 
On the other hand, since $\mL$ is a flat $\mR$-module, multiplication by $\pi^n$ yields an isomorphism of the $\mathcal{A}/\mathcal{I}$-Lie lattice $\pi^n\mathcal{N}$ of $\mathcal{T}(\bY)$ with $N_{\pi^n\mathcal{L}}(\mathcal{I})/\mathcal{I}(\pi^n\mathcal{L})$. According to \cite[5.2]{AWII} one has $\widehat{U(\pi^n\mathcal{N})}_K \cong \widehat{U(\mathcal{C}_n)}_K/I\widehat{U(\mathcal{C}_n)}_K$, where 

$$\mathcal{C}_n:=C_{\pi^n\mathcal{L}}(F)= \lbrace x\in \pi^n\mathcal{L}: x.f=0 \,\,\, \forall f \in F\rbrace$$ denotes  the centraliser of $F$ in $\pi^n\mathcal{L}$. So we have
\[N_n\otimes_{\widehat{U(\pi^n\mathcal{N})}_K}\widehat{U(\pi^n\mathcal{L})}_K/I\widehat{U(\pi^n\mathcal{L})}_K \cong N_n\otimes_{\widehat{U(\mathcal{C}_n)}_K}\widehat{U(\pi^n\mathcal{L})}_K.\]
Hence applying \cite[Prop. $6.1$]{AWIII} gives
\begin{align*}
j_{\widehat{U(\pi^n\mathcal{L})}_K }(i_+N\otimes_{\wideparen{\mathcal{D}}(\bX,G)}T_n)&=j_{\widehat{U(\pi^n\mathcal{L})}_K }(N_n\otimes_{\widehat{U(\mathcal{C}_n)}_K}\widehat{U(\pi^n\mathcal{L})}_K)\\
&=j_{\widehat{U(\mathcal{C}_n)}_K/F\widehat{U(\mathcal{C}_n)}_K}(N_n)+ r =j_{\widehat{U(\pi^n\mathcal{N})}_K}(N_n)+r\\
&=j_{S_n}(N_n)+r.
\end{align*} 
Here, the last equality follows from Proposition \ref{keypro} and Lemma \ref{keylemma2}. Finally, for $n$ sufficiently large, one has that:
\begin{align*}
d_{\wideparen{\mathcal{D}}(\bX,G)}(i_+N)&= 2d- j_{\wideparen{\mathcal{D}}(\bX,G)}(i_+N)\\
&=2d-j_{T_n}(i_+N\otimes_{\wideparen{\mathcal{D}}(\bY,G)}T_n)\\
&=2d-(r+j_{S_n}(N_n))\\
&=r +(2d-2r-j_{\wideparen{\mathcal{D}}(\bY,G)}(N))\\
&= d_{\wideparen{\mathcal{D}}(\bY,G)}(N)+\dim A-\dim A/I.
\end{align*}

\end{proof}
Now, let $i: \,\, \bY \rightarrow \bX$ be a closed embedding of smooth rigid varieties, $G$ be a $p$-adic Lie group which acts continuously on $\bX$ and which preserves $\bY$. Let $\mathcal{I}\subset \mathcal{O}_{\bX}$ be the ideal sheaf defining $\bY$. Recall from \cite{AWG} that the above affinoid version $i_+$ extends to a general equivariant pushforward functor 
$i_+: \mathcal{C}_{\bY/G}^r \rightarrow \mathcal{C}^r_{\bX/G}.$ Its construction makes crucial use of the fact \cite[Theorem $6.2$]{AWII} that there is an admissible covering $\mathcal{B}$ of $\bX$ of connected affinoid subdomains $\bU$ such that
\begin{itemize}
\item[$(i)$]there is a free $\mathcal{A}$-Lie lattice $\mathcal{L}=\partial_1\mathcal{A}\oplus...\oplus \partial_d\mathcal{A}$ for some affine formal model $\mathcal{A}\subset \mathcal{O}(\bU)$ satisfying $[\mathcal{L},\mathcal{L}]\subset \pi.\mathcal{L}$ and $\mathcal{L}.\mathcal{A}\subset \pi\mathcal{A}$,
\item[$(ii)$]either $\mathcal{I}(\bU)=\mathcal{I}(\bU)^2$, or $\mathcal{I}(\bU)$ admits a generating set $F=\lbrace f_1,...,f_r\rbrace$ with $\partial_i(f_j)=\delta_{ij}$ for every $i=1...,d$ and $j=1,...,r$.
\end{itemize} 
Let $\bU\in \mathcal{B}$. By definition, there is a free $\mathcal{A}$-Lie lattice $\mathcal{L}=\partial_1\mathcal{A}\oplus...\oplus \partial_d\mathcal{A}$ for some affine formal model $\mathcal{A}\subset \mathcal{O}(\bU)$ satisfying the conditions $(i)$ and $(ii)$ of the above theorem.  Following \cite[Lemma $4.4.2$]{AWG}, there exists a compact open subgroup $H$ of $G$ which stabilies $\bU$, $\mathcal{A}$ and $\mathcal{L}$.  The subgroup $H$ is then called  \textit{$\bU$-good}. Let $\mathcal{N}\in \mathcal{C}_{\bY/G}^r$. Then $i_+\mathcal{N}$ of $\mathcal{N}$ can be defined locally as follows
\[i_+\mathcal{N}(\bU):=\varprojlim_H M[\bU,H]\]
for any $\bU\in \mathcal{B}$, where $M[\bU,H]:=\mathcal{N}(\bU\cap \bY)\wideparen{\otimes}_{\wideparen{\mathcal{D}}(\bU\cap \bY,H)}\wideparen{\mathcal{D}}(\bU,H)/\mathcal{I}(\bU)\wideparen{\mathcal{D}}(\bU,H)$ and $H$ runs over the set of all $\bU$-good subgroups of $G$.

\begin{pro}\label{bernsteindim}
Let $i: \,\, \bY \rightarrow \bX$ be a closed embedding of smooth rigid varieties, $G$ be a $p$-adic Lie group which acts continuously on $\bX$ and which preserves $\bY$. Then for every $\mathcal{N}\in \mathcal{C}^r_{\bY/G}$
\[d_{\bX}(i_+\mathcal{N})=d_{\bY}(\mathcal{N}) +\dim \bX- \dim \bY.\]
\end{pro}
\begin{proof}
Given the above result, this is a direct consequence of Proposition \ref{dimest}. 
\end{proof}
We give two applications.

\begin{theorem} 
Let $\bX$ be a smooth rigid analytic variety with a continuous $G$-action. Suppose that 
$\bY$ is a Zariski closed subspace of $\bX$ which is stable under the $G$-action. If 
Bernstein's inequality holds in $\mathcal{C}_{\bX/G}$, then it holds in $\mathcal{C}_{\bY/G}$.
\end{theorem}
\begin{proof}
This is a direct consequence of Proposition $\ref{bernsteindim}$ and Proposition $\ref{dimcom}$.
\end{proof}

 \begin{theorem}\label{Kash} Let $i: \,\, \bY \rightarrow \bX$ be a closed embedding of smooth rigid varieties, $G$ be a $p$-adic Lie group which acts continuously on $\bX$ and which preserves $\bY$. 
 Then Kashiwara's equivalence restricts to an equivalence between $\mathcal{C}^{wh}_{\bY/G}$
 and the category of weakly holonomic equivariant $\wideparen{\mathcal{D}}_{\bX}$-modules supported on $\bY$.
 \end{theorem}
 \begin{proof} This is a direct consequence of the preceding proposition.
 \end{proof}
 \begin{coro}\label{push}
 Let $i: \,\, \bY \rightarrow \bX$ be a closed embedding of smooth rigid varieties, $G$ be a $p$-adic Lie group which acts continuously on $\bX$ and which preserves $\bY$. 
 Then $i_+\mathcal{O}_{\bY}$ is a weakly holonomic $G$-equivariant $\wideparen{\mathcal{D}}_{\bX}$-module.
  \end{coro} 
 
 \begin{proof} This follows from \ref{Kash} and \ref{structure_sheaf}. \end{proof}
 

\subsection{Weak holonomicity and geometric induction}
Let $\bX$ be a smooth rigid analytic space and $G$ be a $p$-adic Lie group acting continuously on $\bX$. Suppose that $P$ is a closed subgroup of $G$ such that $G/P$ is compact. Note that under this condition, the set of double cosets $\vert H\setminus G/P\vert$ is finite for every open subgroup $H\leq G$.\\
\\
We recall from \cite[$2.2$]{AWG} the geometric induction functor 
$$\text{ind}_P^G: \mathcal{C}_{\bX/P}\longrightarrow \mathcal{C}_{\bX/G}$$
which is locally defined as follows. Let $\mathcal{N} \in \mathcal{C}_{\bX/P}$. Let $\bU\in \bX_{w}(\mathcal{T})$ be an affinoid open subset, $H$ be a $\bU$-small subgroup of $G$ and $s \in G$. If $J \leq G$ is a subgroup, we write ${}^sJ=sJs^{-1}$, $J^s=s^{-1}Js$. Then we set
$$[s]\mathcal{N}(s^{-1}\bU):= \lbrace [s]m : m \in \mathcal{N}(s^{-1}\bU)\rbrace.$$
Note that $H$ is open in $G$, the subgroup $P\cap H^s$ is also open in $P$ and the pair $(s^{-1}\bU, P\cap H^s)$ is small. Hence $\mathcal{N}(s^{-1}\bU)\rbrace$ is a $\wideparen{\mathcal{D}}(s^{-1}\bU, P\cap H^s)$-module. So $[s]\mathcal{N}(s^{-1}\bU)$ can be equipped with a structure of $\wideparen{\mathcal{D}}(\bU, {}^sP\cap H)$-module via the isomorphism of $K$-algebras 
\[s^{-1}: \wideparen{\mathcal{D}}(\bU, {}^sP\cap H)\tilde{\longrightarrow} \wideparen{\mathcal{D}}(s^{-1}\bU, P\cap H^s).\]
This is a coadmissile $\wideparen{\mathcal{D}}(\bU, {}^sP\cap H)$-module and one forms the coadmissible $\wideparen{\mathcal{D}}(\bU,H)$-module:
\begin{center}
$M(\bU,H,s)= \wideparen{\mathcal{D}}(\bU,H) \wideparen{\otimes}_{\wideparen{\mathcal{D}}(\bU, {}^sP\cap H)}[s]\mathcal{N}(s^{-1}\bU)$.
\end{center}
The $\wideparen{\mathcal{D}}(\bU,H)$-module $M(\bU,H,s)$ only depends on the double coset $HsP$ which contains $s$ (\cite[Proposition $3.2.7$]{AWG}), which means that if $t \in HsP$ such that $s=h^{-1}th'$ with $h\in H, h' \in P$, then $M(\bU,H,s) \cong M(\bU,H,t)$ as $\wideparen{\mathcal{D}}(\bU,H)$-modules.This allows to define for each double coset $Z \in H\setminus G/P$:
\[M(\bU,H,Z):= \lim_{s \in Z}M(\bU,H,s).\]
Note that $M(\bU,H,Z) \cong M(\bU,H,s)$ in $\mathcal{C}_{\wideparen{\mathcal{D}}(\bU,H)}$ for all $s\in Z$. Since $\mid H\setminus G/P\mid$ is finite, we obtain that  
$$M(\bU,H):=\bigoplus_{Z \in H \setminus G/P}M(\bU,H,Z)$$ 
is also a coadmissible $\wideparen{\mathcal{D}}(\bU,H)$-module. If $J \leq H$ are $\bU$-small subgroups of $G$ then there is an isomorphism of $\wideparen{\mathcal{D}}(\bU,J)$-modules 
$M(\bU,J)\tilde{\rightarrow} M(\bU,H).$ In this situation as $\wideparen{\mathcal{D}}(\bU,H)$-modules
\begin{align*}
\text{ind}_P^G(\mathcal{N})(\bU)&=\varprojlim_H \bigoplus_{Z \in H\setminus G/P}\lim_{s \in Z} \wideparen{\mathcal{D}}(\bU,H) \wideparen{\otimes}_{\wideparen{\mathcal{D}}(\bU, {}^sP\cap H)}[s]\mathcal{N}(s^{-1}\bU)\\
&=\varprojlim_H M(\bU,H),
\end{align*}
where the inverse limit is taken over the set of $\bU$-small subgroups $H$ of $G$.  
\begin{pro} \label{indpro} Geometric induction $\text{ind}_P^G$ preserves weak holonomicty, i.e restricts to a functor
$$\mathcal{C}^{wh}_{\bX/P}\longrightarrow \mathcal{C}^{wh}_{\bX/G}.$$

\end{pro}
\begin{proof}
Since the sum $M(\bU,H):=\bigoplus_{Z \in H \setminus G/P}M(\bU,H,Z)$ is finite, one has 
\begin{center}
$Ext_{\wideparen{\mathcal{D}}(\bU, H)}^i(M(\bU,H), \wideparen{\mathcal{D}}(\bU,H))\cong \bigoplus_{Z \in H\setminus G/P}Ext_{\wideparen{\mathcal{D}}(\bU, H)}^i(M(\bU,H,Z), \wideparen{\mathcal{D}}(\bU,H)).$
\end{center}
In particular, $Ext_{\wideparen{\mathcal{D}}(\bU, H)}^i(M(\bU,H), \wideparen{\mathcal{D}}(\bU,H))=0$ if and only if 
\begin{center}
$ Ext_{\wideparen{\mathcal{D}}(\bU, H)}^i(M(\bU,H,Z), \wideparen{\mathcal{D}}(\bU,H))=0$ for all $Z \in H \setminus G/P$. 
\end{center}
This shows that
\begin{equation}\label{ind1}
j(M(\bU,H))=\inf \lbrace j(M(\bU,H,Z)): Z \in H \setminus G/P\rbrace.
\end{equation}
Now let $Z \in H \setminus G/P$. Since $M(\bU,H, Z)\cong M(\bU,H,s)$ in $\mathcal{C}_{\wideparen{\mathcal{D}}(\bU,H)}$ for any choice of $s \in Z$ and the map $\wideparen{\mathcal{D}}(\bU,{}^sP\cap H) \longrightarrow \wideparen{\mathcal{D}}(\bU,H)$ is faithfully $c-$flat \cite[Lemma $3.5.3$]{AWG} (note that ${}^sP\cap H$ is closed in $H$), we obtain
\begin{center}
$Ext^i_{\wideparen{\mathcal{D}}(\bU,H)}(\wideparen{\mathcal{D}}(\bU,H) \wideparen{\otimes}_{\wideparen{\mathcal{D}}(\bU, {}^sP\cap H)}[s]\mathcal{N}(s^{-1}\bU), \wideparen{\mathcal{D}}(\bU,H))$\\
$\cong Ext^i_{\wideparen{\mathcal{D}}(\bU,{}^sP\cap H)}([s]\mathcal{N}(s^{-1}\bU),\wideparen{\mathcal{D}}(\bU,{}^sP\cap H))\wideparen{\otimes}_{\wideparen{\mathcal{D}}(\bU,{}^sP\cap H)}\wideparen{\mathcal{D}}(\bU,H).$
\end{center}
This implies:  
\begin{equation}\label{ind2}
j_{\wideparen{\mathcal{D}}(\bU,H)}(M(\bU,H,Z))=j_{\wideparen{\mathcal{D}}(\bU,H)}(M(\bU,H,s))=j_{\wideparen{\mathcal{D}}(\bU, {}^sP\cap H)}([s]\mathcal{N}(s^{-1}\bU)).
\end{equation}
Next, the isomorphism of $K$-algebras $\wideparen{\mathcal{D}}(\bU,{}^sP\cap H) \tilde{\longrightarrow} \wideparen{\mathcal{D}}(s^{-1}\bU, P\cap H^s)$ implies that 
\[ Ext^i_{\wideparen{\mathcal{D}}(\bU,{}^sP\cap H)}([s]\mathcal{N}(s^{-1}\bU),\wideparen{\mathcal{D}}(\bU,{}^sP\cap H))\cong  Ext^i_{\wideparen{\mathcal{D}}(s^{-1}\bU,{}P\cap H^s)}(\mathcal{N}(s^{-1}\bU),\wideparen{\mathcal{D}}(s^{-1}\bU,{}P\cap H^s)).\]
By consequence, 
\begin{equation}\label{ind3}
j_{\wideparen{\mathcal{D}}(\bU,{}^sP\cap H)}([s]\mathcal{N}(s^{-1}\bU))=j_{\wideparen{\mathcal{D}}(s^{-1}\bU,{}P\cap H^s)}(\mathcal{N}(s^{-1}\bU)).
\end{equation}
Now since $\mathcal{N}$ is weakly holonomic, (\ref{ind1}), (\ref{ind2}) and (\ref{ind3}) imply that $d(\text{ind}_P^G(\mathcal{N}))\leq \dim{\bX}$. 
\end{proof}

 \begin{coro}
 Let $i: \,\, \bY \rightarrow \bX$ be a closed embedding of smooth rigid varieties, $G$ be a $p$-adic Lie group which acts continuously on $\bX$ and suppose that the stabilizer $G_{\bY}$ of $\bY$ in $G$ is co-compact. Then
$ \mathrm{ind}_{G_{\bY}}^G i_+\mathcal{O}_{\bY}$ is a weakly holonomic $G$-equivariant $\wideparen{\mathcal{D}}_{\bX}$-module.
  \end{coro} 
 
 \begin{proof} 
This follows from the preceding result and \ref{push}.
\end{proof}

 
\bibliographystyle{plain}
\bibliography{Biblio}
\vskip10pt

\noindent {\small Tobias Schmidt, Bergische Universität Wuppertal, Gaußstraße 20, 42119 Wuppertal, Germany \newline {\it E-mail address: \url{toschmidt@uni-wuppertal.de}} }

\vskip10pt 

\noindent {\small Thi Minh Phuong Vu, FPT University, Hoa Lac Hi-tech Park, Km29, Thang Long Boulevard, Thach Hoa, Thach That, Ha Noi, Vietnam
 \newline {\it E-mail address: \url{phuongvtm11@fe.edu.vn}} }

\end{document}